\documentclass{amsart}
\usepackage{amssymb,amsfonts, color,epsf}
\usepackage{mathabx}
\usepackage [cmtip,arrow]{xy}
\xyoption{all}
\usepackage {pb-diagram,pb-xy}
\usepackage{enumerate}
\usepackage{accents}
\usepackage[noautoscale]{youngtab}
\usepackage{subfigure}

\ifx\pdfpageheight\undefined
   \usepackage[dvips,colorlinks=true,linkcolor=blue,citecolor=red,%
      urlcolor=green]{hyperref}
   \usepackage[dvips]{graphicx}
   \makeatletter
   \edef\Gin@extensions{\Gin@extensions,.mps}
   \makeatother
\else
 \usepackage[pdftex]{graphicx}
   \usepackage[bookmarksopen=false,pdftex=true,breaklinks=true,%
      backref=page,pagebackref=true,plainpages=false,%
      hyperindex=true,pdfstartview=FitH,colorlinks=true,%
      pdfpagelabels=true,colorlinks=true,linkcolor=blue,%
      citecolor=red,urlcolor=green,hypertexnames=false%
      ]%
   {hyperref}
\fi
\usepackage{bm}

\usepackage{tikz-cd}
\makeatletter
\tikzset{
  column sep/.code=\def\pgfmatrixcolumnsep{\pgf@matrix@xscale*(#1)},
  row sep/.code   =\def\pgfmatrixrowsep{\pgf@matrix@yscale*(#1)},
  matrix xscale/.code=%
    \pgfmathsetmacro\pgf@matrix@xscale{\pgf@matrix@xscale*(#1)},
  matrix yscale/.code=%
    \pgfmathsetmacro\pgf@matrix@yscale{\pgf@matrix@yscale*(#1)},
  matrix scale/.style={/tikz/matrix xscale={#1},/tikz/matrix yscale={#1}}}
\def\pgf@matrix@xscale{1}
\def\pgf@matrix@yscale{1}
\makeatother

\usepackage{amsmath,amssymb,amsfonts}
\newtheorem{theorem}{Theorem}
\newtheorem{lemma}{Lemma}[section]
\newtheorem{corollary}{Corollary}

\newtheorem{claim}{Claim}[section]
\newtheorem*{claim*}{Claim}

\newtheorem{hypothesis}{Hypothesis}[section]

\newtheorem{property}{Property}[section]

\newtheorem*{theorem*}{Theorem}
\newtheorem*{corollary*}{Corollary}
\theoremstyle{definition}
\newtheorem{definition}{Definition}[section]
\newtheorem{example}{Example}[section]

\newtheorem{notation}{Notation}[section]

\newtheorem{manualtheoreminner}{Theorem}
\newenvironment{manualtheorem}[1]{%
  \renewcommand\themanualtheoreminner{#1}%
  \manualtheoreminner \itshape
}{\endmanualtheoreminner}

\newtheorem{manuallemmainner}{Lemma}
\newenvironment{manuallemma}[1]{%
  \renewcommand\themanuallemmainner{#1}%
  \manuallemmainner \itshape
}{\endmanuallemmainner}

\newtheorem{manualclaiminner}{Claim}
\newenvironment{manualclaim}[1]{%
  \renewcommand\themanualclaiminner{#1}%
  \manualclaiminner \itshape
}{\endmanualclaiminner}

\newtheorem{manualpropertyinner}{Property}
\newenvironment{manualproperty}[1]{%
  \renewcommand\themanualpropertyinner{#1}%
  \manualpropertyinner \itshape
}{\endmanualpropertyinner}

\usepackage{algorithm}
\usepackage{algpseudocode}

\usepackage{tikz}

\algnewcommand\algorithmicinput{\textbf{Input:}}
\algnewcommand\INPUT{\item[\algorithmicinput]}
\algnewcommand\algorithmicoutput{\textbf{Output:}}
\algnewcommand\OUTPUT{\item[\algorithmicoutput]}

\algnewcommand\algorithmicproc{\textbf{Procedure:}}
\algnewcommand\PROCEDURE{\item[\algorithmicproc]}
\algnewcommand\algorithmiccomplexity{\textbf{Complexity:}}
\algnewcommand\COMPLEXITY{\item[\algorithmiccomplexity]}

\newlength{\continueindent}
\usepackage{etoolbox}
\makeatletter
\newcommand*{\ALG@customparshape}{\parshape 2 \leftmargin \linewidth \dimexpr\ALG@tlm+\continueindent\relax \dimexpr\linewidth+\leftmargin-\ALG@tlm-\continueindent\relax}
\apptocmd{\ALG@beginblock}{\ALG@customparshape}{}{\errmessage{failed to patch}}
\makeatother

\theoremstyle{remark}
\newtheorem{remark}{Remark}[section]

\theoremstyle{observation}

\definecolor{DarkBlue}{rgb}{0,0.1,0.55}

\numberwithin{equation}{section}

\newcommand {\hide}[1]{}

\newcommand {\junk}[1]{}

\newcommand {\R} {\mathrm{R}}

\newcommand {\D}     {\mbox{\rm D}}

\newcommand {\Sphere}{\mbox{${\bf S}$}}     

\newcommand {\Z}  {\mathbb{Z}}
 
\newcommand {\Q}         {\mathbb{Q}}

\newcommand {\RR} {{\mathcal R}}

\newcommand {\la}   {{\langle}}
\newcommand {\ra}   {{\rangle}}
\newcommand {\eps} {{\varepsilon}}
\newcommand {\E} {{\rm ext}}

\newcommand{\card}{\mathrm{card}}
\newcommand{\rank}{\mathrm{rank}}

\def\addots{\mathinner{\mkern1mu
\raise1pt\vbox{\kern7pt\hbox{.}}
\mkern2mu\raise4pt\hbox{.}\mkern2mu
\raise7pt\hbox{.}\mkern1mu}}

\newcommand{\HH}  {\mbox{\rm H}}

\newcommand{\Hom}{\mathrm{Hom}}

\newcommand\colim{\mathrm{colim}}

\newcommand{\level}{\mathrm{level}}

\newcommand{\Top}{\mathbf{Top}}

\newcommand{\hocolim}{\mathrm{hocolim}}

\newcommand{\Nerve}{\mathcal{N}}

\newcommand{\Simp}{\mathbf{Simp}}

\newcommand{\Comp}{\mathbf{size}}

\newcommand{\Def}{\mathrm{Def}}

\newcommand{\length}{\mathrm{length}}

\begin{document}
\title[Efficient simplicial replacement of semi-algebraic sets]
{
Efficient simplicial replacement of semi-algebraic sets
}
\author{Saugata Basu}
\address{Department of Mathematics,
Purdue University, West Lafayette, IN 47906, U.S.A.}
\email{sbasu@math.purdue.edu}

\author{Negin Karisani}
\address{Department of Computer Science, 
Purdue University, West Lafayette, IN 47906, U.S.A.}
\email{nkarisan@cs.purdue.edu}

\date{\textbf{\today}}
\thanks{
Basu was  partially supported by NSF grants
CCF-1618918, DMS-1620271 and CCF-1910441.
}
\begin{abstract}
Designing an algorithm with a singly exponential complexity for computing semi-algebraic triangulations
of a given semi-algebraic set has been a holy grail in algorithmic semi-algebraic geometry.
More precisely, given a description of a semi-algebraic set $S \subset \mathbb{R}^k$ 
by a first order quantifier-free formula in the language of the reals, 
the goal is to output a simplicial complex $\Delta$, whose geometric
realization, $|\Delta|$,  is semi-algebraically homeomorphic to $S$. 
In this paper we consider a weaker version of this question.
We prove that  for any $\ell \geq  0$, there exists an algorithm which takes as input a description of a semi-algebraic subset  
$S \subset \mathbb{R}^k$ given by
a quantifier-free first order formula $\phi$ in the language of the reals, and produces as output a simplicial complex $\Delta$,
whose geometric realization, $|\Delta|$ is $\ell$-equivalent to $S$. 
The complexity of our algorithm is bounded by
$(sd)^{k^{O(\ell)}}$, where $s$ is the number of polynomials appearing in the formula $\phi$, and $d$ a bound
on their degrees. For fixed $\ell$, this bound is \emph{singly exponential} in $k$. 
In particular, since $\ell$-equivalence 
implies that the \emph{homotopy groups} up to dimension $\ell$  of $|\Delta|$ are isomorphic to those of $S$,
we obtain a reduction (having singly exponential complexity) of the problem of computing the first $\ell$ homotopy groups of $S$ to the 
combinatorial problem of computing the first $\ell$ homotopy groups of a finite simplicial complex of size 
bounded  by $(sd)^{k^{O(\ell)}}$.

\end{abstract}
\maketitle

\tableofcontents

\section{Introduction}
\subsection{Background}
Let $\R$ be a real closed field and $\D$ an ordered domain contained in $\R$.

The problem of effective computation of topological properties of semi-algebraic subsets of $\R^k$ has a long
history.  Semi-algebraic subsets of $\R^k$ are subsets defined by first-order formulas in the language of ordered
fields (with parameters in $\R$). Since the first-order theory of real closed fields admits 
quantifier-elimination, we can assume that each semi-algebraic subset $S \subset \R^k$ is defined by some 
\emph{quantifier-free} formula $\phi$. A quantifier-free formula $\phi(X_1,\ldots,X_k)$ in the 
language of ordered fields with parameters in $\D$, is a formula with atoms of the form $P = 0, P > 0, P < 0$,
$P \in \D[X_1,\ldots,X_k]$.

Semi-algebraic subsets of $\R^k$ have tame topology. In particular, closed and bounded semi-algebraic subsets
of $\R^k$ are semi-algebraically triangulable (see for example \cite[Chapter 5]{BPRbook2}). 
This means that there exists a finite simplicial complex $K$, whose 
\emph{geometric realization}, $|K|$,  considered as a subset of $\R^N$ for some $N >0$, is semi-algebraically homeomorphic to $S$. 
The semi-algebraic homeomorphism
$|K| \rightarrow S$ is called a \emph{semi-algebraic triangulation of $S$}. All topological
properties of $S$ are then encoded in the finite data of the simplicial complex $K$.

For instance, taking $\R = \mathbb{R}$,
the (singular) homology groups, $\HH_*(S)$, of $S$ are isomorphic to the simplicial homology groups
of the simplicial chain complex $\mathrm{C}_\bullet(K)$ of the simplicial complex $K$, and the latter is 
a complex of free $\Z$-modules having finite ranks (here and elsewhere in the paper, unless stated otherwise,  all homology and cohomology groups
are with coefficients in $\Z$). 

The problem of designing an efficient algorithm for obtaining semi-algebraic
triangulations has attracted a lot of attention over the years.
One reason behind this is that once we have such a triangulation, we 
can then compute discrete topological invariants, such as the 
ranks of the homology groups (i.e. the Betti numbers) of the given semi-algebraic set
with just some added linear algebra over $\Z$. 
     
There exists a classical algorithm
which takes as input a quantifier-free formula defining a semi-algebraic set $S$, and 
produces as output a semi-algebraic triangulation of $S$  (see for instance \cite[Chapter 5]{BPRbook2}).  
However,  this algorithm is based on the technique of \emph{cylindrical algebraic decomposition}, and hence 
the complexity of this algorithm is prohibitively expensive,  being doubly exponential in $k$. 
More precisely, given a description by a quantifier-free formula involving 
$s$ polynomials of degree at most $d$, of a closed and bounded semi-algebraic subset of $S \subset \R^k$, there
exists an algorithm computing a semi-algebraic triangulation of $h: |K| \rightarrow S$, whose complexity
is bounded by $(s d)^{2^{O(k)}}$. Moreover, the size of the simplicial complex $K$ (measured by the number of 
simplices) is also bounded by $(s d)^{2^{O(k)}}$.

\subsubsection{Doubly exponential vs singly exponential} 
One can ask whether the doubly exponential behavior for the semi-algebraic triangulation problem is intrinsic
to the problem. One reason to think that it is not so comes from the fact that the ranks of the homology groups of 
$S$ (following the same notation as in the previous paragraph), 
and so in particular those of the simplicial complex $K$, is bounded by
$(O(sd))^k$ (see for instance \cite[Chapter 7]{BPRbook2}), 
which is singly exponential in $k$. So it is natural to ask if this singly exponential upper bound
on $\rank(\HH_*(S))$ is ``witnessed''  by an efficient semi-algebraic triangulation of small (i.e. singly exponential) size.
This is not known.

In fact, designing an algorithm with a \emph{singly exponential} complexity 
for computing a semi-algebraic triangulation of a given semi-algebraic set 
has remained a holy grail in the field of algorithmic real algebraic geometry and little progress has been made over the last 
thirty years on this problem (at least for general semi-algebraic sets). 
We note here that designing algorithms with singly exponential complexity has being a \emph{leit motif}  in the research in algorithmic semi-algebraic geometry over the past decades -- starting from the so called ``critical-point method'' which resulted in algorithms for testing emptiness, connectivity, computing the Euler-Poincar\'e characteristic, as well as for the first few Betti numbers of semi-algebraic sets (see \cite{Basu-survey} for a history of these developments and contributions of many authors). More recently, such algorithms has also been 
developed in other (more numerical) models of computations \cite{ BCL2019,BCT2020.1, BCT2020.2}
(we discuss the connection of these works with the results presented in this paper in Section~\ref{subsec:prior}).

\subsubsection{Triangulation vs simplicial replacement}
While the problem of designing an algorithm with singly exponential complexity for the problem of semi-algebraic triangulation is completely open,
there has been some progress in designing efficient algorithms for certain related problems. As mentioned above
a semi-algebraic triangulation of a closed and bounded semi-algebraic set $S$ produces a finite simplicial complex, which
encodes all topological properties (i.e. which are homeomorphism invariants) of $S$. It is well known that homeomorphism
invariants are notoriously difficult to compute (for instance, it is an undecidable problem to determine whether two 
simplicial complexes are homeomorphic \cite{Markov2}). What is much more computable are the homology groups
of semi-algebraic sets. Homology groups are in fact homotopy (rather than homeomorphism) invariants. Homotopy equivalence is a much weaker equivalence relation compared to homeomorphism.  
In the absence of  a singly exponential complexity triangulation of semi-algebraic sets, 
it is reasonable to ask for an algorithm
which given a semi-algebraic set $S \subset \R^k$ described by a quantifier-free formula involving
$s$ polynomials of degrees bounded by $d$, computes a simplicial complex $K$, such that
its geometric realization $|K|$ is \emph{homotopy equivalent} to $S$ having complexity bounded by
$(sd)^{k^{O(1)}}$.
We will call such a simplicial complex a \emph{simplicial replacement} of the semi-algebraic set $S$.

The main results of this paper can be summarized as follows. The precise statements appear in the 
next section after the necessary definitions of various objects some of which are a bit technical.

\subsection{Summary of results}
In the statements below  $\ell \in \Z_{\geq 0}$ is a fixed constant.

\begin{theorem*}[cf. Theorems~\ref{thm:alg} and \ref{thm:alg'} below]
Given any closed semi-algebraic subset of $S \subset \R^k$, there exists a simplicial complex $K$
homologically $\ell$-equivalent to $S$ whose size is bounded singly exponentially in $k$ (as a function of the number and degrees of polynomials appearing in the description of $S$).
If  $\R = \mathbb{R}$, then $K$ is $\ell$-equivalent to $S$.
Moreover, there exists an algorithm  (Algorithm~\ref{alg:simplicial-replacement}) which computes the complex $K$ given $S$, and whose complexity is bounded singly exponentially in $k$.
\end{theorem*}

The problem of designing  efficient (symbolic and exact) algorithms for computing the Betti numbers 
of semi-algebraic sets have been considered before, and algorithms with singly exponential complexity was given for computing the first (resp. the first $\ell$ for any fixed $\ell$) Betti numbers
in \cite{BPRbettione} (resp. \cite{Bas05-first}). The algorithm given in the \cite{BPRbettione} (resp. \cite{Bas05-first}) computes a complex of vector spaces having
isomorphic homology (with coefficients in $\Q$) up to dimension one  (resp. $\ell$) 
as that of the given semi-algebraic set. However,
information with regards to homotopy is lost. The algorithm implicit in the theorem stated above
produces a simplicial complex having the same homotopy type up to dimension $\ell$ as the given
semi-algebraic set.  Thus the above theorem can be viewed as a homotopy-theoretic generalization
of the results in \cite{BPRbettione} and \cite{Bas05-first}.

The above theorem can be used for the problem of computing the homotopy groups 
of semi-algebraic sets. Homotopy groups are  much finer invariants than homology groups but are
also more difficult to compute. In fact the problem of deciding whether the first homotopy group 
(i.e. the fundamental group) of a semi-algebraic set defined over $\mathbb{R}$ 
is trivial or not is an undecidable problem.
Nevertheless, using the above theorem  
we have the following corollary which gives an algorithmic reduction 
having singly exponential complexity of the problem of computing the first $\ell$ homotopy groups
of a given closed semi-algebraic set to a purely combinatorial problem.
 
\begin{corollary*}[cf. Corollaries~\ref{cor:alg':1} and \ref{cor:alg':2} below]
Let $\R = \mathbb{R}$,
There exists a reduction having singly exponential complexity,  of the problem of computing the
first $\ell$ homotopy groups of any given closed semi-algebraic subset $S \subset \R^k$, to the 
problem of computing the first $\ell$  homotopy groups of a finite simplicial complex.  
This implies that there exists an algorithm with singly exponential complexity
which given as input a  closed semi-algebraic set $S \subset \mathbb{R}^k$ 
guaranteed to be simply connected, outputs
the description of the first $\ell$ homotopy groups of $S$ (in terms of generators and relations).
\end{corollary*}

The algorithmic results mentioned above are consequences of a topological construction which can
be interpreted as a generalization of the classical ``nerve lemma''  in topology. We state it here informally.

Assume that there exists a ``black-box'' that given as input any closed semi-algebraic set $S \subset \R^k$,
produces as output a cover of $S$ by closed semi-algebraic subsets of $S$ which are 
homologically
$\ell$-connected.

\begin{theorem*}[cf. Theorem~\ref{thm:main} below]
Given a black-box as above,
there exists for every closed semi-algebraic set $S$ a poset $\mathbf{P}(S)$ 
(see Definition~\ref{def:poset} below)
which depends on the given black-box, of controlled complexity (both in terms of the description of $S$ and the 
complexity of the black-box), such that the geometric realization of the order-complex of $\mathbf{P}(S)$ is 
homologically
$\ell$-equivalent to $S$.  
\end{theorem*}
 
\begin{remark}
\label{rem:GV}
In the results stated above we make the assumption that the input semi-algebraic sets are closed. 
Gabrielov and Vorobjov \cite{GaV} gave a construction for replacing an arbitrary semi-algebraic subset of $\mathbb{R}^k$ by a closed and bounded one having homology and homotopy groups  isomorphic to the given semi-algebraic set. 
Even though Gabrielov and Vorobjov proved their result over $\mathbb{R}$, the construction was extended to arbitrary real closed fields (with the approximating set defined over a real closed extension of the ground field).
It is proved in \cite{BPRbook2} (Theorem 7.45), that the approximating set is in fact semi-algebraically homotopy equivalent to the (extension of the)  given set. Using this latter result one could remove the assumption of
being closed and bounded in Theorems~\ref{thm:alg} and \ref{thm:alg'}.
We choose not to do this in this paper in order not to add
yet another layer of technical complication involving a new set of infinitesimals.
\end{remark}

The rest of the paper is organized as follows.
In Section~\ref{sec:precise} we give precise statements of the main results summarized 
above after introducing the necessary definitions regarding the different notions of 
topological equivalence that we use in the paper and also the definition of complexity
of algorithms that we use. In Section~\ref{sec:simplicial-replacement} we define the key
mathematical object (namely, a poset that we associate to any closed covering of a semi-algebraic set)
and prove its main properties (Theorems~\ref{thm:main} and \ref{thm:main'}). 
In Section~\ref{sec:algo} we describe algorithms for computing efficient simplicial replacements
of semi-algebraic sets thereby proving Theorems~\ref{thm:alg} and \ref{thm:alg'}.   Finally, in Section~\ref{sec:conclusion} we state some
open questions and directions for future work in this area.

\section{Precise statements of the main results}
\label{sec:precise}
In this section we will describe in full detail  
the main results summarized in the  previous section. We first 
introduce certain preliminary definitions and notation.
 
\subsection{Definitions of topological equivalence and complexity}

We begin with the precise definitions of the two kinds of 
topological equivalence that we are going to use in this paper.

\subsubsection{Topological  equivalences}
\begin{definition}[$\ell$-equivalences]
\label{def:equivalence-spaces}
We say that a map $f:X \rightarrow Y$ between two topological spaces is an $\ell$-equivalence,
if the induced homomorphisms between the homotopy groups $f_*: \pi_i(X) \rightarrow \pi_i(Y)$ are isomorphisms for $0 \leq i \leq \ell$ \cite[page 68]{Novikov}.
\end{definition}

\begin{remark}
\label{rem:equivalence}
Note that our definition of $\ell$-equivalence deviates a little from the standard one which requires that
homomorphisms between the homotopy groups $f_*: \pi_i(X) \rightarrow \pi_i(Y)$ are isomorphisms for $0 \leq i \leq \ell-1$, and only an epimorphism for $i=\ell$. 
An $\ell$-equivalence in our sense is an $\ell$-equivalence in the traditional sense.
\end{remark}

The relation of $\ell$-equivalence as defined above is not an equivalence relation since it is not 
symmetric. In order to make it symmetric one needs to ``formally invert'' $\ell$-equivalences.

\begin{definition}[$\ell$-equivalent and homologically $\ell$-equivalent]
\label{def:ell-equivalent}
We will say that  \emph{$X$ is $\ell$-equivalent to $Y$}  (denoted $X \sim_\ell Y$), if and only if there exists 
spaces, $X=X_0,X_1,\ldots,X_n=Y$ and $\ell$-equivalences  $f_1,\ldots,f_{n}$ as shown below:
\[
\xymatrix{
&X_1 \ar[ld]_{f_1}\ar[rd]^{f_2} &&X_3\ar[ld]_{f_3} \ar[rd]^{f_4}& \cdots&\cdots&X_{n-1}\ar[ld]_{f_{n-1}}\ar[rd]^{f_{n}} & \\
X_0 &&X_2  && \cdots&\cdots &&  X_n&
}.
\]
It is clear that $\sim_\ell$ is an equivalence relation.

By replacing the homotopy groups, $\pi_i(\cdot)$ with homology groups $\HH_i(\cdot)$
(resp. cohomology groups $\HH^i(\cdot)$ with arrows reversed) in Definitions~\ref{def:equivalence-spaces} and \ref{def:ell-equivalent},
we get the notion of two topological spaces $X,Y$  being \emph{homologically $\ell$-equivalent} (denoted 
$X \overset{h}{\sim}_\ell Y$)
(resp. \emph{cohomologically $\ell$-equivalent} (denoted 
$X \overset{ch}{\sim}_\ell Y$)). 

This is a strictly weaker
equivalence relation, since there are spaces $X$ for which $\HH_1(X) = 0$, but $\pi_1(X) \neq 0$.

We extend the above definitions to $\ell = -1$ by using the convention that 
$X \sim_{-1} Y$ (resp. $X \overset{h}{\sim}_{-1} Y$, $X \overset{ch}{\sim}_{-1} Y$), if and only if $X,Y$ are both non-empty or both empty.
\end{definition}

\begin{definition}[$\ell$-connected and homologically $\ell$-connected]
\label{def:ell-connected}
We say that a topological space  $X$ is \emph{$\ell$-connected}, for $\ell \geq 0$,  if $X$ is connected and 
$\pi_i(X) = 0$ for $0 < i \leq \ell$. We will say that $X$ is $(-1$)-connected if $X$ is non-empty.
We say that $X$ is \emph{homologically $\ell$-connected} if $X$ is connected and 
$\HH_i(X) = 0$ for $0 < i \leq \ell$.
\end{definition}

\begin{definition}[Diagrams of topological spaces]
\label{def:diagram-of-spaces}
A diagram of topological spaces is a functor, $X:J \rightarrow \Top$, from a small category $J$ to $\Top$. 
\end{definition}

We extend Definition~\ref{def:equivalence-spaces} to diagrams of topological spaces.
We denote by $\Top$ the category of topological spaces.

\begin{definition}[$\ell$-equivalence between diagrams of topological spaces]
\label{equivalence-diagrams}
Let $J$ be a small category, and $X,Y: J \rightarrow \Top$ be two functors. We say a natural transformation $f:X \rightarrow Y$ is an $\ell$ equivalence, if
the induced maps, 
\[
f(j)_*: \pi_i(X(j)) \rightarrow \pi_i(Y(j))
\]
are isomorphisms for all $j \in J$ and $0 \leq i \leq \ell$.

We will say that  \emph{a diagram $X:J \rightarrow \Top$ is $\ell$-equivalent to the diagram $Y:J \rightarrow \Top$}  (denoted as before by $X \sim_\ell Y$), if and only if there exists diagrams
 $X=X_0,X_1,\ldots,X_n=Y:J \rightarrow \Top$ and $\ell$-equivalences  $f_1,\ldots,f_{n}$ as shown below:
\[
\xymatrix{
&X_1 \ar[ld]_{f_1}\ar[rd]^{f_2} &&X_3\ar[ld]_{f_3} \ar[rd]^{f_4}& \cdots&\cdots&X_{n-1}\ar[ld]_{f_{n-1}}\ar[rd]^{f_{n}} & \\
X_0 &&X_2  && \cdots&\cdots &&  X_n&
}.
\]
It is clear that $\sim_\ell$ is an equivalence relation.

In the above definition, by replacing the homotopy groups with homology 
(resp. cohomology) groups we obtain the 
notion of homological (resp. cohomological) 
$\ell$-equivalence between diagrams, which we will denote as before
by $\overset{h}{\sim}_\ell$ (resp. $\overset{ch}{\sim}_\ell$). 
\end{definition}

One particular diagram will be important in what follows.

\begin{notation} [Diagram of various unions of a finite number of subspaces]
\label{not:diagram-Delta}
Let $J$ be a finite set, $A$ a topological space, 
and $\mathcal{A} = (A_j)_{j \in J}$ a tuple of subspaces of $A$  indexed by $J$.

For any subset 
$J' \subset J$,
\footnote{In this paper $A \subset B$ will mean $A \cap B = A$ allowing the possibility that $A = B$.
Also, when we denote $\alpha \prec \beta$ in a poset we allow the possibility $\alpha = \beta$, reserving $\alpha \precneq \beta$ to denote $\alpha \prec \beta, \alpha \neq \beta$.
} 
we denote 
\begin{eqnarray*}
\mathcal{A}^{J'} &=& \bigcup_{j' \in J'} A_{j'}, \\
\mathcal{A}_{J'} &=& \bigcap_{j' \in J'} A_{j'}, \\
\end{eqnarray*}

We consider $2^J$ as a category whose objects are elements of $2^J$, and whose only morphisms 
are given by: 
\begin{eqnarray*}
2^J(J',J'') &=& \emptyset  \mbox{ if  } J' \not\subset J'', \\
2^J(J',J'') &=& \{\iota_{J',J''}\} \mbox{  if } J' \subset J''.
\end{eqnarray*} 
We denote by $\Simp^J(\mathcal{A}):2^J \rightarrow \Top$ the functor (or the diagram) defined by
\[
\Simp^J(\mathcal{A})(J') = \mathcal{A}^{J'}, J' \in 2^J,
\]
and
$\Simp^J(\mathcal{A})(\iota_{J',J''})$ is the inclusion map $\mathcal{A}^{J'} \hookrightarrow \mathcal{A}^{J''}$.
\end{notation}

\subsubsection{Definition of complexity of algorithms}
We will use the following notion of ``complexity of an algorithm''  in this paper. We follow the same definition as used in the book \cite{BPRbook2}. 
  
\begin{definition}[Complexity of algorithms]
\label{def:complexity}
In our algorithms we will take as input quantifier-free first order formulas whose terms
are  polynomials with coefficients belonging to an ordered domain $\D$ contained in a real closed field $\R$.
By \emph{complexity of an algorithm}  we will mean the number of arithmetic operations and comparisons in the domain $\D$.
If $\D = \mathbb{R}$, then
the complexity of our algorithm will agree with the  Blum-Shub-Smale notion of real number complexity \cite{BSS}.
In case, $\D = \Z$, then we are able to deduce the bit-complexity of our algorithms in terms of the bit-sizes of the coefficients
of the input polynomials, and this will agree with the classical (Turing) notion of complexity.
\end{definition}

\begin{remark}[Separation of complexity into algebraic and combinatorial parts
\footnote{Note that this notion of separation of complexity into algebraic and combinatorial parts is distinct from that used in \cite{BPRbook2}, where ``combinatorial part'' refers to the part depending on the number of polynomials, and the``algebraic part'' refers to the dependence on the degrees of the polynomials. }
]
\label{rem:complexity}
In the definition of complexity given above we are counting only arithmetic operations involving elements of
the ring generated by the coefficients of the input formulas. Many algorithms in semi-algebraic geometry
have the following feature. After a certain number of operations involving elements of the coefficient ring $\D$, the 
problem is reduced to solving a combinatorial or a linear algebra problem defined over 
$\mathbb{Z}$.

A typical example is an algorithm for computing the Betti numbers of 
a semi-algebraic set via computing a semi-algebraic triangulation.
Once a
simplicial complex whose geometric realization is semi-algebraically homeomorphic to the given semi-algebraic set has been computed, the problem of computing the Betti numbers of the given semi-algebraic set  is reduced to linear algebra over $\mathbb{Z}$. 
Usually,
this separation of the cost of an algorithm into a part that involves arithmetic operations over $\D$, and a
part that is independent of $\D$, is not very important 
since often the complexity of the second part is subsumed by that of the first part. However, in this paper the fact that we are only counting arithmetic operations in $\D$ is more significant. In one application
that we discuss, namely that of computing the homotopy groups of a given semi-algebraic set  (see Corollary~ \ref{cor:alg':1}), 
we give a reduction (having single exponential complexity) to a problem whose definition is independent
of $\D$, namely computing the homotopy groups of a simplicial complex.
Note that the problem of deciding whether the first homotopy  group of a simplicial complex is trivial or not is an undecidable problem (this fact follows from the undecidability of the word problem for groups \cite{Novikov}). 
\end{remark}

\subsubsection{$\mathcal{P}$-formulas and $\mathcal{P}$-semi-algebraic sets}
\begin{notation}[Realizations, $\mathcal{P}$-, $\mathcal{P}$-closed
semi-algebraic sets]
  \label{not:sign-condition} 
  For any finite set of polynomials $\mathcal{P}
  \subset \R [ X_{1} , \ldots ,X_{k} ]$, 
  we call any quantifier-free first order formula $\phi$ with atoms, $P =0, P < 0, P>0, P \in \mathcal{P}$, to
  be a \emph{$\mathcal{P}$-formula}. 
  Given any semi-algebraic subset $Z \subset \R^k$,
  we call the realization of $\phi$ in $Z$,
  namely the semi-algebraic set
  \begin{eqnarray*}
    \RR(\phi,Z) & := & \{ \mathbf{x} \in Z \mid
    \phi (\mathbf{x})\}
  \end{eqnarray*}
  a \emph{$\mathcal{P}$-semi-algebraic subset of $Z$}.
  
  If $Z = \R^k$, we often denote the realization of $\phi$ in $\R^k$ by
  $\RR(\phi)$.
  
  If $\Phi = (\phi_j)_{j \in J}$ is a tuple of formulas indexed by a finite set $J$, 
  $Z \subset \R^k$ a semi-algebraic subset,
  we will denote by $\RR(\Phi,Z)$ the tuple $(\RR(\phi_j,Z))_{j \in J}$, and
  call it the realization of $\Phi$ in $Z$.
  For $J \subset J'$, we will denote by $\Phi|_{J'}$ the tuple $(\phi_j)_{j \in J'}$.
  
  We say that a quantifier-free formula $\phi$ is \emph{closed}  
  if it is a formula in disjunctive normal form with no negations, and with atoms of the form $P \geq 0, P \leq 0$ (resp. $P > 0, P < 0$),  
where $P \in \D[X_1,\ldots,X_k]$. If the set of polynomials appearing in a closed (resp. open) formula
is contained in a finite set $\mathcal{P}$, we will call such a formula a $\mathcal{P}$-closed 
formula,
 and we call
  the realization, $\RR \left(\phi \right)$, a \emph{$\mathcal{P}$-closed 
  semi-algebraic set}.
We say that a formula $\phi$ is a \emph{closed}-formula if $\phi$ is a $\mathcal{P}$-closed
formula for some finite set of polynomials $\mathcal{P}$.
  \end{notation}

We will also use the following notation.
  \begin{notation}
For $n \in \Z$ we denote by $[n] = \{0,\ldots,n\}$. In particular, $[-1] = \emptyset$. 
\end{notation}

 Finally, we are able to state the main results proved in this paper.
 \subsection{Efficient simplicial replacements of semi-algebraic sets}
  \begin{theorem}
  \label{thm:alg}
  There exists an algorithm that takes as input
  \begin{enumerate}
  \item
  a $\mathcal{P}$-closed 
  formula $\phi$ for some finite set 
  $\mathcal{P} \subset \D[X_1,\ldots,X_k]$;
  \item
  $\ell, 0 \leq \ell \leq k$;
  \end{enumerate}
  and produces as output a simplicial complex $\Delta_{\ell}(\phi)$ such that
  $|\Delta_{\ell}(\phi)| \overset{h}{\sim}_\ell \RR(\phi)$.
  The complexity of the algorithm  is bounded by $(s d)^{k^{O(\ell)}}$, where $s = \card(\mathcal{P})$ and
  $d = \max_{P \in \mathcal{P}} \deg(P)$.
  
  More generally, 
  there exists an algorithm that takes as input
  \begin{enumerate}
  \item
  a tuple $\Phi = (\phi_0,\ldots,\phi_{N})$ of $\mathcal{P}$-closed 
  formulas for some finite set 
  $\mathcal{P} \subset \D[X_1,\ldots,X_k]$;
  \item
  $\ell, 0 \leq \ell \leq k$;
  \end{enumerate}
  and produces as output a simplicial complex $\Delta_{\ell}(\Phi)$,
  and for each $J \subset [N]$ a subcomplex $\Delta_{\ell}(\Phi|_{J})$,
  such that 
  \[
  (J \mapsto |\Delta_{\ell}(\Phi|_{J})|)_{J \subset [N]}  \overset{h}{\sim}_\ell \Simp^{[N]}(\RR(\Phi)).
  \]
  The complexity of the algorithm is  bounded by $(N s d)^{k^{O(\ell)}}$, where $s = \card(\mathcal{P})$ and  $d = \max_{P \in \mathcal{P}} \deg(P)$.
  \end{theorem}
 
 Theorem~\ref{thm:alg} is valid over arbitrary real closed fields. In the special case of $\R = \mathbb{R}$, we have the following stronger version of Theorem~\ref{thm:alg}, where we are able to replace
 homological $\ell$-equivalence by $\ell$-equivalence. 
 
  \begin{manualtheorem}{\ref*{thm:alg}$^\prime$}
  \label{thm:alg'}
  Let $\R =\mathbb{R}$.
  There exists an algorithm that takes as input
  \begin{enumerate}
  \item
  a $\mathcal{P}$-closed 
  formula $\phi$ for some finite set 
  $\mathcal{P} \subset \D[X_1,\ldots,X_k]$;
  \item
  $\ell, 0 \leq \ell \leq k$;
  \end{enumerate}
  and produces as output a simplicial complex $\Delta_{\ell}(\phi)$ such that
  $|\Delta_{\ell}(\phi)| \sim_\ell \RR(\phi)$.
  The complexity of the algorithm is  bounded by $(s d)^{k^{O(\ell)}}$, where $s = \card(\mathcal{P})$ and
  $d = \max_{P \in \mathcal{P}} \deg(P)$.
  
  More generally, 
  there exists an algorithm that takes as input
  \begin{enumerate}
  \item
  a tuple $\Phi = (\phi_0,\ldots,\phi_{N})$ of $\mathcal{P}$-closed 
  formulas for some finite set 
  $\mathcal{P} \subset \D[X_1,\ldots,X_k]$;
  \item
  $\ell, 0 \leq \ell \leq k$;
  \end{enumerate}
  and produces as output a simplicial complex $\Delta_{\ell}(\Phi)$, 
  and for each $J \subset [N]$ a subcomplex $\Delta_{\ell}(\Phi|_{J})$
  such that 
  \[
  (J \mapsto |\Delta_{\ell}(\Phi|_{J})|)_{J \subset [N]}  \sim_\ell \Simp^{[N]}(\RR(\Phi)).
  \]
  The complexity of the algorithm is  bounded by $ (N s d)^{k^{O(\ell)}}$, where $s = \card(\mathcal{P})$ and
  $d = \max_{P \in \mathcal{P}} \deg(P)$.
  \end{manualtheorem}

 \begin{remark}
 \label{rem:homology-vs-homotopy}
 One main tool that we use is the Vietoris-Begle theorem (see proofs of Claims \ref{thm:main:proof:claim:1}, \ref{thm:main:proof:claim:2}). Since, there are many versions of the Vietoris-Begle theorem in the literature we make precise what we use below.
 
 It follows from \cite[Main Theorem]{Smale} that if
 $X \subset \mathbb{R}^m, Y \subset \mathbb{R}^n$ are compact semi-algebraic subsets (and so are locally contractible), and 
 $f:X \rightarrow Y$ is a semi-algebraic continuous map such that for every $y \in Y$, $f^{-1}(y)$ is $\ell$-connected, then 
 $f$ is an $\ell$-equivalence. 
 We will refer to this version of the Vietoris-Begle theorem as the 
 \emph{homotopy version of the Vietoris-Begle theorem}.
 Since, $\ell$-equivalence implies homological $\ell$-equivalence
 (see for example \cite[pp. 124, \S4.1B]{Viro-Fuchs}), $f$ is a homological
 $\ell$-equivalence  as well.
 
 Alternatively, if we assume that $f^{-1}(y)$ is only homologically $\ell$-connected for each $y \in Y$, then we can conclude that
 $f$ is a homological $\ell$-equivalence (see for example, the statement
 of the  Vietoris-Begle theorem in \cite{Ferry2018}). This latter
 theorem is also valid 
 for semi-algebraic maps between closed and bounded semi-algebraic sets
 over arbitrary real closed fields, once we know it for maps between compact semi-algebraic subsets over $\mathbb{R}$. This follows from a standard argument using the  Tarski-Seidenberg transfer principle and the fact that homology groups of closed bounded semi-algebraic sets can be defined in terms of finite triangulations. We will refer to this version of the Vietoris-Begle theorem  as
 the \emph{homological version of the Vietoris-Begle theorem}.
 \end{remark}

 \subsection{Application to computing homotopy groups of semi-algebraic sets}
 \label{subsec:homotopy}
 One important new contribution of the current paper compared to previous  algorithms
for computing topological invariants of semi-algebraic sets \cite{BPRbettione, Bas05-first} is that
for any given semi-algebraic subset $S \subset \mathbb{R}^k$, our algorithms give information on not just the 
homology groups but the homotopy groups of $S$ as well.

Computing homotopy groups of semi-algebraic sets is a considerably harder problem than 
 computing homology groups. There is no algorithm to decide whether the 
 fundamental group of a finite simplicial complex is trivial \cite{Novikov}. 
 As such the problem of deciding whether the fundamental
 group of any semi-algebraic subset $S \subset \mathbb{R}^k$  
 is trivial or not is an undecidable problem. 
 
On the other hand algorithms for computing topological invariants of a given semi-algebraic set $S \subset \R^k$, defined by a $\mathcal{P}$-formula where $\mathcal{P} \subset \D[X_1,\ldots,X_k]$, usually involve two kinds of operations. 
 \begin{enumerate}[(a)]
 \item Arithmetic operations and comparisons amongst elements of the ring $\D$;
 \item Operations that do not involve elements of $\D$.
 \end{enumerate}
 In our complexity bounds we only count the first kind of operations (i.e. those which involve elements 
 of $\D$).
 
 From this point of view  it makes sense to ask for any algorithmic problem involving formulas defined 
 over $\D$, if there is a reduction to another problem whose input is independent of $\D$.
 Theorem~\ref{thm:alg'} gives precisely such a reduction for computing the first $\ell$ homotopy groups of any given semi-algebraic set defined by a formula involving coefficients from any fixed
 subring $\D \subset \mathbb{R}$. 
 
 \begin{corollary}
 \label{cor:alg':1}
 For every fixed $\ell$, and an ordered domain $\D \subset \mathbb{R}$, there exists a 
 a reduction of the problem of 
 computing the first $\ell$ homotopy groups of a semi-algebraic set defined by a 
 quantifier-free formula with coefficients in $\D$,  to that
 of the problem of computing the first $\ell$ homotopy groups of a finite simplicial complex.
 The complexity of this reduction is bounded singly exponentially in the size of the input.
 \end{corollary} 
  
 While the problem of computing the fundamental group as well as the higher homotopy groups
 of a finite simplicial complex is clearly an extremely challenging problem, there has been recent
 breakthroughs. If a simplicial complex $K$  is $1$-connected then \v{C}adek et al. \cite{Cadek} has given
 an algorithm for computing a description of the homotopy groups
 $\pi_i(|K|)$, $2 \leq i \leq \ell$, which has complexity polynomially bounded in the size of the simplicial
 complex $K$ for every fixed $\ell$. This result coupled with Theorem~\ref{thm:alg'} gives the
 following corollary.
 
 \begin{corollary}
 \label{cor:alg':2}
 Let $\R =\mathbb{R}, \D \subset \R$ and $\ell \geq 2$.
  There exists an algorithm that takes as input
  \begin{enumerate}
  \item
  a $\mathcal{P}$-closed 
  formula $\phi$ for some finite set 
  $\mathcal{P} \subset \D[X_1,\ldots,X_k]$;
  \item
  $\ell, 0 \leq \ell \leq k$;
  \end{enumerate}
  such that $\RR(\phi)$ is simply connected,
  and outputs descriptions of the abelian groups $\pi_i(\RR(\phi))$, $2\leq i \leq \ell$ in terms
  of generators and relations.
 
  The complexity of the algorithm is  bounded by $(s d)^{k^{O(\ell)}}$, where $s = \card(\mathcal{P})$ and
  $d = \max_{P \in \mathcal{P}} \deg(P)$.
  \end{corollary}
  
  \begin{remark}
 Note that we do not have an effective algorithm for checking the hypothesis that the given semi-algebraic set is simply connected.
  \end{remark}

\subsection{Comparison with prior and related results}
\label{subsec:prior}
As stated previously, there is no algorithm known for computing the Betti numbers of semi-algebraic sets
having singly exponential complexity. However, algorithms with singly exponential complexity are known for
computing certain (small) Betti numbers. The zero-th Betti number of a semi-algebraic set is just the number of its semi-algebraically connected components. 
Counting the number of  semi-algebraically connected components  of a given semi-algebraic set
is a well-studied problem and algorithms with singly exponential complexity are known for solving this problem \cite{BPR99, GV92, Canny93a}. 
In \cite{BPRbettione} a singly exponential complexity algorithm is given for computing the first Betti number of semi-algebraic sets, and this was extended to the first $\ell$ (for any fixed constant $\ell$) Betti numbers in \cite{Bas05-first}. These algorithms do not 
produce a simplicial complex homotopy equivalent (or $\ell$-equivalent) to the given semi-algebraic set.

In \cite{BCL2019, BCT2020.1, BCT2020.2}, the authors take a different approach. Working over $\mathbb{R}$, and given a 
well-conditioned semi-algebraic subset $S\subset \mathbb{R}^k$,  
they compute a witness complex whose geometric realization is $k$-equivalent to $S$. The size of this witness
complex is bounded singly exponentially in $k$. However, the complexity depends on the condition number of the input 
 (and so this bound is not uniform), and the algorithm will fail for ill-conditioned input when the condition number becomes
 infinite. This is unlike the kind of algorithms we consider in the current paper, which are supposed to work for all inputs
 and with uniform complexity upper bounds. So these approaches are not comparable.
 
 While the approaches in \cite{BPRbettione, Bas05-first} and those in \cite{BCL2019, BCT2020.1, BCT2020.2} are not comparable, since the meaning of what constitutes
 an algorithm and the notion of complexity are different, there is a common connection between the results of these papers and those in the current paper which we elucidate below.
 
\subsubsection{Covers}
\label{subsubsec:covers}
A standard method in algebraic topology  for computing homology/cohomology 
 of a  space  $X$ is by means of an appropriately chosen \emph{cover}, $(V_\alpha \subset X)_{\alpha \in I}$, of  $X$ by
open or closed subsets.  
Suppose that  $X \subset \mathbb{R}^k$ is a closed or open semi-algebraic set. 
 Let $\mathcal{V} = (V_\alpha \subset X)_{\alpha \in I}$ be a finite cover of $X$ by open or closed semi-algebraic subsets, such that for each non-empty 
subset $J \subset I$, the intersection $V_J = \bigcap_{\alpha \in J} V_\alpha$ is either empty or contractible. We will say that such covers have the \emph{Leray property} and refer to them as \emph{Leray covers}.
 One can then associate to the cover $\mathcal{V}$, a simplicial complex, $\Nerve(\mathcal{V})$, the nerve of $\mathcal{V}$ defined as follows.
 
The set of $p$-simplices of $\Nerve(\mathcal{V})$ is defined by
 \[
 \Nerve(\mathcal{V})_p = \{ \{\alpha_0,\ldots,\alpha_p\} \subset 2^I \mid
V_{\alpha_0} \cap \cdots \cap V_{\alpha_p} \neq \emptyset\}.
\] 
It follows from a classical result of algebraic topology that the geometric realization $|\Nerve(\mathcal{V})|$
is homotopy equivalent to $X$, and moreover for each $\ell \geq 0$, the  geometric realization of the
$(\ell+1)$-st skeleton of $\Nerve(\mathcal{V})$,
\[
\mathrm{sk}_{\ell+1}(\Nerve(\mathcal{V})) = \{ \sigma \in \Nerve(\mathcal{V}) \mid
\card(\sigma) \leq \ell+2\}.
\]
 is homologically $\ell$-equivalent (resp. $\ell$-equivalent) to $X$ (resp. when $\R = \mathbb{R}$). 

The algorithms for computing the Betti numbers in \cite{BCL2019, BCT2020.1, BCT2020.2} proceeds by computing the $k$-skeleton
of the nerve of a cover having the Leray property whose size is bounded singly exponentially in $k$, and computing the simplicial homology groups of this complex. 
However, the bound on the size of the cover is not uniform but depends on
a real valued parameter -- namely the condition number of the input -- and hence the size of the cover can become infinite.
In fact, computing a singly exponential sized cover by semi-algebraic subsets having the Leray property
of an arbitrary semi-algebraic sets is an open
problem. If one solves this problem then one would also solve immediately the problem of designing an algorithm for computing all the Betti numbers of arbitrary semi-algebraic sets with singly exponential complexity in full generality.

The algorithms in \cite{BPRbettione,Bas05-first} which are able to compute some of the Betti numbers in dimensions $>0$  also depends on the existence of small covers having size bounded singly exponentially, albeit satisfying a much weaker property than the Leray property. The weaker property is that only the sets $V_\alpha, \alpha \in I$ (i.e. the elements of the cover) are contractible. No assumption is
made on the non-trivial finite intersections amongst the sets of the cover. Covers satisfying this weaker property can indeed
be computed with singly exponential complexity (this is one of the main results of \cite{BPRbettione} but see Remark \ref{rem:cover}), and using this fact one
is able to compute the first $\ell$ Betti numbers of semi-algebraic subsets of $\R^k$ for every fixed $\ell$ with singly exponential complexity. The 
algorithms in \cite{BPRbettione} and \cite{Bas05-first} do not construct a simplicial complex homotopy equivalent or $\ell$-equivalent to the given semi-algebraic set $S$ unlike  
the algorithm in \cite{BCL2019}.
\subsubsection{Main technical contribution}
\label{subsubsec:technical}
The main technical result that underlies the algorithmic result 
of the current paper is the following. Fix
$0 \leq \ell \leq k$. Suppose for every closed and bounded semi-algebraic set $S$ one has a covering of $S$ by 
closed and bounded semi-algebraic subsets which are $\ell$-connected (see Definition~\ref{def:ell-connected}) and which
has singly exponentially bounded complexity (meaning that the number of sets in the cover, the number of polynomials used in the quantifier-free formulas defining these sets and their degrees are all bounded singly exponentially in $k$). Moreover, since it is clear that contractible covers with singly exponential complexity exists, this is not a vacuous assumption. Using $\ell$-connected covers repeatedly  we build a simplicial 
complex of size bounded singly exponentially which is $\ell$-equivalent to the given semi-algebraic set. The definition
of this simplicial complex is a bit involved (much more involved than the nerve complex of a Leray cover) and appears in Section \ref{sec:simplicial-replacement}.  Its main properties are encapsulated in Theorem~\ref{thm:main}.

Two remarks are in order.

\begin{remark}
\begin{enumerate}[1.]
\item
Firstly, the Leray property can be weakened to require that  for every $t$-wise intersection,
$V_J, \card(J) = t$ is either empty or $(\ell-t+1)$-connected \cite{Bjorner}. 
We call this the \emph{$\ell$-Leray property}. The nerve complex,
$\Nerve(\mathcal{V})$ is then $\ell$-equivalent to $X$ \cite{Bjorner}. However, the property that we
use is much weaker -- namely that only the elements of the cover are $\ell$-connected and we 
make no assumptions on the connectivity of the intersections of two or more sets of the cover. 
This is due to the fact that controlling the connectivity of the intersections is very difficult and we do not
know of any algorithm with singly exponential complexity
for computing covers having the $\ell$-Leray property for $\ell \geq 1$.

\item
Secondly, note that 
to be $\ell$-connected
is a weaker property than being contractible. Unfortunately, at present we do not know of algorithms for computing $\ell$-connected covers, for
$\ell > 0$ that has much better complexity asymptotically than the algorithm in \cite{BPRbettione} for computing
covers by contractible semi-algebraic sets. However, it is still possible that there could be algorithms with much better complexity 
for computing $\ell$-connected covers (at least for small $\ell$) compared to computing contractible covers. 
\end{enumerate}
\end{remark}

\section{Simplicial replacement in an abstract setting}
\label{sec:simplicial-replacement}
We now arrive at the technical core of the paper.
Given a finite set $J$,  a tuple, 
$\Phi = (\phi_j)_{j \in J}$,  of 
closed formulas with $k$ free variables, and numbers $i,m \geq 0$,
we will describe the construction of a poset,
that we denote by $\mathbf{P}_{m,i}(\Phi)$.
We will assume that the  realizations, $\RR(\phi_j), j \in J$, of the  formulas in the tuple are homologically $\ell$-connected semi-algebraic subsets of $\R^k$ for some $\ell \geq 0$.  
In case $\R = \mathbb{R}$,  substitute 
``$\ell$-connected''  for ``homologically $\ell$-connected''. 
The poset $\mathbf{P}_{m,i}(\Phi)$ will have the property that
the geometric realization of its order complex, $\Delta(\mathbf{P}_{m,i}(\Phi))$,  is 
homologically $(m-1)$-equivalent ($(m-1)$-equivalent if $\R = \mathbb{R}$) to 
$\RR(\Phi)^J$. More generally, for each $J' \subset J$, $\mathbf{P}_{m,i}(\Phi|_{J'})$ can be identified
as a subposet of $\mathbf{P}_{m,i}(\Phi)$, and the diagram of inclusions of the corresponding 
geometric realizations is  homologically $(m-1)$-equivalent to the diagram $\Simp^J(\RR(\Phi))$ 
($(m-1)$-equivalent if $\R = \mathbb{R}$) (cf. Theorems~\ref{thm:main} and \ref{thm:main'}).
The poset $\mathbf{P}_{m,i}(\Phi)$ will then encode in a finite combinatorial way information which
determines the first $m$ homotopy groups of $\RR(\Phi)^{J'}$ for all $J' \subset J$, and the morphisms
$\pi_h(\RR(\Phi)^{J'}) \rightarrow \pi_h(\RR(\Phi)^{J''})$ induced by inclusions, 
for $0 \leq h \leq m-1$ and $J' \subset J'' \subset J$.  
(The significance of the subscript $i$ in the notation $\mathbf{P}_{m,i}(\Phi)$ will become clear later.)

\subsection{Outline of the main idea}
We begin with an outline explaining the main ideas behind the construction. First observe that if the realizations
of the sets in the given tuple, in addition to being $\ell$-connected, satisfies the $\ell$-Leray property (i.e.
each $t$-wise intersections amongst them is  $(\ell -t +1)$-connected), then it follows from
\cite{Bjorner} that the poset of the non-empty intersections (with the poset relation being canonical inclusions) satisfies the property that the geometric realization of its order complex (see Definition~\ref{def:order-complex}) 
is $\ell$-equivalent to $\RR(\Phi)^J$. The  same is true for all the subposets obtained by restricting the intersections to only amongst those indexed by some subset $J' \subset J$.
However if the $\ell$-Leray property fails to hold then the poset of canonical inclusions may fail to have the desired property.

Consider for example, the tuple 
\[
\Phi=(\phi_0,\phi_1),
\]
where 
\begin{eqnarray*}
\phi_0 &:=& (X_1^2 + X_2^2 -1 = 0) \wedge (X_2 \geq 0), \\
\phi_1 &:=& (X_1^2 + X_2^2 -1 = 0) \wedge (X_2 \leq 0).
\end{eqnarray*}
The realizations $\RR(\phi_0), \RR(\phi_1)$ are the upper and lower semi-circles covering the unit circle in the
plane.

The intersection $\RR(\phi_0) \cap \RR(\phi_1) = \RR(\phi_0 \wedge \phi_1)$ is the disjoint union of two points. The Hasse diagram of the  poset of canonical inclusions of the sets defined by $\phi_0$, $\phi_1$, and $\phi_0 \wedge \phi_1$ is:

\[
\xymatrix{
\phi_0  && \phi_1 \\
&\phi_0 \wedge \phi_1 \ar[lu] \ar[ru] &
}
\]
and the order complex of the poset is the simplicial complex shown in Figure~\ref{fig:circle1}.
The geometric realization  of the order complex is clearly not homotopy equivalent to the 
\[
\RR(\Phi)^{\{0,1\}} = \RR(\phi_0) \cup \RR(\phi_1)
\] 
which is equal to the unit circle. This is not surprising since the cover of the circle by the two closed semi-circle is not a Leray cover (and in fact not $\ell$-Leray for any $\ell \geq 0$).

\begin{figure}[h!]
	\begin{center}
		\includegraphics[scale=0.75]{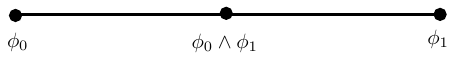}%
	\caption{\small Order complex for non-Leray cover}
	\label{fig:circle1}
	\end{center}
\end{figure}

One way of repairing this situation is to go one step further and choose a good (in this case $\infty$-connected) cover for the
intersection  $\RR(\phi_0) \cap \RR(\phi_1)$ defined by $\psi_0,\psi_1$, where
\begin{eqnarray*}
\psi_0 &:=& (X_1+1 = 0) \wedge (X_2 = 0), \\
\psi_1 &:=& (X_1-1=0) \wedge (X_2 = 0 ).
\end{eqnarray*}

The Hasse diagram of the  poset of canonical inclusions of the sets defined by $\phi_0$, $\phi_1$, $\psi_0$, and
$\psi_1$  

\[
\xymatrix{
\phi_0  & \phi_1 \\
\psi_0 \ar[u]\ar[ru] & \psi_1\ar[u] \ar[lu],
}
\]
and the order complex of the poset is shown in Figure~\ref{fig:circle2}.
It is easily seen to have the same homotopy type (homeomorphism type even in this case)
to the circle.
\begin{figure}[h]
	\centering
		\includegraphics[scale=0.75]{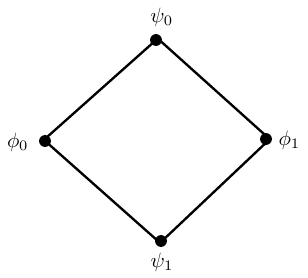}%
		\caption{\small Order complex for modified poset}
	\label{fig:circle2}
\end{figure}

The very simple example given above motivates the definition of the poset $\mathbf{P}_{m,i}(\Phi)$ in
general.  
We assume that we have available not just the given tuple of sets, and the non-empty
intersections amongst them, but also that we can cover any given non-empty intersections that arise in our
construction using $\ell$-connected closed (resp. open) semi-algebraic sets (we do not assume that these covers satisfy the stronger $\ell$-Leray property). The poset we define depends on the choice of 
these covers and not just on the formulas in the tuple $\Phi$ (unlike the standard nerve complex 
of the tuple $\RR(\Phi)$).
The choices that we make are encapsulated in the functions
$\mathcal{I}_{k,i}$ and $\mathcal{C}_{k,i}$ below. In practice, they would correspond to some
effective algorithm for computing well-connected covers of semi-algebraic sets.

\begin{remark}
\label{rem:cover}
There is one technical detail that serves to obscure a little the clarity of the construction. 
It arises due to the fact that the only algorithm with single exponential complexity  that exists in the literature
for computing well connected ($\infty$-connected or equivalently contractible) covers is the one
in \cite{BPRbettione}.
However, the algorithm
requires that the polynomials describing the given set
$S$ be in strong general position (see Definition~\ref{def:general-position}). In order to satisfy this requirement one needs to
initially perturb the given polynomials and replace the given set by another one, say $S'$, which is infinitesimally larger but has the same homotopy type as the given set $S$ (see Lemma~\ref{lem:monotone}). The algorithm then computes 
closed formulas whose realizations cover $S'$ and moreover are each semi-algebraically contractible. While there is a semi-algebraic retraction from $S'$ to $S$, this retraction is not guaranteed to restrict to
the elements of the cover. 
Our poset construction  is designed to be compatible with the fact that the covers we assume to exist
actually are covers of  infinitesimally larger sets   (i.e. that of $S'$ instead of $S$ following the notation of the previous sentence).
This 
necessitates the use of iterated Puiseux extensions in what follows.
\end{remark}

Of course, the introduction of infinitesimals  could be avoided by choosing sufficiently small 
positive elements in the field $\R$ itself and thus avoid making extensions. This would be more difficult to
visualize, and so we prefer to use the language of infinitesimal extensions. In the special case when
$\R = \mathbb{R}$, we prefer not to make non-archimedean extensions, since we discuss homotopy groups, so we take the alternative approach. However, we believe that the infinitesimal language is conceptually easier to grasp and so we use it in the general case.

Before giving the definition of the poset we first need to introduce
some mathematical preliminaries and notation.

\subsection{Real closed extensions and Puiseux series}
We will need some
properties of Puiseux series with coefficients in a real closed field. We
refer the reader to \cite{BPRbook2} for further details.

\begin{notation}
  For $\R$ a real closed field we denote by $\R \left\langle \eps
  \right\rangle$ the real closed field of algebraic Puiseux series in $\eps$
  with coefficients in $\R$. We use the notation $\R \left\langle \eps_{1},
  \ldots, \eps_{m} \right\rangle$ to denote the real closed field $\R
  \left\langle \eps_{1} \right\rangle \left\langle \eps_{2} \right\rangle
  \cdots \left\langle \eps_{m} \right\rangle$. Note that in the unique
  ordering of the field $\R \left\langle \eps_{1}, \ldots, \eps_{m}
  \right\rangle$, $0< \eps_{m} \ll \eps_{m-1} \ll \cdots \ll \eps_{1} \ll 1$.
  
 If $\bar\eps$ denotes the (possibly infinite) sequence $(\eps_1,\eps_2,\ldots)$ we will denote by
 $\R\la\bar\eps\ra$ the real closed field $\bigcup_{m\geq 0} \R\la\eps_1,\ldots,\eps_m\ra$.
 
Finally, given a finite sequence $(\bar\eps_1,\ldots,\bar\eps_m)$ we will denote by
$\R\la\bar\eps_1,\ldots,\bar\eps_m\ra$ the real closed field
$\R \left\langle \bar\eps_{1} \right\rangle \left\langle \bar\eps_{2} \right\rangle
  \cdots \left\langle \bar\eps_{m} \right\rangle$.
\end{notation}

\begin{notation}
\label{not:lim}
  For elements $x \in \R \left\langle \eps \right\rangle$ which are bounded
  over $\R$ we denote by $\lim_{\eps}  x$ to be the image in $\R$ under the
  usual map that sets $\eps$ to $0$ in the Puiseux series $x$.
\end{notation}

\begin{notation}
\label{not:extension}
  If $\R'$ is a real closed extension of a real closed field $\R$, and $S
  \subset \R^{k}$ is a semi-algebraic set defined by a first-order formula
  with coefficients in $\R$, then we will denote by $\E(S, \R') \subset \R'^{k}$ the semi-algebraic subset of $\R'^{k}$ defined by
  the same formula.
 \footnote{Not to be confused with the homological functor $\mathrm{Ext}(\cdot,\cdot)$ which
 unfortunately also appears in this paper.}
 It is well known that $\E(S, \R')$ does
  not depend on the choice of the formula defining $S$ 
  \cite[Proposition 2.87]{BPRbook2}.
\end{notation}

\begin{notation}
\label{not:monotone}
Suppose $\R$ is a real closed field,
and let $X \subset \R^k$ be a closed and bounded  semi-algebraic subset, and $X^+ \subset \R\la\eps\ra^k$
be a semi-algebraic subset bounded over $\R$. 
Let for $t \in \R, t >0$, $\widetilde{X}^+_{t} \subset \R^k$ denote the semi-algebraic
subset obtained by replacing $\eps$ in the formula defining $X^+$ by $t$, and it is
clear that  for $0 < t \ll 1$, $\widetilde{X}^+_t$ does not depend on the formula chosen. We say that $X^+$ is \emph{monotonically decreasing to $X$}, and denote $X^+ \searrow X$ if the following conditions are satisfied.
\begin{enumerate}[(a)]
\item
for all $0 < t < t'  \ll 1$,  $\widetilde{X}^+_{t} \subset  \widetilde{X}^+_{t'}$;
\item
\[
\bigcap_{t > 0} \widetilde{X}^+_{t} = X;
\]
or equivalently $\lim_\eps X^+ =  X$.
\end{enumerate}
More generally,
if $X \subset \R^k$ be a closed and bounded  semi-algebraic subset, and $X^+ \subset \R\la\eps_1,\ldots,\eps_m\ra^k$
a semi-algebraic subset bounded over $\R$,
we will say $X^+ \searrow X$ if and only if 
\[
X^+_{m+1} = X^+ \searrow X^+_m,  \;X^+_m \searrow X^+_{m-1}, \ldots, X^+_{2} \searrow X^+_1 = X,
\]
where for $i=1,\ldots, m$, $X^+_i = \lim_{\eps_i} X^+_{i+1}$. 

Note that if $\bar\eps= (\eps_1,\eps_2,\ldots)$ is an infinite sequence,  and $X^+ \subset \R\la\bar\eps\ra^k$ is a semi-algebraic subset bounded over $\R$,   
then there exists $m \geq 1$,  and semi-algebraic subset
$X^+_m \subset \R\la\eps_1,\ldots,\eps_m\ra^k$ closed and bounded over $\R$, 
such that $X^+ = \E(X^+_m,\R\la\bar\eps\ra)$.

In this case,  if $X \subset \R^k$ be a closed and bounded  semi-algebraic subset, we will say
$X^+ \searrow X$ if and only if 
\[
X^+_{m+1} = X^+ \searrow X^+_m,  \;X^+_m \searrow X^+_{m-1}, \ldots, X^+_{2} \searrow X^+_1 = X,
\]
where for $i=1,\ldots, m$, $X^+_i = \lim_{\eps_i} X^+_{i+1}$. 

Finally,
if $\bar\eps_1,\ldots,\bar\eps_m$ are sequences (possibly infinite), 
$X \subset \R^k$ be a closed and bounded  semi-algebraic subset, and $X^+ \subset \R\la\bar\eps_1,\ldots,\bar\eps_m\ra^k$
a semi-algebraic subset bounded over $\R$,
we will say $X^+ \searrow X$ if and only if 
\[
X^+_{m+1} = X^+ \searrow X^+_m,  \;X^+_m \searrow X^+_{m-1}, \ldots, X^+_{2} \searrow X^+_1 = X,
\]
where for $i=1,\ldots, m$, $X^+_i = \lim_{\bar\eps_i} X^+_{i+1}$. 
\end{notation}

The following lemma will be useful later.
\begin{lemma}
\label{lem:monotone}
Let $X \subset \R^k$ be a closed and bounded  semi-algebraic subset, and $X^+ \subset \R\la\bar\eps_1,\ldots,\bar\eps_m\ra^k$
a semi-algebraic subset bounded over $\R$, such that
$X^+ \searrow X$.
Then,
$\E(X,\R\la\bar\eps_1,\ldots,\bar\eps_m\ra)$ is semi-algebraic deformation retract of $X^+$.
\end{lemma}

\begin{proof}
See proof of Lemma 16.17 in 
\cite{BPRbook2}.
\end{proof}

\begin{notation}
\label{not:ball}
  For $x \in \R^{k}$ and $R \in \R$, $R>0$, we will denote by $B_{k} (0,R)$
  the open Euclidean ball centered at $0$ of radius $R$.
  We will denote by $\overline{B_k(0,R)}$ the closed Euclidean ball  centered at $0$ of radius $R$.  
   If $\R'$ is a real
  closed extension of the real closed field $\R$ and when the context is
  clear, we will continue to denote by $B_{k} (0,R)$ the extension $\E(B_{k} (0,R) , \R')$, and
  similarly for $\overline{B_k(0,R)}$. This should not cause any confusion.
  Similarly, we will denote by $\Sphere^{k-1}(0,R)$ the sphere of dimension $k-1$ in $\R^k$
  centered at $0$ of radius $R$.
\end{notation}

We refer the reader to \cite[Chapter 6]{BPRbook2} for the definitions of homology and
cohomology groups of semi-algebraic sets over arbitrary real closed fields.

\subsection{Definition of the poset $\mathbf{P}_{m,i}(\Phi)$}

\subsubsection{Simplified view of the definition of the poset $\mathbf{P}_{m,i}(\Phi)$}
Before giving a precise definition of the poset  $\mathbf{P}_{m,i}(\Phi)$, 
we first give a simplified version. 
We make the following two simplifications 
in order to illustrate the key idea.
\begin{enumerate}[(a)]
    \item 
    We ignore the role of the index $i$ in what follows. 
The necessity of the extra parameter $i$ is due to the fact that 
the hypothesis we assume (Hypothesis~\ref{hyp:black-box} in the following paragraph) 
is slightly stronger than we are able to assume 
for designing effective algorithms for computing the poset (see Remark~\ref{rem:cover}).
The actual hypothesis that we use is encapsulated in Property~\ref{property:thm:main} below.
\item
Secondly, 
in order to keep a geometric view of the construction, we will talk about tuples
$\mathcal{S} = (S_j)_{j \in J}$
of semi-algebraic sets,  instead of tuples of formulas $\Phi = (\phi_j)_{j \in J}$ 
defining them. As above, in order to give an effective algorithms, and analyzing its complexity,  we need to describe the poset in terms of formulas rather than sets, which we do in the precise definition that follows this simplified version.  
\end{enumerate}

We make the following hypothesis.

\begin{hypothesis}[Black-box hypothesis]
\label{hyp:black-box}
There exists a black-box (or algorithm) that given a closed and bounded  semi-algebraic set
$S \subset \mathbb{R}^k$ as input,
produces a cover $(S_\alpha)_{\alpha \in \mathcal{C}(S)}$ of $S$ by
closed and bounded $\ell$-connected semi-algebraic sets.
\end{hypothesis}

\begin{definition}[The order complex of a poset]
\label{def:order-complex}
Let $(\mathbf{P},\preceq)$ be a poset.  We denote by $\Delta(\mathbf{P})$ the simplicial
complex whose simplices are chains of $\mathbf{P}$.
\end{definition}

Suppose $\mathcal{S} = (S_j)_{j \in J}$
is a finite tuple of $\ell$-connected closed semi-algebraic subsets of $\mathbb{R}^k$. 

Our goal is to define a poset $\mathbf{P}_{m}(\mathcal{S})$ 
such that:
\begin{property}
\label{property:P}
\[
|\Delta(\mathbf{P}_{m}(\mathcal{S}))| \overset{ch}{\sim}_m
\mathcal{S}^J
\]
(see Definition~\ref{def:order-complex}).
We will say that the poset $\mathbf{P}_{m}(\mathcal{S})$ satisfies Property~\ref{property:P} for the pair $(m,\mathcal{S})$. 
\end{property}

\begin{remark}
\label{rem:cohomological}
We use cohomological $m$-equivalence in Property~\ref{property:P}.  In the final
construction we will lose a dimension while passing from cohomological
equivalence to (homological or homotopical) 
equivalence because of the use of the universal coefficients theorem (see the  proof of Claim~\ref{thm:main:proof:claim:5} inside the proof of Theorem~\ref{thm:main}), and we will end up with  
\[
|\Delta(\mathbf{P}_{m}(\mathcal{S}))| {\sim}_{m-1}
\mathcal{S}^J.
\]
\end{remark}

The main idea is to approximate homotopically the diagram of sets 
\[
(\mathcal{S}_I)_{I \subset J, \card(I) \leq m+2}
\]
(see Notation~\ref{not:diagram-Delta}),
and the inclusion maps 
\[
\mathcal{S}_{I'} \hookrightarrow \mathcal{S}_I, I \subset I',
\]
by a corresponding diagram
of (the geometric realizations of the order complexes of)  posets 
\[
(\mathbf{P}_{m - \card(I)+1,I})_{I \subset J, \card(I) \leq m+2}
\]
(where the poset $\mathbf{P}_{m - \card(I)+1,I}$ corresponds to
$\mathcal{S}_I$),
and poset
inclusions 
\[
\mathbf{P}_{m-\card(I')+1,I'} \hookrightarrow \mathbf{P}_{m - \card(I)+1,I}, I \subset I'.
\]

The construction is by induction on $m$ (we call this the global induction below).

\begin{enumerate}[1.]
    \item (Base case of the global induction,  $m=-1$.)  
    Suppose $\mathcal{S} = (S_j)_{j \in J}$
is a finite tuple of $\ell$-connected closed and bounded  semi-algebraic subsets of $\mathbb{R}^k$. 
We define the poset $\mathbf{P}_{-1}(\mathcal{S})$ to be just the index set $J$, with no non-trivial order relations. 
It is depicted in Figure~\ref{fig:poset:base-case}.
It is clear that 
$\mathbf{P}_{-1}(\mathcal{S})$ satisfies Property~\ref{property:P} for the pair
$(-1, \mathcal{S})$.

    \item (Induction hypothesis of the global induction.)
We assume that for each $m', -1 \leq m' <  m$, and each
finite tuple $\mathcal{S} = (S_j)_{j \in J}$ of $\ell$-connected closed and bounded  semi-algebraic subsets of $\mathbb{R}^k$,
we have defined a poset $\mathbf{P}_{m'}(\mathcal{S})$ satisfying 
Property~\ref{property:P} for the pair $(m',\mathcal{S})$.

    \item (Inductive step of the global induction, going from $<m$ to $m$.)
    Using the inductive hypothesis, 
we now define a poset $\mathbf{P}_{m}(\mathcal{S})$ satisfying 
Property~\ref{property:P} for the pair $(m,\mathcal{S})$,
for any tuple $\mathcal{S}$ of $\ell$-connected closed and bounded semi-algebraic subsets of $\mathbb{R}^k$.

Fix a finite tuple $\mathcal{S} = (S_j)_{j \in J}$
of $\ell$-connected closed and bounded semi-algebraic subsets of $\mathbb{R}^k$. 
We will define $\mathbf{P}_m(\mathcal{S})$ below in steps. 
The poset $\mathbf{P}_m(\mathcal{S})$ as a set will be a disjoint union 
of the index set $J$, and certain subposets $\mathbf{P}_{m - \card(I) +1,I}$, where
$I$ where $I \subset J, 2 \leq \card(I) \leq m+2$.
We define the subposets $\mathbf{P}_{m - \card(I) +1,I}$ by downward induction 
(we call this the local induction below)
on
$\card(I)$, starting from the base case, $\card(I) = m+2$.

\begin{enumerate}[(a)]
\item (Base case of the local induction, $\card(I) = m+2$.)
We first consider the semi-algebraic sets $\mathcal{S}_I, \card(I)= m+2$. Associated to each such $I$, 
we define a poset, which we denoted by $\mathbf{P}_{-1,I}$ as follows:
Using Hypothesis~\ref{hyp:black-box} 
applied to the semi-algebraic set 
$\mathcal{S}_I$ 
we obtain 
a cover 
$(\mathcal{S}_{I,\alpha})_{\alpha \in \mathcal{C}(\mathcal{S}_I)}$ 
of 
$\mathcal{S}_I$
by closed and bounded $\ell$-connected  semi-algebraic sets. 
We define 
\[
\mathbf{P}_{-1,I} = \mathbf{P}_{-1}((\mathcal{S}_{I,\alpha})_{\alpha \in \mathcal{C}(\mathcal{S}_I)}) = \mathcal{C}(\mathcal{S}_I)
\]
with no non-trivial order relation. 
It is depicted in Figure~\ref{fig:poset:base-case}. 
It  is clear that 
$\mathbf{P}_{-1,I}$ satisfies Property~\ref{property:P} for the pair
$(-1, (\mathcal{S}_{I,\alpha})_{\alpha \in \mathcal{C}(\mathcal{S}_I)})$.
\item (Going from $m+2$ to $m+1$.)
Next we consider subsets $I$ of cardinality $m+1$. For each such subset we construct 
a poset $\mathbf{P}_{0,I}$ satisfying two conditions:
\begin{enumerate}
    \item 
    For each set $I'$, with $\card(I') = \card(I)+1$, and $I \subset I'$, the poset $\mathbf{P}_{-1,I'}$ already defined is isomorphic to a sub-poset of $\mathbf{P}_{0,I}$; 
    \item 
    $|\Delta(\mathbf{P}_{0,I})|$ is cohomologically $0$-equivalent to $\mathcal{S}_I$.
\end{enumerate}

We apply Hypothesis~\ref{hyp:black-box}, 
to the semi-algebraic set 
$\mathcal{S}_I$
as input and obtain a cover 
$(\mathcal{S}_{I,\alpha})_{\alpha \in \mathcal{C}(\mathcal{S}_I)}$
of 
$\mathcal{S}_I$ 
by closed and bounded $\ell$-connected semi-algebraic sets.
We let 
\[
\mathbf{P}_{-1,I} = \mathbf{P}_{-1}((\mathcal{S}_{I,\alpha})_{\alpha \in \mathcal{C}(\mathcal{S}_I)}).
\]

Let $J_I$ be the union of the indexing set 
$\mathcal{C}(\mathcal{S}_I)$,
with the posets $\mathbf{P}_{-1,I'}$ for each $I'$ with $I \subset I', \card(I') = \card(I) + 1$. 
Notice that for each $\alpha \in J_I$, there is an $\ell$-connected closed and bounded
semi-algebraic set associated to it. Denote this set by $D(\alpha)$.

For every pair $\alpha,\beta \in J_I$, we again apply Hypothesis~\ref{hyp:black-box} to obtain a cover of
$D(\alpha) \cap D(\beta)$ by $\ell$-connected closed and bounded semi-algebraic sets,
$(S_{I,\alpha,\beta,\gamma})_{\gamma \in I_{\alpha,\beta}}$ 
where $I_{\alpha,\beta} = \mathcal{C}(D(\alpha) \cap D(\beta))$. The poset $\mathbf{P}_{0,I}$ is defined to be the set
$J_I \cup \bigcup_{\alpha,\beta \in J_I} I_{\alpha,\beta}$, and the non-trivial
order relations are $\gamma \precneq \alpha,\beta$ for each $\gamma \in I_{\alpha,\beta}$.
It is depicted in Figure~\ref{fig:poset:inductive-first-step}.

\item (Local induction hypothesis.) 
We assume that we have already defined the posets 
$\mathbf{P}_{m-\card(I')+1,I'}$, with  $\card(I') > \card(I)$. 

\item (Inductive step in general for the local induction.)
We construct the poset $\mathbf{P}_{m-\card(I)+1,I}$
as follows.
We apply Hypothesis~\ref{hyp:black-box} with the semi-algebraic set $\mathcal{S}_I$ as input and 
obtain
a cover 
$(\mathcal{S}_{I,\alpha})_{\alpha \in \mathcal{C}(\mathcal{S}_I)}$ 
of 
$\mathcal{S}_I$ 
by closed and bounded $\ell$-connected semi-algebraic sets.
Let $J_I$ be the union of the indexing set 
$\mathcal{C}(\mathcal{S}_I)$, 
with the posets $\mathbf{P}_{m-\card(I')+1,I'}$ for each $I'$ with $I \subset I', \card(I') = \card(I) + 1$. 
Notice that for each $\alpha \in J_I$, there is an $\ell$-connected closed and bounded
semi-algebraic set associated to it. Denote this set by $D(\alpha)$.

We define the poset $\mathbf{P}_{m - \card(I) + 1,I}$ using the global induction hypothesis.
The global inductive hypothesis gives us that  for any finite tuple of $\ell$-connected closed and bounded semi-algebraic set (in particular, the tuple of sets $(D(\alpha))_{\alpha \in J_I}$) we have defined a poset $\mathbf{P}_{m - \card(I)+1}((D(\alpha))_{\alpha \in J_I})$, 
which satisfies Property~\ref{property:P} for the pair $(m - \card(I)+1, (D(\alpha))_{\alpha \in J_I})$ (since $m - \card(I) +1 < m$). 

We define 
\[
\mathbf{P}_{m -\card(I)+1,I} = \mathbf{P}_{m - \card(I)+1}((D(\alpha))_{\alpha \in J_I}).
\]
This finishes the local induction and we have defined 
$\mathbf{P}_{m - \card(I) +1,I}$, for each
$I \subset J, 2 \leq \card(I) \leq m+2$.
\end{enumerate}

Finally, we define 
\begin{equation}
\label{eqn:P-simplified}
\mathbf{P}_{m}(\mathcal{S})
= J \cup  
\bigcup_{I \subset J, 2 \leq \card(I) \leq m+2} \mathbf{P}_{m -\card(I) +1, I}.
\end{equation}

The partial order in the poset $\mathbf{P}_{m}(\mathcal{S})$ is specified as follows.
By the local induction,  each of the poset $\mathbf{P}_{m -\card(I) +1, I}$ comes with a partial order.
We extend these orders as follows:
\begin{enumerate}[(a)]
    \item For each $I \subset I' \subset J$, with $2 \leq \card(I)\leq \card(I') \leq m+2$,
    there is a subposet of $\mathbf{P}_{m - \card(I)+1,I}$ canonically isomorphic to the poset
    $\mathbf{P}_{m - \card(I')+1,I'}$. For each element $\alpha$ of the former and the corresponding element $\alpha'$ of the latter we set $\alpha' \precneq \alpha$.
    
    \item For each $j \in J$, and $\alpha \in \mathbf{P}_{m -\card(I)+1,I}, j \in I$,
    we set the element $\alpha \precneq j$.
\end{enumerate}
This ends the definition of the poset $\mathbf{P}_m(\mathcal{S})$ completing the global induction. Figure~\ref{fig:poset:last-step} depicts $\mathbf{P}_{m}(\mathcal{S})$ in terms of subposets $\mathbf{P}_{m -\card(I)+1,I}$. In Claim~\ref{claim:proof:thm:main:complexity:4} we will show that the height of the poset $\mathbf{P}_{m}(\mathcal{S})$ is bounded by $2m+2$.

Notice that for any chain $\alpha_k \precneq \alpha_{k-1} \precneq \ldots \precneq \alpha_0$ of elements in $\mathbf{P}_{m}(\mathcal{S})$, we have a sequence of inclusion maps of semi-algebraic sets $D(\alpha_k) \hookrightarrow D(\alpha_{k-1}) \hookrightarrow \ldots \hookrightarrow D(\alpha_0)$. 
It is depicted in Figure~\ref{fig:poset} for a hypothetical space with four elements in the initial covering.
\end{enumerate}

\begin{figure}[h]
	\centering
	\subfigure[ {\footnotesize {{${\rm \bf P}_{-1,I}= {\rm \bf P}_{-1}((\mathcal{S}_{I,\alpha})_{\alpha \in \mathcal{C}(\mathcal{S}_I)})$}}: The elements of the poset, i.e. $J_I$, correspond to the elements of the cover $\mathcal{C}(\mathcal{S}_I)$, with no non-trivial order relation.}]{
		\label{fig:poset:base-case}
		\includegraphics[scale=0.74]{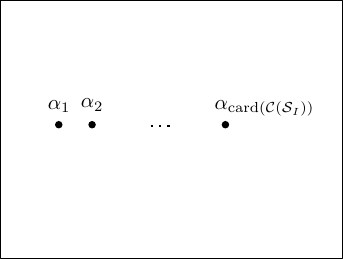}
	} 
	\hspace*{.2cm}
	\subfigure[{\footnotesize {{${\rm \bf P}_{0,I}$}}: 
		At the top level, the elements of ${\rm \bf P}_{0,I}$ correspond to the cover $\mathcal{C}(\mathcal{S}_I)$ and elements of the posets ${\rm \bf P}_{-1, I_i}$, where $\card(I_i) = \card(I)+1 $ and $I \subset I_i$. At the bottom level we have elements of the posets ${\rm \bf P}_{-1}((\mathcal{S}_{I,\alpha_i,\alpha_j,\gamma})_{\gamma \in \mathcal{C}(D(\alpha_i)\cap D(\alpha_j))})$---shown as a box---for every pair $\alpha_i$ and $\alpha_j$ at the top level. The order relations are between the pairs and the elements of their corresponding posets at the bottom level.}]{
		\label{fig:poset:inductive-first-step}
		\includegraphics[scale=0.65]{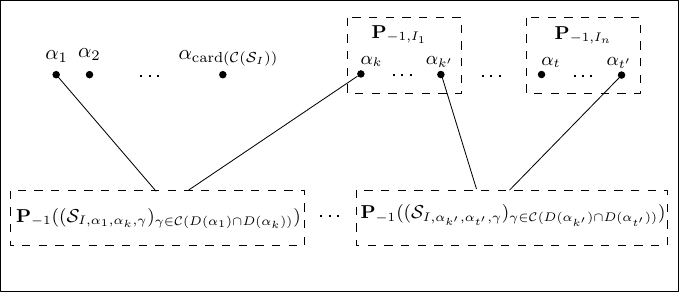} 
	}
	
	\subfigure[{\footnotesize { ${\rm \bf P}_{m}(\mathcal{S}) = J \cup \bigcup_{I \subset J, 2 \leq \card(I) \leq m+2} \mathbf{P}_{m -\card(I) +1, I}$}: The top level of the poset corresponds to the elements of $J$. Next, we have elements of the posets $\mathbf{P}_{m - 1, I_i^{(2)}}$ where $I^{(2)}_i \subset J$ and $\card(I^{(2)}_i) =2$---denoted by the superscript (2). Similarly at the lower levels, we have elements of the posets corresponded to subsets $I_i^{(m')} \subset J$ with $\card(I_i^{(m')})=m'$ and $m' \leq m+2$. The partial order relations are defined between $j \in \{1, \ldots, n\}$ at the top level and the elements of $\mathbf{P}_{m - 1, I_i^{(2)}}$, if $j \in I^{(2)}_i$. Furthermore, in addition to the order relations within each poset, if $I_j^{(m'-1)} \subset I_i^{(m')}$ then $\mathbf{P}_{m-m'+1 , I_i^{(m')}} \hookrightarrow \mathbf{P}_{m-m'+2 , I_j^{(m'-1)}}$, hence for each element $\alpha'$ of the $\mathbf{P}_{m-m'+1 , I_i^{(m')}}$ and the corresponding element $\alpha$ of the $\mathbf{P}_{m-m'+2 , I_j^{(m'-1)}}$ we set $\alpha' \precneq \alpha$.
	}]{
		\label{fig:poset:last-step}
		\includegraphics[width=\linewidth]{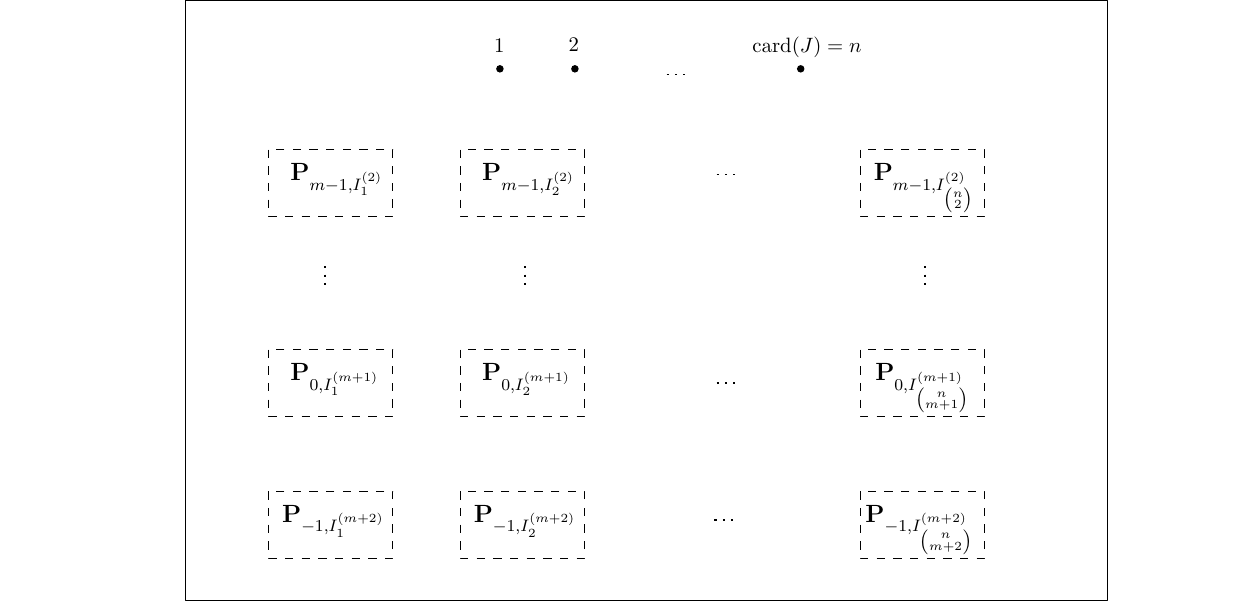} 
	}
		\vspace{-0.4cm}
	\caption{\small A simple illustration of the simplified view of the poset.} \label{fig:simplified_poset}
\end{figure}

\begin{figure}[h]
	\begin{center}
		\includegraphics[scale=0.60]{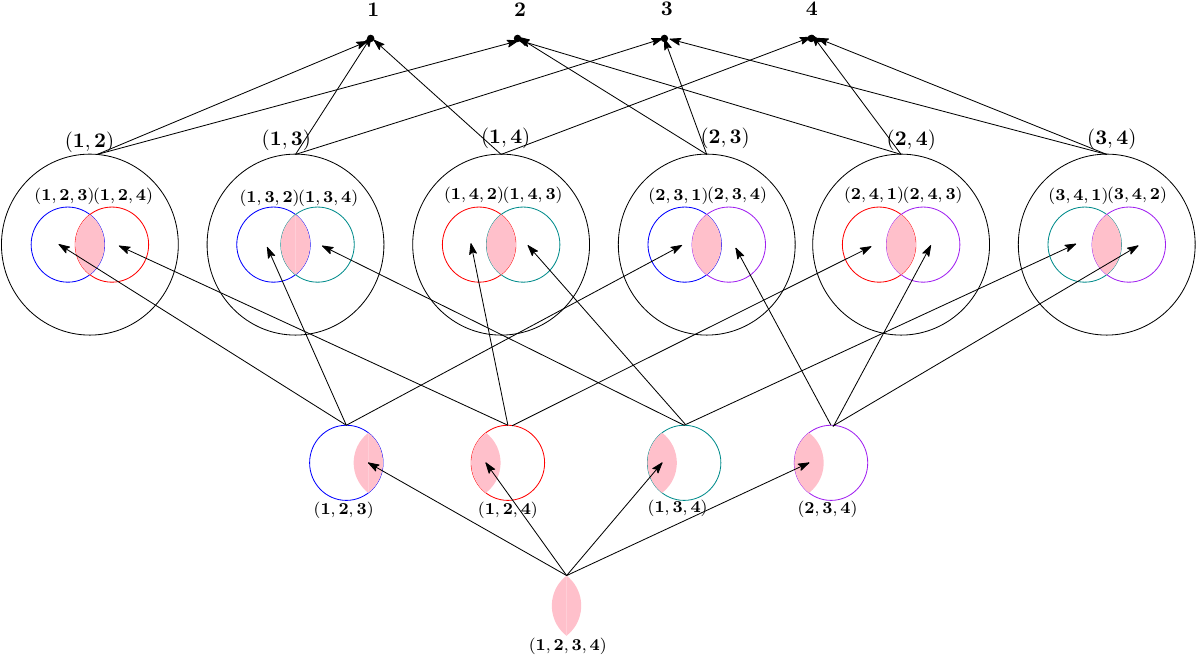}%
	\caption{\small Poset $\mathbf{P}_m(\mathcal{S})$ such that $|\Delta(\mathbf{P}_m(\mathcal{S}))|$ is $m$-equivalent to $\bigcup_{j \in J} S_j$ with $m=2$, $J = \{1,2,3,4\}$.}
	\label{fig:poset}
	\end{center}
\end{figure}

The following two examples are illustrative.

\begin{example}
Let $\ell = \infty, m \geq 2$, $\mathcal{S} = (S_1,S_2)$, where $S_1, S_2$ are the closed upper and lower hemispheres of the unit sphere in $\mathbb{R}^3$ 
(see Figure~\ref{Fig:f1}).

Using \eqref{eqn:P-simplified} we get
\begin{equation}
\label{eqn:P-simplified-example:1}
\mathbf{P}_m(\mathcal{S}) = \{1,2\} \cup \mathbf{P}_{m -2 +1,\{1,2\}}.
\end{equation}

Let 
$\mathcal{C}(\mathcal{S}_{\{1,2\}})$
be the cover of 
$\mathcal{S}_{\{1,2\}}$ 
by two closed semi-circles
$T_3, T_4$, and let $\mathcal{T} = (T_3,T_4)$. 

Note that $T_3 \cap T_4$ is a set containing two points $W_5,W_6$ (say), and 
the only possibility for $\mathcal{C}(T_3 \cap T_4)$, is the tuple $\mathcal{W} = (W_5,W_6)$.
Then, 
\begin{equation}
\label{eqn:P-simplified-example:2}
\mathbf{P}_{m - 1}(\mathcal{T}) = \{3, 4\} \cup \mathbf{P}_{m - 2, \{3,4\}}
\end{equation}
and the subposet $\mathbf{P}_{m - 2, \{3,4\}}$ is isomorphic to the
poset
\begin{equation}
\label{eqn:P-simplified-example:3}
\mathbf{P}_{m-2}(\mathcal{W}) = \{5,6\}.
\end{equation}
Substituting \eqref{eqn:P-simplified-example:3} into
\eqref{eqn:P-simplified-example:2}
and 
\eqref{eqn:P-simplified-example:2} into 
\eqref{eqn:P-simplified-example:1}
we finally obtain that the Hasse diagram of the poset $\mathbf{P}_m(\mathcal{S})$ is 
	\[
		\xymatrix{
			1 && 2 \\
			3 \ar[u]\ar[rru] &&  4 \ar[u]\ar[llu] \\
			5 \ar[u]\ar[rru] && 6 \ar[u]\ar[llu] 
		}
		\]
The order complex of this poset is homotopy equivalent (in fact, in this case is homeomorphic) to the sphere.
\end{example}

\begin{example}
Now let $\ell = m = 2$, $\mathcal{S} = (S_1,S_2)$, where $S_1, S_2$ are the closed upper and lower hemispheres of the unit sphere in $\mathbb{R}^k$,
$k > 5$. That is $S_1$ (resp. $S_2$) is the intersection of the unit sphere in $\R^k$,
with the set defined by $X_k \geq 0$ (resp. $X_k \leq 0$).

Using \eqref{eqn:P-simplified} we get
\begin{equation*}
\mathbf{P}_m(\mathcal{S}) = \{1,2\} \cup \mathbf{P}_{m -2 +1,\{1,2\}}.
\end{equation*}

Let 
$\mathcal{C}(\mathcal{S}_{\{1,2\}})$ 
be the cover of 
$\mathcal{S}_{\{1,2\}}$ 
by two closed semi-spheres 
$T_3, T_4$,
(i.e. $T_3$ (resp. $T_4$) is the intersection of 
$\mathcal{S}_{\{1,2\}}$ 
with $X_{k-1} \geq 0$ (resp. $X_{k-1} \leq 0$),
and let $\mathcal{T} = (T_3,T_4)$. 

Note that $W_5 = T_3 \cap T_4$ is a $(k-3)$-dimensional sphere, and
since $k > 5$, $W_5$ is $2$-connected and we can take 
$\mathcal{C}(W_5) = (W_5)$.
 
\begin{equation*}
\mathbf{P}_{1}(\mathcal{T}) = \{3, 4\} \cup \{5 \}
\end{equation*}
with Hasse diagram
\[
		\xymatrix{
			3 && 4 \\
		     &5 \ar[lu]\ar[ru]&
		}
		\]

Finally we obtain that the Hasse diagram of the poset $\mathbf{P}_2(\mathcal{S})$ is 
	\[
		\xymatrix{
			1 && 2 \\
			3 \ar[u]\ar[rru] &&  4 \ar[u]\ar[llu] \\
		 &5\ar[lu]\ar[ru]& 
		}
		\]
The order complex of this poset is contractible and is $2$-equivalent (but in this case not homotopy equivalent) to 
$\Sphere^{k-1}$ for $k >5$.
\end{example}

With the definition of the poset $\mathbf{P}_m(\mathcal{S})$ it is possible to
prove the following theorem. We do not include a proof of this theorem since it is subsumed
by Theorem~\ref{thm:main'}.
\begin{theorem*}
With the same notation as in the Definition of $\mathbf{P}_m(\mathcal{S})$ defined above:
\begin{equation*}
|\Delta(\mathbf{P}_{m}(\mathcal{S}))| {\sim}_{m-1} \bigcup_{j \in J} S_j.
\end{equation*}

More generally, we have the diagrammatic homological $(m-1)$-equivalence
\begin{equation*}
(J' \mapsto |\Delta(\mathbf{P}_{m}(\mathcal{S}|_{J'})|)_{J' \in 2^J}  \overset{h}{\sim}_{m-1} \Simp^J(\mathcal{S}),
\end{equation*}
where $\mathcal{S}|_{J'} = (S_j)_{j \in J'}$.
\end{theorem*}

We now return to the precise definition of the poset $\mathbf{P}_{m,i}(\Phi)$ that we
are going to the use.

\subsubsection{Precise definition of $\mathbf{P}_{m,i}(\Phi)$}
We begin with a few useful notation that we will use in the construction. 
\begin{notation}
We will denote by $\mathcal{F}_{\R,k}$ the set of 
quantifier-free formulas
with coefficients in $\R$ and $k$ variables,
whose realizations are closed in $\R^k$.
\end{notation}

We also use the following convenient notation.

\begin{notation}[The relation $\subset_{\leq n}$]
For any $n \in \Z_{\geq 0}$, and sets $A,B$, we will write $A \subset_{\leq n} B$
to mean $A \subset B$ and $0 < \card(A) \leq n$. 
\end{notation}

We are now in a position to define a poset associated to a given finite tuple of formulas that will play
the key technical role in the rest of the paper. 

We first fix the following.
\begin{enumerate}
\item
Let $\R = \R_0 \subset \R_1 \subset \R_2 \subset \cdots $
be a sequence of extensions of real closed fields.
\item
We also fix two  sequences of functions,
\[
\mathcal{I}_{i,k}: \mathcal{F}_{\R_i,k} \rightarrow \Z_{\geq -1},
\]
and
\[
\mathcal{C}_{i,k}: \mathcal{F}_{\R_i,k}  \rightarrow  \bigcup_{p \geq 0}(\mathcal{F}_{\R_{i+1},k})^{[p]},
\]
\end{enumerate}

\begin{remark}
The definition of the poset $\mathbf{P}_{m,i}(\cdot)$ given below does not depend on any specific
properties of the functions $\mathcal{I}_{i,k}(\cdot)$ and $\mathcal{C}_{i,k}(\cdot)$. 
Later we will prove key topological properties of 
$\mathbf{P}_{m,i}(\cdot)$ (see Theorems~\ref{thm:main} and \ref{thm:main'} below) under certain  assumptions on $\mathcal{I}_{i,k}(\cdot)$ and $\mathcal{C}_{i,k}(\cdot)$ (see Properties~\ref{property:thm:main} and 
\ref{property:thm:main'} below).
\end{remark}

For each $i \geq 0$, and 
$-1 \leq m \leq k$,  
a non-empty finite set $J$, and 
$\Phi \in (\mathcal{F}_{\R_i,k})^J$,
we define
a poset $(\mathbf{P}_{m,i}(\Phi), \prec)$.

We first need an auxilliary definition which will be used in the definition
of the underlying set, $\mathbf{P}_{m,i}(\Phi)$,
of the  poset $(\mathbf{P}_{m,i}(\Phi), \prec)$.

\begin{definition}
\label{def:J}
Let $J$ be a non-empty finite set, and $\Phi \in (\mathcal{F}_{\R_i,k})^J$. 
We first define for each  subset $I \subset_{\leq m+2} J$, 
a set $J_{m,i,I,\Phi}$, and 
an 
element $\Phi_{m,i,I,J} \in (\mathcal{F}_{\R_{i+1},k})^{J_{m,i,I,\Phi}}$
(using downward induction on $\card(I)$).

\noindent
Base case ($\card(I) = m+2$): In this case we define,
\begin{equation}
\label{eqn:J:base}
J_{m,i, I,\Phi} = \{I\} \times [ \mathcal{I}_{i,k}( \bigwedge_{j \in I} \Phi(j))],
\end{equation}
and for $(I,p) \in J_{m,i,I,\Phi}$,
\[
\Phi_{m,i, I,J}((I,p)) =  \mathcal{C}_{i,k}( \bigwedge_{j \in I} \Phi(j))(p).
\]
\noindent
Inductive step: Suppose we have defined $J_{m,i,I',\Phi}$ and $\Phi_{m,i, I',J}$ for all $I'$ with $\card(I') = \card(I) +1$.
We define 
\begin{equation}
\label{eqn:J:induction}
J_{m,i,I,\Phi} = \left(\{I\} \times   [ \mathcal{I}_{i,k}( \bigwedge_{j \in I} \Phi(j)] \right) \cup \bigcup_{I \subset I' \subset J, \card(I') = \card(I)+1} J_{m,i,I',\Phi},
\end{equation}
and 

\begin{eqnarray*}
\Phi_{m,i,I,J}(\alpha) &=&   \mathcal{C}_{i,k}( \bigwedge_{j \in I} \Phi(j))(p), \mbox{ if $\alpha =(I,p) \in \{I\} \times [ \mathcal{I}_{i,k}( \bigwedge_{j \in I} \Phi(j))]$}, \\
&=& 
\Phi_{m,i,I',J}(\alpha), \mbox{ if $\alpha \in J_{m,i,I',\Phi}$ for some $I' \supset I$, with}  \\
&& \hspace*{2.1cm} \card(I') = \card(I)+1.
\end{eqnarray*}
\end{definition}

The following properties of  
$J_{m,i,I,\Phi}$ and $\Phi_{m,i,I,J}$ are obvious from the above definition.
Using  the same notation as in Definition~\ref{def:J}:
\begin{lemma}
\label{lem:J}
\begin{enumerate}[(a)]
\item $\card(J_{m,i,I,\Phi}) < \infty$ for each $I \subset_{\leq m+2} J$;
\item  
For $I,I' \subset J$ with $\card(I \cup I') \leq m+2$,
\[
J_{m,i, I \cup I',\Phi}\subset J_{m,i,I,\Phi}\cap J_{m,i,I',\Phi}.
\]
\item
\label{itemlabel:lem:J:c}
If $I'\subset  I \subset_{\leq m+2} J \subset J'$, then
$J_{m,i,I,\Phi} \subset J'_{m,i,I',\Phi}$, and for $\alpha \in J_{m,i,I,\Phi}$,
$\Phi_{m,i,I,J}(\alpha) = \Phi_{m,i,I',J'}(\alpha)$.
\end{enumerate}
\end{lemma}

\begin{proof}
Follows directly from Definition~\ref{def:J}.
\end{proof}

We now define the  set
$\mathbf{P}_{m,i}(\Phi)$.

\begin{definition}[The underlying set 
of the poset $(\mathbf{P}_{m,i}(\Phi), \prec)$]
\label{def:poset}
We define the set $\mathbf{P}_{m,i}(\Phi)$ 
using induction on $m$. 

\noindent
Base case ($m=-1$):  For each finite set $J$, and $\Phi \in (\mathcal{F}_{\R_i,k})^J$ we define 

\[
\mathbf{P}_{-1,i}(\Phi) = \bigcup_{j \in J} \{\{j\}\} \times \{\emptyset \}.
\]

\noindent
Inductive step: Suppose we have defined the sets $(\mathbf{P}_{m',i'}(\Phi'),\prec)$ for all $m'$ with $-1 \leq m' < m$,  $i' \geq 0$,
for all non-empty finite sets $J'$ and all $\Phi' \in (\mathcal{F}_{\R_{i'},k})^{J'}$. 

We complete the inductive step by defining: 
\begin{equation}
\label{eqn:P}
\mathbf{P}_{m,i}(\Phi)  = \bigcup_{j \in J} \{\{j\}\} \times \{\emptyset\} \cup  \bigcup_{I \subset J, 1 < \card(I) \leq m+2}   \{I\} \times \mathbf{P}_{m -\card(I) +1,i+1}(\Phi_{m,i,I,J}).
\end{equation}
\end{definition}

We  now specify the partial order on 
$\mathbf{P}_{m,i}(\Phi)$.
For this it will be useful to have 
the following alternative characterization of the elements of the 
poset $\mathbf{P}_{m,i}(\Phi)$  as \emph{tuples of sets}. This characterization follows simply by unravelling the inductive definition of 
the set $\mathbf{P}_{m,i}(\Phi)$ given above.

\subsubsection{Characterization of the elements of the 
poset $\mathbf{P}_{m,i}(\Phi)$  as tuples of sets}
\label{subsubsec:tuple-of-sets}
The elements of $\mathbf{P}_{m,i}(\Phi)$ are all finite  tuples of sets (of varying lengths) 
\[
(I_0,\ldots,I_{r},\emptyset),
\]
satisfying the following conditions.
\begin{enumerate}[1.]
\item 
$I_0$ is a subset of  $J_0= J$, 
$\card(I_0) = 1$ if $r = 0$, and  
$2 \leq  \card(I_0) \leq m+2$ otherwise.
\item 
$I_1$ is a subset of 
$J_1 = \left(J_0\right)_{m_0,i_0,I_0,\Phi_0}$ (see Eqn. \eqref{eqn:J:induction}, Definition~\ref{def:poset})
with
\begin{eqnarray*}
m_0 &=& m,\\
i_0 &=& i, \\
\Phi_0 &=& \Phi,
\end{eqnarray*}
and 
\[
2 \leq \card(I_1) \leq (m_0 - \card(I_0) + 1) +2.
\]
\item
$I_2$ is a subset of $J_2 = \left(J_1 \right)_{m_1,i_1,I_1, \Phi_{1}}$, 
where
\begin{eqnarray*}
m_1 &=& m_0 - \card(I_0) + 1,\\
i_1 &=& i_0+1, \\
\Phi_{1} &=& (\Phi_0)_{m_0,i_0,I_0, J_0},
\end{eqnarray*}
and
\[
2 \leq \card(I_2) \leq (m_1- \card(I_1) +1) + 2.
\]
\item
Continuing in the above fashion, 
\begin{equation}
    \label{eqn:P-alternate}
    I_{r-1} \subset  J_{r-1} = \left(J_{r-2} \right)_{m_{r-2},i_{r-2}, I_{r-2}, \Phi_{r-2}},
\end{equation}
where 
\begin{eqnarray*}
m_{r-2} &=& m_{r-3} - \card(I_{r-3}) + 1, \\
i_{r-2} &=& i_{r-3}+1, \\
\Phi_{r-2} &=& (\Phi_{r-3})_{m_{r-3},i_{r-3},I_{r-3}, J_{r-3}},
\end{eqnarray*}
and
\begin{equation}
\label{eqn:P-alternate:2}
2 \leq \card(I_{r-1}) \leq m_{r-2} +2 =
(m +r -1 - \sum_{j=0}^{r-2}\card(I_j)) + 2.
\end{equation}
\item
Finally,
\[
I_{r} \subset 
J_{r} =\left(J_{r-1} \right)_{m_{r-1},i_{r-1},I_{r-1}, \Phi_{r -1}},
\]
where
\[
\Phi_{r -1} = (\Phi_{r -2})_{m_{r-2},i_{r-2},I_{r -2},J_{r-2}},
\]
and
\[
\card(I_r) =  1.
\]
\end{enumerate}
(We show later (see Claim~\ref{claim:proof:thm:main:complexity:1}) 
that for tuples $(I_0,\ldots,I_r,\emptyset)$ satisfying the above conditions,
$r \leq m+1$.)

\begin{definition}[Partial order on $\mathbf{P}_{m,i}(\Phi)$]
\label{def:order-on-P}
The partial order $\prec$ on $\mathbf{P}_{m,i}(\Phi)$ 
is defined  as follows.

For 
$
\alpha = (I^\alpha_0,\ldots,I^\alpha_{r_\alpha},\emptyset),\beta = (I^\beta_0,\ldots,I^\beta_{r_\beta},\emptyset) \in
 \mathbf{P}_{m,i}(\Phi),
$
 \begin{equation}
 \label{eqn:order-alternate}
 \beta \prec \alpha \Leftrightarrow (r_\alpha \leq r_\beta) \mbox{ and } I^\alpha_j \subset I^\beta_j, 0 \leq j \leq r_\alpha.  
 \end{equation}
\end{definition}

\subsection{Main properties of the poset $\mathbf{P}_{m,i}(\Phi)$}

We will now state and prove the important properties of the poset $\mathbf{P}_{m,i}(\Phi)$ that
motivates its definition. 

\begin{lemma}
\label{lem:property:poset}
For each $J' \subset  J'' \subset J$, and $-1 \leq m' \leq m'' \leq m$,
we have a poset inclusion,
\[
\mathbf{P}_{m',i}(\Phi|_{J'}) \hookrightarrow \mathbf{P}_{m'',i}(\Phi|_{J''}).
\]   
\end{lemma}

\begin{proof}
Follows from Definition~\ref{def:poset} and Part \eqref{itemlabel:lem:J:c} of Lemma~\ref{lem:J}.
\end{proof}

We now state a lemma  which will be useful later, that states a key property of the partial order relation in $\mathbf{P}_{m,i}(\Phi)$. 
Using  the same notation as in Definition~\ref{def:poset}:
\begin{lemma}
\label{lem:property-order}
Suppose that $I' \subset I \subset J$. 
\begin{enumerate}[(a)]
\item
\label{itemlabel:lem:property-order:a}
The poset
$
\mathbf{P}_{m -\card(I) +1,i+1}(\Phi_{m,i,I,J})
$ 
is a subposet of 
$
\mathbf{P}_{m -\card(I') +1,i+1}(\Phi_{m,i,I',J})
$.

\item
\label{itemlabel:lem:property-order:b}
For each $\alpha ,\alpha' \in \mathbf{P}_{m -\card(I) +1,i+1}(\Phi_{m,i,I,J})$,
\[
\alpha \prec_{\mathbf{P}_{m -\card(I) +1,i+1}(\Phi_{m,i,I,J})} \alpha' \Leftrightarrow (I,\alpha) \prec_{\mathbf{P}_{m,i}(\Phi)} (I',\alpha').
\]
\end{enumerate}
\end{lemma}

\begin{proof}
Part~\eqref{itemlabel:lem:property-order:a} follows from the fact that
$J_{m,i,I,\Phi} \subset J_{m,i,I',\Phi}$,
$m -\card(I) +1 \leq  m -\card(I') +1$, and 
Lemma~\ref{lem:property:poset}.

Part~\eqref{itemlabel:lem:property-order:b}
follows immediately from the definition of the partial order
on $\mathbf{P}_{m,i}(\Phi)$ (see Definition~\ref{def:order-on-P}).
\end{proof}

Let $\R$ be a real closed field and $R \in \R, R> 0$.
We say that the tuple  
\[
\left((\R_i)_{i \geq 0},R,k, (\mathcal{I}_{i,k})_{i \geq 0}, (\mathcal{C}_{i,k})_{i \geq 0} \right)
\]
satisfies  the \emph{homological $\ell$-connectivity property over $\R$}  if it satisfies the following conditions.

\begin{property}
\label{property:thm:main}
\begin{enumerate}[1.]
\item
For each $i \geq 0$, $\R_i = \R\la\bar\eps_1,\ldots,\bar\eps_i\ra$ where for $j=1,\ldots,i$,  $\bar{\eps}_j$ denotes the sequence $\eps_{j,1},\eps_{j,2},\ldots$.
\item
For each $\phi \in \mathcal{F}_{\R_i,k}$:
\begin{enumerate}[(a)]
\item
If $\RR(\phi, \overline{B_k(0,R)})$ is empty then, $\mathcal{I}_{i,k}(\phi) = -1$.
\item
\[
\left(\bigcup_{j \in [\mathcal{I}_{i,k}(\phi)]} \RR(\mathcal{C}_{i,k}(\phi)(j), \overline{B_k(0,R)})) \right) \searrow 
\left(\RR(\phi,  \overline{B_k(0,R)})\right)
\]
(see Notation~\ref{not:monotone}).
Notice that in the case $\RR(\phi, \overline{B_k(0,R)})$ is empty,  $\mathcal{I}_{i,k}(\phi)= -1$, hence 
$[\mathcal{I}_{i,k}(\phi)] = \emptyset$, and so  $\bigcup_{j \in [\mathcal{I}_{i,k}(\phi)]} \RR(\mathcal{C}_{i,k}(\phi)(j),\overline{B_k(0,R)})$ is an empty union, and is thus empty as well.
\item
For $j \in [\mathcal{I}_{i,k}(\phi)]$, $\RR(\mathcal{C}_{i,k}(\phi)(j), \overline{B_k(0,R)})) $ is homologically
$\ell$-connected.
\end{enumerate} 
\end{enumerate} 
\end{property}

\begin{notation}
\label{not:phi_t}
Let $\phi$ be a quantifier-free formula with coefficients in $\R[\bar\eps]$. Then $\phi$ is defined over $\R[\bar\eps_1',\bar\eps'_2,\ldots,\bar\eps_i']$ where 
$\bar\eps_j'$ is a finite sub-sequence of the sequence $\bar\eps_j$. For 
$\bar{t} = (\bar{t}_1,\ldots, \bar{t}_i)$, where for $1 \leq j \leq i$, $\bar{t}_j$ is  a tuple of elements
of $\mathbb{R}$ of the same length as $\bar\eps_j'$, we will denote by $\phi_{\bar{t}}$
the formula defined over $\mathbb{R}$ obtained by replacing $\bar\eps'_j$ by $\bar{t}_j$ in the formula $\phi$.
\end{notation}

For any finite sequence $\bar{t} = (t_1,\ldots,t_N)$, by the phrase ``for all sufficiently
small and positive $\bar{t}$'' we will mean `` for all sufficiently small $t_1 \in \mathbb{R}_{> 0}$, and having chosen $t_1$,  for all sufficiently small $t_2 \in \mathbb{R}_{> 0}$, ... ''.

We  will say that 
\[
\left((\R_i)_{i \geq 0},R,k, (\mathcal{I}_{i,k})_{i \geq 0}, (\mathcal{C}_{i,k})_{\geq 0} \right)
\]
satisfies  the \emph{$\ell$-connectivity property over $\R = \mathbb{R}$}  if it satisfies the following conditions.

\begin{manualproperty}{\ref*{property:thm:main}$^\prime$} 
\label{property:thm:main'}
\begin{enumerate}[1.]
\item
$\R_0 = \mathbb{R}$ and for each , $i > 0$, $\R_i = \R\la\bar\eps_1,\ldots,\bar\eps_i\ra$.
\item
For each $\phi \in \mathcal{F}_{\R_i,k}$:
\begin{enumerate}[(a)]
\item
If $\RR(\phi,\overline{B_k(0,R)})$ is empty then, $\mathcal{I}_{i,k}(\phi) = -1$.
\item
\[
\left(\bigcup_{j \in [\mathcal{I}_{i,k}(\phi)]} \RR(\mathcal{C}_{i,k}(\phi)(j), \overline{B_k(0,R)})) \right) \searrow 
\left(\RR(\phi, \overline{B_k(0,R)})\right)
\]

\item
For $j \in [\mathcal{I}_{i,k}(\phi)]$, 
and all sufficiently small and positive $\bar{t}$,
\[
\RR(\mathcal{C}_{i,k}(\phi)(j)_{\bar{t}}, \overline{B_k(0,R)}))
\] 
is $\ell$-connected.
\end{enumerate} 
\end{enumerate} 
\end{manualproperty}

The following two theorems give the important topological properties of the posets defined above
that will be useful for us.

\begin{theorem}
\label{thm:main}
Suppose that the tuple
\[
\left((\R_i)_{i \geq 0}, R, k, (\mathcal{I}_{i,k})_{i \geq 0}, (\mathcal{C}_{i,k})_{i \geq 0} \right)
\]
satisfies  the  homological $\ell$-connectivity property over $\R$
(see Property~\ref{property:thm:main}).
Then, for $-1 \leq m \leq  \ell$, every
finite set $J$, and 
$\Phi \in (\mathcal{F}_{k,\R_i})^J$,
such that for each $j \in J$, $\RR(\Phi(j),\overline{B_k(0,R)})$ is homologically $\ell$-connected, 
\begin{equation}
\label{eqn:thm:main}
|\Delta(\mathbf{P}_{m,i}(\Phi))| \overset{h}{\sim}_{m-1} \RR(\Phi,\overline{B_k(0,R)})^J.
\end{equation}

More generally, we have the diagrammatic homological $(m-1)$-equivalence
\begin{equation}
\label{eqn:thm:main:diagrammatic}
(J' \mapsto |\Delta(\mathbf{P}_{m,i}(\Phi|_{J'})|)_{J' \in 2^J}  \overset{h}{\sim}_{m-1} \Simp^J(\RR(\Phi,\overline{B_k(0,R)})).
\end{equation}
\end{theorem}

In the case $\R = \mathbb{R}$ we can derive a stronger conclusion from a stronger assumption.

\begin{manualtheorem}{\ref*{thm:main}$^\prime$}
\label{thm:main'}
Suppose that
\[
\left((\R_i)_{i \geq 0},R,k, (\mathcal{I}_{i,k})_{i \geq 0}, (\mathcal{C}_{i,k})_{i \geq 0} \right)
\]
satisfies the $\ell$-connectivity property over $\R = \mathbb{R}$ (cf.
Property~\ref{property:thm:main'}).

Then, 
for $-1 \leq m \leq \ell$, each finite set $J$,  and
$\Phi \in (\mathcal{F}_{\R,k})^J$,
such that for each $j \in J$, $\RR(\Phi(j),\overline{B_k(0,R)})$ is $\ell$-connected, 
\[
|\Delta(\mathbf{P}_{m,i}(\Phi))| \sim_{m-1} \RR(\Phi,\overline{B_k(0,R)})^J.
\]

More generally, we have the diagrammatic $(m-1)$-equivalence:
\begin{equation}
\label{eqn:thm:main':diagrammatic}
(J' \mapsto |\Delta(\mathbf{P}_{m,i}(\Phi|_{J'})|)_{J' \in 2^J}  \sim_{m-1} \Simp^J(\RR(\Phi,\overline{B_k(0,R)})).
\end{equation}
\end{manualtheorem}

Before proving Theorems~\ref{thm:main} and \ref{thm:main'} we discuss an example.

\subsection{Example of the sphere $\Sphere^2$ in $\R^3$}	
In order to illustrate the main ideas behind the definition of the poset, $\mathbf{P}_{m,i}(\Phi)$, defined above we discuss a very simple example. Starting from a cover of the two dimensional unit sphere in $\mathbb{R}^3$ by two closed hemispheres, 
we show how we construct the associated poset. We will assume that there is an algorithm available as a black-box which given any 
closed formula $\phi$ such that $\RR(\phi)$ is bounded,  
produces a tuple of quantifier-free closed formulas as output, such that 

\begin{enumerate}[(a)]
	\item
	the realization of each formula
	in the tuple is contractible;
	\item
	\label{itemlabel:examble:b}
	the union of the realizations is a semi-algebraic set infinitesimally larger than $\RR(\phi) $, and such that $\RR(\phi)$ is a semi-algebraic deformation retract of the union.
\end{enumerate}

\begin{figure}[H]
	\centering
	\subfigure[]{%
		\label{Fig:f1}%
		\includegraphics[height=3.25in]{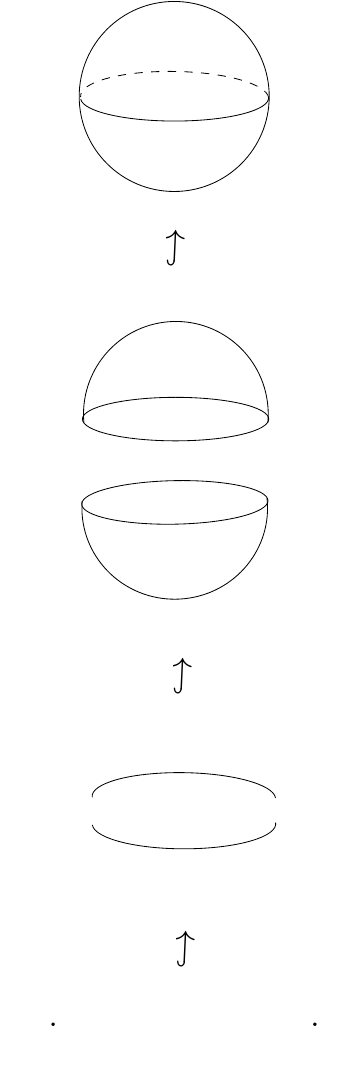}}%
	\qquad
	\subfigure[]{%
		\label{Fig:f2}%
		\includegraphics[height=3.25in]{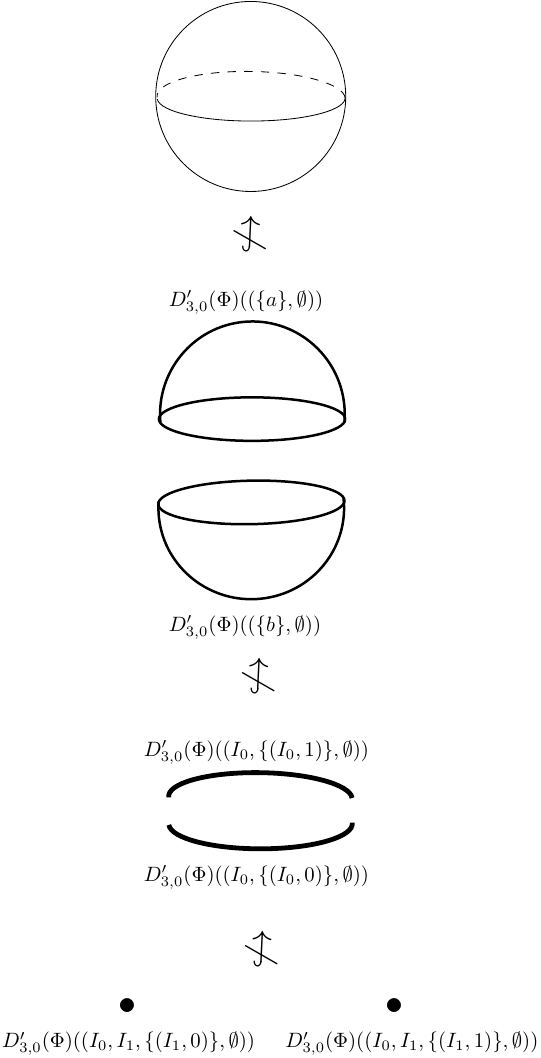}}%
	\qquad
	\subfigure[]{%
		\label{Fig:f3}%
		\includegraphics[height=3.25in]{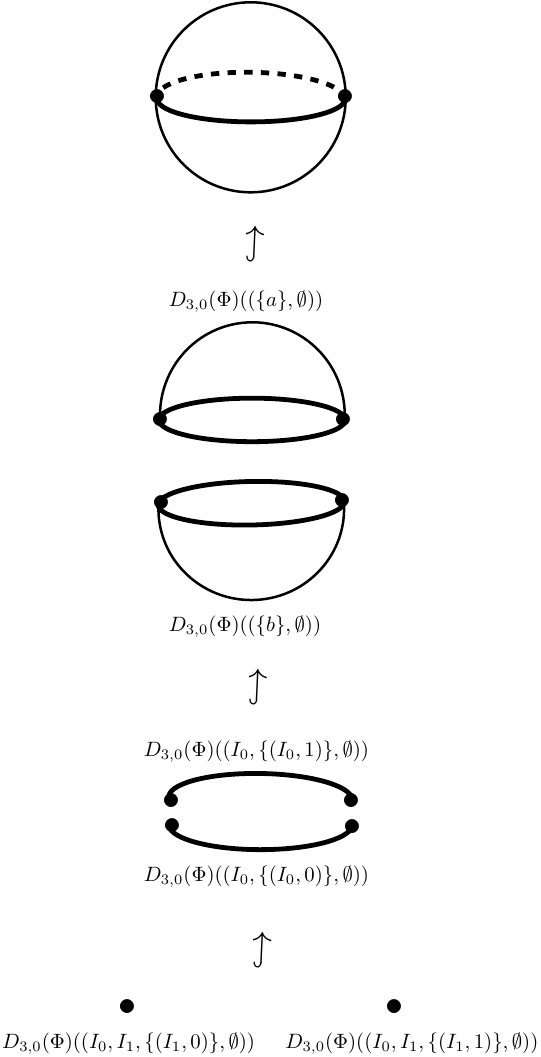}}%
	\caption{\small (a) The ideal situation, (b) $D_{m,i}'(\Phi)(.)$, and (c) $D_{m,i}(\Phi)(.)$ }\label{Fig:Func}
\end{figure}

Therefore, at each step of our construction the cover by contractible sets that we consider, is actually a cover of a semi-algebraic set which is infinitesimally larger than that but with the same homotopy type as the original set. 
As a result, the \emph{inclusion property}  -- namely, that each element of the cover is included in the set that it is part of a cover of -- which is expected from the elements of a cover  will not hold. 

We first describe the situation in the case when Part~\eqref{itemlabel:examble:b} above is replaced with: \\

\noindent
(b$^\prime$) 
the union of the realizations is equal to  $\RR(\phi)$. \\

We call this the \emph{ideal situation}.
Figure~\ref{Fig:f1} displays three levels of the construction in the ideal situation for the sphere. 
In the first step, we have two closed contractible hemispheres that cover the whole sphere. The intersection of the two hemispheres is a circle, and the next level shows the two closed semi-circles as its cover. The bottom level consists of  two points which is  the intersection of these semi-circles. Clearly, the inclusion property holds in this case. 

Unfortunately, as mentioned before we cannot assume that we are in the ideal situation. This is because
the only algorithm with a singly exponential complexity that is currently known for computing covers by contractible sets, satisfies 
Property \eqref{itemlabel:examble:b}  rather than the ideal Property 
(b$^\prime$).   
In the non-ideal situation we will obtain in the first step a cover of an infinitesimally thickened sphere
by two thickened hemispheres where the thickening is in terms of some infinitesimal $\eps_0, 0 < \eps_0 \ll 1$. The intersection of these two thickened hemispheres is a thickened circle, and which is 
covered by two thickened semi-circles whose union is infinitesimally larger than the thickened circle.
The new infinitesimal is $\eps_1$ and $0 < \eps_1 \ll \eps_0 \ll 1$.
Finally, in the next level,  the intersection of the two thickened semi-circles is covered by
two thickened points involving a third infinitesimal $\eps_2$, such that $0 < \eps_2 \ll \eps_1 \ll \eps_0  \ll 1$.

We associate to each element $\alpha \in \mathbf{P}_{m,i}(\Phi)$ two semi-algebraic sets
$D_{m,i}(\Phi)(\alpha)$, $D'_{m,i}(\Phi)(\alpha)$.
The association $D_{m,i}(\Phi)(\cdot)$ is functorial in the sense that if $\alpha, \beta \in \mathbf{P}_{m,i}(\Phi)$, then
$\alpha \prec \beta  \Leftrightarrow D_{m,i}(\Phi)(\alpha) \subset D_{m,i}(\Phi)(\beta)$. This
functoriality is important since it allows us to define the homotopy colimit of the functor $D_{m,i}(\Phi)$. The association $\alpha \mapsto D'_{m,i}(\Phi)(\alpha)$ does not have the functorial property. However, it follows directly from its definition that $D'_{m,i}(\Phi)$ is contractible (or  $\ell$-connected in the more general setting). Finally, we are able to show that $D'_{m,i}(\Phi)(\alpha)$
is homotopy equivalent to $D_{m,i}(\Phi)(\alpha)$ for each $\alpha \in \mathbf{P}_{m,i}(\Phi)$, and 
thus the functor $D_{m,i}(\Phi)$ has the advantage of being functorial as well as satisfying the 
connectivity property. 

In this example, we display $D_{m,i}'(\Phi)(\alpha)$ and $D_{m,i}(\Phi)(\alpha)$ for all different
$\alpha \in \mathbf{P}_{m,i}(\Phi)$ in Figures~\ref{Fig:f2} and \ref{Fig:f3}.

For the rest of this example we assume the covers of sphere are in the ideal situation. This assumption will
not change the poset $\mathbf{P}_{m,i}(\Phi)$ that we construct.

In order to reconcile with the notation used in the definition of the poset
$\mathbf{P}_{m,i}(\Phi)$,
we will assume that the different covers described above (which are not Leray but $\infty$-connected)
correspond to the values of the maps $\mathcal{I}_{i,3}$ and $\mathcal{C}_{i,3}$ evaluated at the corresponding formulas
which we describe more precisely below.\\

\begin{enumerate}[Step 1.]
	\item \label{step:p3}
	Let $a,b$ denote the closed upper and lower hemispheres of the sphere  $\Sphere^2(0,1) \subset \mathbb{R}^3$, defined by formulas 
	\begin{eqnarray*}
		\phi_a &:=&  (X_1^2 + X_2^2 + X_3^2 - 1 = 0) \wedge (X_3 \geq 0), \\
		\phi_b &:=&  (X_1^2 + X_2^2 + X_3^2 - 1 = 0) \wedge (X_3 \leq 0).
	\end{eqnarray*}
	
	Let $J = J_0 =  \{a,b\}$, and	$\Phi \in \mathcal{F}_{\mathbb{R},3}^J$ be defined by $\Phi(a) = \phi_a, \Phi(b) = \phi_b$. Moreover, since $\card(J)=2$,
	
	{\small
		\[
	\mathbf{P}_{3,0}(\Phi) = \{(\{a\}, \emptyset),(\{b\},\emptyset)\} \cup  \bigcup_{I_0 \subset J,  \card(I_0)=2} \{I_0\}\times \mathbf{P}_{2,1}(\Phi_{3,0,I_0,J_0}).
	\]
	}
	
	Following the notation used in Definition~\ref{def:poset},
	let $I_0 = J_0 = J  = \{a,b\}$. \\
	
	\item \label{step:p2}
	We suppose that $\mathcal{I}_{0,3}(\phi_a \wedge \phi_b) = 1$, and 
	$\mathcal{C}_{0,3}(\phi_a \wedge \phi_b)(0) = \phi_c$, $\mathcal{C}_{0,3}(\phi_a \wedge \phi_b)(1) = \phi_d$,
	where 
	\begin{eqnarray*}
		\phi_c &:=&  (X_1^2 + X_2^2 + X_3^2 - 1 = 0) \wedge (X_3 = 0) \wedge (X_2 \geq 0), \\
		\phi_d&:=&  (X_1^2 + X_2^2 + X_3^2 - 1 = 0) \wedge (X_3 = 0)  \wedge (X_2 \leq 0),
	\end{eqnarray*}
	denote the two semi-circles. 
	
	\begin{gather*}
		J_1 = J_{3,0,I_0,\Phi} = \{I_0\}  \times [1] 
		= \{ (I_0,0), (I_0,1)\}, \\
		\Phi_1 = \Phi_{3,0,I_0,J_0}, \\
		\Phi_{1}((I_0,0)) 
		= \phi_c, \\
		\Phi_{1}((I_0,1)) 
		= \phi_d.\\
	\end{gather*}

	{\small
		\[
		\mathbf{P}_{2,1}(\Phi_{1}) = \{(\{(I_0, 0)\}, \emptyset), (\{(I_0, 1)\}, \emptyset)\}  \cup \bigcup_{I_1 \subset J_1,  \card(I_1)=2} \{I_1 \}\times \mathbf{P}_{1,2}((\Phi_1)_{2,1,I_1,J_1}).
		\]
	}
	Now let $I_1 = J_1$. \\
	
	\item \label{step:p1_1}
	Suppose that
	$\mathcal{I}_{1,3}(\phi_c \wedge \phi_d) = 1$, and 
	$\mathcal{C}_{1,3}(\phi_c \wedge \phi_d)(0) = \phi_e$,\\ $\mathcal{C}_{1,3}(\phi_c \wedge \phi_d)(1) = \phi_f$,
	where 
	\begin{eqnarray*}
		\phi_e &:=&  (X_1^2 + X_2^2 + X_3^2 - 1 = 0) \wedge (X_3 = 0) \wedge (X_2 = 0) \wedge (X_1 + 1 = 0), \\
		\phi_f&:=&  (X_1^2 + X_2^2 + X_3^2 - 1 = 0) \wedge (X_3 = 0)  \wedge (X_2 = 0)  \wedge (X_1 - 1 = 0).
	\end{eqnarray*}
	
	\begin{gather*}
		J_2 = (J_1)_{2,1,I_1,\Phi_1} = \{I_1\}  \times [1] 
		= \{ (I_1,0), (I_1,1)\}, \\
		\Phi_2 = (\Phi_1)_{2,1,I_1,J_1} \\
		\Phi_{2}((I_1,0)) = \phi_e, \\
		\Phi_{2}((I_1,1)) = \phi_f. \\
	\end{gather*}
	
	{\small
		\[
		\mathbf{P}_{1,2}(\Phi_{2}) = \{(\{(I_1, 0)\}, \emptyset), (\{(I_1, 1)\}, \emptyset)\}  \cup \bigcup_{I_2 \subset 
		J_2,  \card(I_2)=2} \{I_2 \}\times \mathbf{P}_{0, 3}((\Phi_2)_{1,2,I_2,J_2}).
		\]
	}
	
	 Let $I_2 = J_2$. \\

	\item \label{step:p1}
		Since $\mathcal{I}_{2,3}(\phi_e \wedge \phi_f) = -1$, hence $\mathbf{P}_{0,3}((\Phi_2)_{1,2,I_2,J_2}) = \emptyset$, and from Step~\ref{step:p1_1}
		
				{\small
			\[
			\mathbf{P}_{1,2}(\Phi_{2}) = \{(\{(I_1, 0)\}, \emptyset), (\{(I_1, 1)\}, \emptyset)\}.
			\]
		}

	\item \label{step:p1complete}
	With these choices of the values of $\mathcal{I}_{\cdot,3}$ and $\mathcal{C}_{\cdot,3}$ for the specific formulas described above, and $\ell = \infty$, from Step~\ref{step:p2} and Step~\ref{step:p1}, the Hasse diagram of the poset $\mathbf{P}_{2,1}(\Phi_{1})$ is as follows.
		
		\[
		\xymatrix{
			( \{(I_0,0)\}, \emptyset) &&  ( \{(I_0,1)\}, \emptyset) \\
			(I_1, \{(I_1,0)\},\emptyset) \ar[u]\ar[rru] && (I_1, \{(I_1,1)\},\emptyset) \ar[u]\ar[llu] 
		}
		\]
		\\
	\item Finally, from Step~\ref{step:p3} and Step~\ref{step:p1complete}, the Hasse diagram of the poset  $\mathbf{P}_{3,0}(\Phi)$ is shown  below. 	
		
		\[
		\xymatrix{
			(\{a\},\emptyset) && (\{b\},\emptyset) \\
			(I_0,  \{(I_0,0)\}, \emptyset) \ar[u]\ar[rru] &&  (I_0,  \{(I_0,1)\}, \emptyset) \ar[u]\ar[llu] \\
			(I_0,I_1, \{(I_1,0)\},\emptyset) \ar[u]\ar[rru] && (I_0,I_1, \{(I_1,1)\},\emptyset) \ar[u]\ar[llu] 
		}
		\]
		\\
\end{enumerate}	
	
The order complex, $\Delta(\mathbf{P}_{3,0}(\Phi))$ is displayed below and 
clearly $|\Delta(\mathbf{P}_{3,0}(\Phi))|$ is homeomorphic to $\Sphere^2(0,1)$.

\begin{figure}[h!]
	\begin{center}
	    \label{fig:sphere}
	    \includegraphics[scale=0.75]{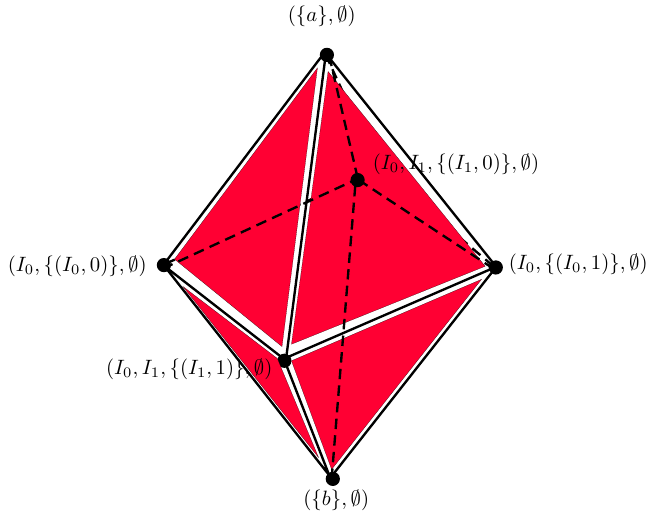}%
	\caption{\small The order complex, $\Delta(\mathbf{P}_{3,0}(\Phi))$}
	\end{center}
	\end{figure}

\subsection{Proofs of Theorems~\ref{thm:main} and \ref{thm:main'}}
In this section we prove Theorem~\ref{thm:main} as well as Theorem~\ref{thm:main'}.
 We first give an outline of the proof of Theorem~\ref{thm:main}.

\subsubsection{Outline of the proof of Theorem~\ref{thm:main}}
In order to prove that $|\Delta(\mathbf{P}_{m,i}(\Phi))|$ is 
homologically $(m-1)$-equivalent to $\RR(\Phi)^J$, we give two homological $(m-1)$-equivalences.
The source of both these maps is a semi-algebraic set which is defined as the homotopy colimit
of a certain functor $D_{m,i}$ from the poset category $\mathbf{P}_{m,i}(\Phi)$ to $\Top$
taking its values in semi-algebraic subsets of $\R^k_{i+m+1}$.
The targets are $|\Delta(\mathbf{P}_{m,i}(\Phi))|$ and
$\RR(\Phi)^J$.
Taken together these two homological $(m-1)$-equivalences imply that
$|\Delta(\mathbf{P}_{m,i}(\Phi))|$ and
$\RR(\Phi)^J$ are homologically $(m-1)$-equivalent.

In what follows, we first define the functor
$D_{m,i}$ as well as an associated map $D'_{m,i}$, also taking values in semi-algebraic
sets, and prove the main properties of these objects that we are going to need in the proof
of Theorem~\ref{thm:main}. 

\subsubsection{Definition of $D_{m,i},D'_{m,i}$}
We now define for each $\alpha = (I_0,\ldots,I_{r},\emptyset) \in \mathbf{P}_{m,i}(\Phi)$, a closed semi-algebraic subset 
$D_{m,i}(\alpha) \subset \overline{B_k(0,R)} \subset \R_{i+m+ 1}^k$, and also  
a semi-algebraic set 
$D'_{m,i}(\alpha) \subset \R_{i+r}^k$.

We define $D_{m,i},D'_{m,i}$ by induction on $m$.
For $m=-1$, we define for  $j \in J$, 
\[
D_{-1,i}(\Phi)((\{j\},\emptyset)) = D'_{-1,i}(\Phi)((\{j\},\emptyset)) = \RR(\Phi(j),\overline{B_k(0,R)}) \subset \R_{i}^k.
\]

We now define  $D_{m,i}(\Phi), D'_{m,i}(\Phi): \mathbf{P}_{m,i}(\Phi) \rightarrow \Top $, using the  fact that 
they are already defined for all $-1 \leq m' < m$. We define:
\begin{eqnarray*}
D_{m,i}(\Phi)((\{j\},\emptyset)) &=&\E(\RR(\Phi(j),\overline{B_k(0,R)}), \R_{i+m+1}) \cup  \\
&&\bigcup_{(I,\alpha) \in \mathbf{P}_{m,i}(\Phi), j \in I}  \E(D_{m-\card(I)+1,i+1}(\Phi_{m,i,I,J})(\alpha),\R_{i+m+1}),
 \\
D_{m,i}(\Phi)((I,\alpha)) &=& \E(D_{m - \card(I) +1, i+1}(\Phi_{m,i,I,J})(\alpha),\R_{i+m+1}), \\
&&  \mbox{ 
$I \subset_{\leq m+2} J, \card(I)> 1$, $\alpha \in \mathbf{P}_{m-\card(I)+1,i+1}(\Phi_{m,i,I,J})$},
\end{eqnarray*}

\begin{eqnarray}
\label{eqn:def:D':0}
D'_{m,i}(\Phi)((\{j\},\emptyset)) &=&
\RR(\Phi(j), \overline{B_k(0,R)}),
\end{eqnarray}
and

\begin{eqnarray*}
D'_{m,i}(\Phi)((I,\alpha)) &=& 
D_{m - \card(I) +1,i+1}'(\Phi_{m,i,I,J})(\alpha),
\end{eqnarray*}
for 
$I \subset_{\leq m+2} J, \card(I)> 1$, $\alpha \in \mathbf{P}_{m-\card(I)+1,i+1}(\Phi_{m,i,I,J})$.

The following lemma is obvious from the definition of 
$D_{m,i}(\alpha)$ given above.
\begin{lemma}
\label{lem:thm:main:proof:claim:0}
For each $\alpha, \beta \in \mathbf{P}_{m,i}(\Phi)$ with $\alpha \prec \beta$,
the morphism $D_{m,i}(\Phi)(\alpha \prec \beta): D_{m,i}(\Phi)(\alpha) \rightarrow D_{m,i}(\Phi)(\beta)$ is an inclusion.
So, 
$D_{m,i}(\Phi)$ is a 
functor  from the poset category $(\mathbf{P}_{m,i}(\Phi), \prec)$ to $\Top$.
\end{lemma}

\begin{remark}
Unlike $D_{m,i}$, $D'_{m,i}$ is not necessarily a functor.
\end{remark}

\begin{lemma}
\label{lem:thm:main:proof:claim:1}
For each $\alpha \in \mathbf{P}_{m,i}(\Phi)$, 
\[
D_{m,i}(\Phi)(\alpha) \searrow D'_{m,i}(\Phi)(\alpha).
\]

\end{lemma}

\begin{proof}
Let 
\[
\alpha = (I_0^\alpha,\ldots, I_{r_\alpha}^\alpha = \{j_\alpha\}, \emptyset)
\] 
with $I_h^\alpha \subset J_h^\alpha, 0 \leq h \leq r_\alpha$ following the same notation
as in Section~\ref{subsubsec:tuple-of-sets} (with an added superscript ${}^\alpha$).

First observe that
\begin{equation}
\label{eqn:thm:main:proof:claim:1.1:1}
D_{m,i}(\Phi)(\alpha) = \E(D'_{m,i}(\Phi)(\alpha), \R_{i+m+1}) \cup \bigcup_{\beta \precneq \alpha} D_{m,i}(\Phi)(\beta). 
\end{equation}

We now prove 
that for each $\alpha \in \mathbf{P}_{m,i}(\Phi)$:
\begin{equation}
\label{eqn:thm:main:proof:claim:1.1:3}
D_{m,i}(\Phi)(\alpha) \searrow D'_{m,i}(\Phi)(\alpha),
\end{equation}
and
\begin{equation}
\label{eqn:thm:main:proof:claim:1.1:4}
\bigcup_{\beta \precneq \alpha} D_{m,i}(\Phi)(\beta) \searrow  \bigcup_{\beta \precneq \alpha} \lim_{\bar\eps_{i+r_\alpha+1}} D_{m,i}'(\Phi)(\beta) \subset D'_{m,i}(\Phi)(\alpha).
\end{equation}

The proof is by induction on the maximum length, $\length(\alpha)$, of any chain
with $\alpha$ as the maximal element.\\

We first note that if $\R' = \R\la\bar\eps\ra$, and $X \subset \R^k$ is
a semi-algebraic subset, then 
\[
\lim_{\bar\eps} \E(X,\R') = \overline{X}.
\]
This follows easily from the definition of $\E(X,\R')$ and 
standard properties of $\lim_{\bar\eps}$.
In particular, if $X$ is a closed semi-algebraic set, then 
\[
\lim_{\bar\eps} \E(X,\R') = X.
\]

\noindent{Base case of the induction,  $\length(\alpha) = 1$:}
It follows from \eqref{eqn:thm:main:proof:claim:1.1:1} and the fact that
that $D'_{m,i}(\Phi)(\alpha)$ is a closed semi-algebraic set, that
\eqref{eqn:thm:main:proof:claim:1.1:3} holds if $\alpha$ is 
a minimal element of the poset $\mathbf{P}_{m,i}(\Phi)$ (and so
$\length(\alpha) = 1$). In this case
\eqref{eqn:thm:main:proof:claim:1.1:4} is trivially true.\\

\noindent{Induction hypothesis:}
We assume now that \eqref{eqn:thm:main:proof:claim:1.1:3} and 
\eqref{eqn:thm:main:proof:claim:1.1:4} is true for all 
$\alpha \in \mathbf{P}_{m,i}(\Phi)$, with $\length(\alpha) < t$.\\

\noindent{Inductive step:}
Suppose that $\alpha \in \mathbf{P}_{m,i}(\Phi)$, with $\length(\alpha) = t$.
The inductive hypothesis implies that \eqref{eqn:thm:main:proof:claim:1.1:3} and \eqref{eqn:thm:main:proof:claim:1.1:4}
both hold with $\alpha$ replaced by $\alpha'$ for 
all $\alpha'  \precneq \alpha$.

Using the fact that $D'_{m,i}(\Phi)(\alpha)$ is closed,
it is easy to check that 
\eqref{eqn:thm:main:proof:claim:1.1:4} implies \eqref{eqn:thm:main:proof:claim:1.1:3}.
So we need to prove only \eqref{eqn:thm:main:proof:claim:1.1:4}.
Using the induction hypothesis we have 
for each $\beta \precneq \alpha$
\begin{equation}
\label{eqn:thm:main:proof:claim:1.1:5}
\bigcup_{\beta \precneq \alpha} D_{m,i}(\Phi)(\beta) \searrow \bigcup_{\beta \precneq \alpha} D_{m,i}'(\Phi)(\beta).
\end{equation}

Now observe that  for any $\beta \in \mathbf{P}_{m,i}(\Phi)$, 
$\beta \precneq \alpha$ if and only if there exist $j_\alpha' \in I_{r_\alpha -1}^\alpha, j_\alpha' \neq j_\alpha$ and $j_\alpha'' \in (J_{r_\alpha}^\alpha)_{m^\alpha_{r_\alpha}, i^\alpha_{r_\alpha},\{j_\alpha,j_\alpha'\},\Phi_{r_\alpha}}$, such that 
\[
\beta \prec \gamma(j_\alpha'')= (I_0^\alpha,\ldots,I_{r_\alpha-1}^\alpha,\{j_\alpha,j_\alpha'\},\{j_\alpha''\},\emptyset),
\]
where we assume that $I_{-1}^\alpha = J$.

Using the above observation we have that
\begin{equation}
\label{eqn:thm:main:proof:claim:1.1:6}
\bigcup_{\beta \precneq \alpha} D_{m,i}'(\Phi)(\beta) = \bigcup_{j_\alpha' \in I^\alpha_{r_\alpha -1}, j_\alpha' \neq j_\alpha} 
 \bigcup_{j_\alpha'' \in  (J_{r_\alpha}^\alpha)_{m^\alpha_{r_\alpha}, i^\alpha_{r_\alpha},\{j_\alpha,j_\alpha'\},\Phi_{r_\alpha}}} \left( \bigcup_{\beta \prec \gamma(j_\alpha'')} D'_{m,i}(\Phi)(\beta) \right),
\end{equation}
where 
\[
\gamma(j_\alpha'')= (I_0^\alpha,\ldots,I_{r_\alpha-1}^\alpha,\{j_\alpha,j_\alpha'\},\{j_\alpha''\},\emptyset).
\]

Applying hypothesis \eqref{eqn:thm:main:proof:claim:1.1:4} we have that 

\begin{equation}
\label{eqn:thm:main:proof:claim:1.1:7}
\left( 
\bigcup_{\beta \precneq \gamma(j_\alpha'')} D'_{m,i}(\Phi)(\beta)
\right)  \searrow \lim_{\bar\eps_{i+r+2}} \bigcup_{\beta \precneq \gamma(j_\alpha'')} D'_{m,i}(\Phi)(\beta) \subset  D'_{m,i}(\Phi)(\gamma(j_\alpha'')).
\end{equation}
Also observe that, 
\begin{equation}
\label{eqn:thm:main:proof:claim:1.1:8}
\left(  \bigcup_{j_\alpha'' \in J_{m,i,\{j_\alpha,j_\alpha'\},\Phi}} D'_{m,i}(\Phi)(\gamma(j_\alpha'')) \right)  \searrow
\left(D'_{m,i}(\Phi)(\alpha) \cap D'_{m,i}(\Phi)(\alpha') \right) \subset D'_{m,i}(\Phi)(\alpha),
\end{equation}
where 
\[
\alpha' = (I_0^\alpha,\ldots, I_{r_\alpha-1}^\alpha, \{j_\alpha'\}, \emptyset).
\]

Finally,
\eqref{eqn:thm:main:proof:claim:1.1:4} now follows from
\eqref{eqn:thm:main:proof:claim:1.1:5}, \eqref{eqn:thm:main:proof:claim:1.1:6},
\eqref{eqn:thm:main:proof:claim:1.1:7},  and \eqref{eqn:thm:main:proof:claim:1.1:8}.

\end{proof}

\begin{lemma}
\label{lem:thm:main:proof:claim:2}
\[
\left( \bigcup_{\alpha \in \mathbf{P}_{m,i}(\Phi)} D_{m,i}(\Phi)(\alpha) \right) \searrow  \RR(\Phi, \overline{B_k(0,R)})^J.
\]
In particular,
$\E(\RR(\Phi, \overline{B_k(0,R)})^J, \R_i)$ is a semi-algebraic deformation
retract of $\bigcup_{\alpha \in \mathbf{P}_{m,i}(\Phi)} D_{m,i}(\Phi)(\alpha) $.
\end{lemma}

\begin{proof}
First note that for each $j \in J$, $(\{j\},\emptyset)$ is a maximal
element of the poset $\mathbf{P}_{m,i}(\Phi)$. It now follows from 
Lemma~\ref{lem:thm:main:proof:claim:0} that
\[
\bigcup_{\alpha \in \mathbf{P}_{m,i}(\Phi)} D_{m,i}(\Phi)(\alpha) 
= \bigcup_{j \in J} D_{m,i}(\Phi)((\{j\},\emptyset)). 
\]
The lemma now follows from Lemma~\ref{lem:thm:main:proof:claim:1} and Eqn.\eqref{eqn:def:D':0}.
\end{proof}

\begin{notation}
\label{not:retraction}
We will denote the deformation retraction
\[
\bigcup_{\alpha \in \mathbf{P}_{m,i}(\Phi)} D_{m,i}(\Phi)(\alpha) \rightarrow \E(\RR(\Phi, \overline{B_k(0,R)})^J, \R_i)
\]
in Lemma~\ref{lem:thm:main:proof:claim:2}
by $r_{m,i}(\Phi)$.
\end{notation}

In the proof of Theorem~\ref{thm:main} we need the notion of the \emph{homotopy colimit}
of a functor which we define below. 

We fix a real closed field $\R$ in the following definition.
\begin{definition}[The topological standard $n$-simplex]
We denote by
\[
|\Delta^n | = \{ (t_0, \ldots, t_n) \in \R^{n+1}_{\geq 0} | \sum_{i=0}^n t_i = 1 \}
\]
the 
standard $n$-simplex defined over $\R$.	
For $0 \leq i \leq n$, we define the 
face operators, 
\[
d^i_n: |\Delta^{n-1}| \rightarrow |\Delta^{n}|,
\] 
by
\[
d^i_n(t_0,\dots,t_{n-1}) = (t_0,\dots,t_{i-1}, 0, t_{i}, \dots,t_{n-1}).
\]
\end{definition}

	\begin{definition}[Homotopy colimit]
		\label{def:hocolim}
		Let $(\mathbf{P},  \prec)$ be a poset category and 
		\[
		D:(\mathbf{P}, \prec) \rightarrow \Top
		\] 
		a functor taking its values in closed and bounded semi-algebraic subsets of $\R^k$,  and such that the morphisms $D(\alpha \prec \beta)$ are 
		inclusion maps.
		The homotopy colimit of $D$ is the quotient space
		\footnote{which is a semi-algebraic set defined over $\R$,  being a quotient space of a proper semi-algebraic equivalence relation, (see for example \cite[page 166]{Dries})}
		\[
		\hocolim(D) = \left(\coprod_{\alpha_0 \precneq \cdots \precneq \alpha_p} |\Delta^p|  \times D(\alpha_0) \right) \Big/ \hspace*{-.1cm} \sim \ ,
		\]
		where the equivalence relation $\sim$ is defined as follows.
		
		For a chain $\sigma = (\alpha_0 \precneq \cdots \precneq \alpha_p)$,
		$t \in |\Delta^p|$, and $x \in D(\alpha_0)$, 
		we denote by $(t,x)_\sigma$, the
		image of $(t,x)$ under the canonical inclusion of $|\Delta^p|  \times D(\alpha_0)$ (corresponding to the chain $\sigma$)
		in the disjoint union
		$\coprod_{\alpha_0 \precneq \cdots \precneq \alpha_p} |\Delta^p|  \times D(\alpha_0)$.
		
		Using the above notation the equivalence relation $\sim$ is defined by:
		\vspace*{0.15cm}
		\begin{equation} 
		\label{eqn:def:hocolimit}
		 (d^i_p (t), x)_\sigma \sim (t, x)_{\sigma'}, 
		\end{equation}
		for $x \in D(\alpha_0)$ and $t \in  |\Delta^{p-1}|$, where $\sigma = (\alpha_0 \precneq \cdots \precneq \alpha_p)$ and
		\vspace*{0.25cm}
		$$
		\sigma' = 
		\left\{
 \begin{array}{ll} 
	(\alpha_1  \precneq \cdots \precneq \alpha_p)  &   \text{ if } i =0, \\ 
	(\alpha_0 \precneq \cdots \alpha_{i-1} \precneq \alpha_{i+1} \precneq \cdots \precneq \alpha_p) &  \text{ if } 0 < i < p, \\ 
	(\alpha_0  \precneq \cdots \precneq \alpha_{p-1}) &  \text{ if }  i = p.
\end{array} \right.
		$$
	
We denote by $\pi^D_1: \hocolim(D) \rightarrow |\Delta(\mathbf{P})|$, 
		$\pi^D_2: \hocolim(D) \rightarrow \colim(D)$ the canonical maps, 
		where $|\Delta(\mathbf{P})|$ is the geometric realization of the order complex of $\mathbf{P}$ (see Definition~\ref{def:order-complex}).
More precisely, $\pi^D_1$ is the map induced from the projection map
\[
\coprod_{\alpha_0 \precneq \cdots \precneq \alpha_p} |\Delta^p|  \times D(\alpha_0) \rightarrow
\coprod_{\alpha_0 \precneq \cdots \precneq \alpha_p} |\Delta^p| 
\]
after taking the quotient by $\sim$,
and 	
$\pi^D_2$ is  the composition of the map induced by the projection
\[
\coprod_{\alpha_0 \precneq \cdots \precneq \alpha_p} |\Delta^p|  \times D(\alpha_0) \rightarrow
\coprod_{\alpha_0 \precneq \cdots \precneq \alpha_p} D(\alpha_0),
\]
and the canonical map to the colimit of the functor $D$, 
which in this case is simply the union $\bigcup_{\alpha \in \mathbf{P}} D(\alpha)$. 
\end{definition}

The following example illustrates the definition given above. 
\begin{example}\label{exp:hocolimt}
Consider the poset $\mathbf{P} = \{a,b,c\}$, with three elements
with $c \precneq a, c \precneq b$ as the only non-trivial ordering relation
(Hasse diagram shown below).

\[
		\xymatrix{
			a && b \\
			& c \ar[ru]\ar[lu]& 
		}
\]

Let $D:\mathbf{P} \rightarrow \Top$ be the functor, with
\begin{eqnarray*}
D(a) &=& \RR((X_1^2 + X_2^2 - 4 =0) \wedge (X_2 \geq 0)), \\
D(b) &=& \RR((X_1^2 + X_2^2 - 4 =0) \wedge (X_2 \leq  0)), \\
D(c) &=& D(a) \cap D(b) \\
&=& \{(-2,0), (2,0)\}.
\end{eqnarray*}

The homotopy colimit of the functor $D$ 
is then the quotient of the disjoint union of the spaces
$$\displaylines{
\Delta^0 \times D(a), \Delta^0 \times D(b), \Delta^0 \times D(c), \cr
\Delta^1 \times D(c), \Delta^1 \times D(c)
}
$$
corresponding to the chains $(a), (b), (c), (c \precneq a), (c \precneq b)$
by the equivalence relation defined in Eqn. \eqref{eqn:def:hocolimit}. The non-trivial identifications induced by the quotienting
are given by (following the notation introduced in Definition~\ref{def:hocolim})
$$
\displaylines{
((0,1), (-2,0))_{(c \precneq a)} \sim   ((1), (-2,0))_{(c)}, \cr
((0,1), (2,0))_{(c \precneq a)} \sim  ((1), (2,0))_{(c)}, \cr
((1,0), (-2,0))_{(c \precneq a)} \sim   ((1), (-2,0))_{(a)}, \cr
((1,0), (2,0))_{(c \precneq a)} \sim  ((1), (2,0))_{(a)}, \cr
((0,1), (-2,0))_{(c \precneq b)} \sim   ((1), (-2,0))_{(c)}, \cr
((0,1), (2,0))_{(c \precneq b)} \sim  ((1), (2,0))_{(c)}, \cr
((1,0), (-2,0))_{(c \precneq b)} \sim   ((1), (-2,0))_{(b)}, \cr
((1,0), (2,0))_{(c \precneq b)} \sim  ((1), (2,0))_{(b)}.
}
$$
The quotient space (as a semi-algebraic set) is shown below in Figure~\ref{fig:hocolimit}.
\end{example}

\begin{figure}[H]
	\begin{center}
		\includegraphics[scale=0.75]{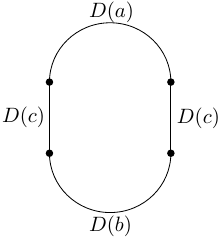}%
		\caption{\small Homotopy colimit of the functor $D$ in Example~\ref{exp:hocolimt}.}
		\label{fig:hocolimit}
	\end{center}
\end{figure}

\begin{proof}[Proof of Theorem~\ref{thm:main}]
The theorem will follow from the following two claims.

\begin{claim}
\label{thm:main:proof:claim:1}
The map $\pi_1^{D_{m,i}(\Phi)}:\hocolim(D_{m,i}(\Phi)) \rightarrow |\Delta(\mathbf{P}_{m,i}(\Phi))| $ is a homological $\ell$-equivalence (and so a homological $(m-1)$-equivalence).
\end{claim}

\begin{claim}
\label{thm:main:proof:claim:2}
The map 
\[
F_{m,i}(\Phi) = r_{m,i}(\Phi) \circ \pi_2^{D_{m,i}(\Phi)}:\hocolim(D_{m,i}(\Phi)) 
\rightarrow 
\E(\RR(\Phi, \overline{B_k(0,R)})^J, \R_i)
\]
is a homological $(m-1)$-equivalence.
\end{claim}

We first deduce the proof of the theorem from these two claims.
The homological $(m-1)$-equivalence in \eqref{eqn:thm:main} now follows from 
Claims~\ref{thm:main:proof:claim:1}, \ref{thm:main:proof:claim:2} and Lemma~\ref{lem:thm:main:proof:claim:2}.

The diagrammatic homological $(m-1)$-equivalence in \eqref{eqn:thm:main:diagrammatic} follows
from the commutativity of the following diagrams of maps.

For each pair $J',J'' \subset J$, with $J' \subset J''$ we have the following commutative diagram,
where the vertical arrows are inclusions, and the slanted arrows induce isomorphisms in
the homology groups up to dimension $m-1$.

\[
\xymatrix{
& \hocolim(D_{m,i}(\Phi|_{J'})) \ar[ld]^{\pi_1^{D_{m,i}(\Phi|_{J'})}} \ar[rd]_{
F_{m,i}(\Phi|_{J'})
}\ar[dd]&\\
|\Delta(\mathbf{P}_{m,i}(\Phi|_{J'}))| \ar[dd]&& \E(\RR(\Phi,\overline{B_k(0,R)})^{J'},\R_i) \ar[dd]  \\
&\hocolim(D_{m,i}(\Phi|_{J''}))  \ar[ld]^{\pi_1^{D_{m,i}(\Phi|_{J''})}} \ar[rd]_{
F_{m,i}(\Phi|_{J''})
}
&\\
|\Delta(\mathbf{P}_{m,i}(\Phi|_{J''}))|  && \E(\RR(\Phi,\overline{B_k(0,R)})^{J''},\R_i) 
}.
\]

This implies that we have the following diagram of morphisms where both arrows are homological
$(m-1)$-equivalences:
\[
\xymatrix{
& (J' \mapsto \hocolim(D_{m,i}(\Phi|_{J'})))_{J' \in 2^J}\ar[ld] \ar[rd] &\\
(J' \mapsto |\Delta(\mathbf{P}_{m,i}(\Phi|_{J'}))|)_{J' \in 2^J} && {\hspace{-1in}\Simp^J(\RR(\Phi,\overline{B_k(0,R)}))}.
}
\]

This proves that the diagrams 
\[
(J' \mapsto |\Delta(\mathbf{P}_{m,i}(\Phi|_{J'}))|)_{J' \in 2^J}
\]
and  
\[
\Simp^J(\RR(\Phi,\overline{B_k(0,R)}))
\] 
are
homologically $(m-1)$-equivalent.

We now proceed to prove Claims \ref{thm:main:proof:claim:1} and \ref{thm:main:proof:claim:2}.

\begin{proof}[Proof of Claim~\ref{thm:main:proof:claim:1}]

Let $t \in |\Delta(\mathbf{P}_{m,i}(\Phi))|$. Then there exists a unique 
simplex $\sigma$ 
of the simplicial complex $\Delta(\mathbf{P}_{m,i}(\Phi))$ 
of the smallest possible dimension
such that $t \in |\sigma|$. 
Let $\alpha_0 \precneq \cdots \precneq \alpha_p$
be the chain in $\mathbf{P}_{m,i}(\Phi)$ corresponding to $\sigma$.
Then,
\[
(\pi^{D_{m,i}(\Phi)}_1)^{-1}(t) =  \{t\} \times D_{m,i}(\Phi)(\alpha_0).
\]

It is clear from its definition that $D'_{m,i}(\Phi)(\alpha)$ is homologically $\ell$-connected.
From Lemma~\ref{lem:thm:main:proof:claim:1} it follows that so is $D_{m,i}(\Phi)(\alpha)$.
It now follows from the 
homological version of the Vietoris-Begle theorem
(see Remark~\ref{rem:homology-vs-homotopy})
that $\pi^{D_{m,i}(\Phi)}_1$ is a homological
$\ell$-equivalence.
\end{proof}

\begin{proof}[Proof of Claim~\ref{thm:main:proof:claim:2}]
The claim will follow from the following claims.
Let 
\[
x \in \E(\RR(\Phi,\overline{B_k(0,R)})^{J},\R_i).
\]
We will prove that the
fiber
$(F_{m,i}(\Phi))^{-1}(x)$ 
is homologically $(m-1)$-connected
which will suffice to prove that 
$F_{m,i}(\Phi)$ 
is a homological $(m-1)$-equivalence by the
homological version of 
Vietoris-Begle theorem (see Remark~\ref{rem:homology-vs-homotopy}).

In order to study the fiber 
$(F_{m,i}(\Phi))^{-1}(x)$ 
we define  for each $I \subset_{\leq m+2} J$
the following posets of $\mathbf{P}_{m,i}(\Phi)$.

We define 
\begin{eqnarray*}
 \mathbf{P}_I(x) &=& \{(I,\alpha) \in \{I\} \times \mathbf{P}_{m-\card(I)+1,i+1}(\Phi_{m,i,I,J}) \mid \\
&& x \in 
\lim_{\bar\eps} D_{m-\card(I)+1,i+1}(\Phi_{m,i,I,J})(\alpha)\} 
\subset \mathbf{P}_{m,i}(\Phi),  
\end{eqnarray*}
and 
\begin{equation*}
\mathbf{Q}_I(x) =  \bigcup_{I \subset I' \subset_{\leq m+2} J} \mathbf{P}_{I'}(x).
\end{equation*}

The motivation behind the definition of the posets $\mathbf{P}_I(x), \mathbf{Q}_I(x)$ is as follows. First observe that
\begin{equation}
\label{eqn:Q:1}
(F_{m,i}(\Phi))^{-1}(x)
= \left|\bigcup_{j \in J} \Delta(\mathbf{Q}_{\{j\}}(x))\right|,
\end{equation}
and
\begin{equation}
\label{eqn:Q:2}
  \bigcap_{j \in I} \mathbf{Q}_{\{j\}}(x) = \mathbf{Q}_I(x).
\end{equation}

Our strategy for proving the homological $(m-1)$-connectedness of 
$(F_{m,i}(\Phi))^{-1}(x)$
is to use the closed covering 
provided by \eqref{eqn:Q:1} and then use the cohomological Mayer-Vietoris 
spectral sequence to reduce the problem to studying the connectivity of 
the various $\left|\Delta(\mathbf{Q}_I(x))\right|$ using \eqref{eqn:Q:2}. 
Finally, 
we prove (see Claim~\ref{thm:main:proof:claim:5}) that for each $I$, $|\Delta(\mathbf{P}_I(x))|$ is homologically
equivalent to $|\Delta(\mathbf{Q}_I(x))|$. 
This last fact allows us to use
induction on the cardinality of $I$ to prove the required connectivity 
statement for the corresponding $|\Delta(\mathbf{Q}_I(x))|$.

We now return to the proof of Claim~\ref{thm:main:proof:claim:2}.
Since, for each $I'$, with $I \subset I' \subset_{\leq m+2} J$, 
\[
\mathbf{P}_{m -\card(I') +1,i+1}(\Phi_{m,i,I',J}) \subset \mathbf{P}_{m-\card(I) +1,i+1}(\Phi_{m,i,I,J}),
\] 
there is an injective map, 
\[
\mathbf{P}_{I'}(x) \rightarrow \mathbf{P}_I(x), (I',\alpha) \mapsto (I,\alpha).
\] 

Thus there is a map
\[
\theta_I(x): \mathbf{Q}_I(x) \rightarrow \mathbf{P}_I(x),
\]
defined by 
\[
\theta_I(x)((I',\alpha)) = (I,\alpha),
\]
for each $(I',\alpha) \in \mathbf{Q}_I(x),$ where $I \subset I' \subset_{\leq m+2} J$.

It is obvious from the above definition  and 
the definition of the partial order in $\mathbf{P}_{m,i}(\Phi)$, that
the map $\theta_I(x)$ is a map of posets (i.e. a map respecting the partial orders of the two posets).

\begin{claim}
\label{thm:main:proof:claim:3}
The map $\theta_I(x)$ induces a simplicial map $\Theta_I(x): \Delta(\mathbf{Q}_I(x))\rightarrow \Delta(\mathbf{P}_I(x))$. 
Moreover, the corresponding map 
$|\Theta_I(x)| : |\Delta(\mathbf{Q}_I(x))| \rightarrow |\Delta(\mathbf{P}_I(x))|$, between the geometric realizations, is a homological equivalence.
\end{claim}

\begin{proof}
Since the map $\theta_i(x)$ is a poset map, it carries a chain of $\mathbf{Q}_I(x)$ to a chain of $\mathbf{P}_I(x)$.
This implies that $\theta_I(x)$ induces a simplicial map $\Theta_I(x): \Delta(\mathbf{Q}_I(x))\rightarrow \Delta(\mathbf{P}_I(x))$.

We now prove the second half of the claim. 
We are going to use the poset fiber theorem proved in
\cite[Lemma 3.2]{Sturmfels-Ziegler} (also \cite[Corollary 3.4]{Bjorner-Wachs}).

For $n \geq 0$, we denote by $\mathcal{B}_n$ the complete Boolean lattice
on a set with $n$ elements. It is a well known fact (see for example \cite{Wachs}) that
$|\Delta(\mathcal{B}_n)|$ is homeomorphic to $[0,1]^n$, and is thus contractible.

Let $(I,\alpha) \in \mathbf{P}_I(x)$, and $I' \subset_{\leq m+2} J$ be
the unique maximal subset of $J$ such that $(I', \alpha) \in \mathbf{P}_{I'}(x)$ (see the schematic diagram in Figure~\ref{fig:poset-fiber} which depicts subposet of the poset shown in Figure~\ref{fig:poset}).

\begin{figure}[h!]
	\begin{center}
		\includegraphics[scale=0.75]{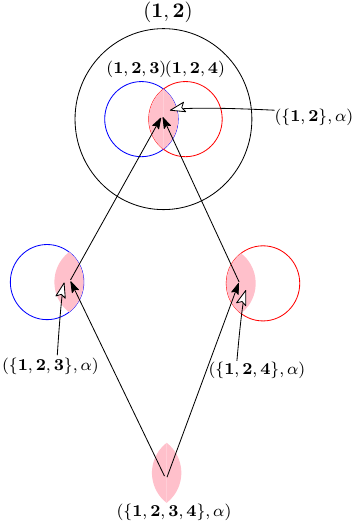}%
	\caption{\small $\theta_I(x)^{-1}((I,\alpha))$ with $I = \{1,2\}$, and $I' = \{1,2,3, 4\}$}
	\label{fig:poset-fiber}
	\end{center}
\end{figure}
Then, 
\[
\theta_I(x)^{-1}((I,\alpha)) = \{ (I'',\alpha) \;\mid\; I \subset I'' \subset I' \}.
\]

Hence, the poset 
$\theta_I(x)^{-1}((I,\alpha))$ is isomorphic as a poset to
$\mathcal{B}_{\card(I') - \card(I)}$.
Thus, $|\Delta(\theta_I(x)^{-1}((I,\alpha)))|$ is contractible.

Moreover, for each $(I'',\alpha) \in \theta_I(x)^{-1}((I,\alpha))$,
\[
\theta_I(x)^{-1}((I,\alpha))_{\succ (I'',\alpha)}
=
\{ (I''',\alpha) \;\mid\; I \subset I''' \subset I'' \},
\]
and hence 
$\theta_I(x)^{-1}((I,\alpha))_{\succ (I'',\alpha)}$
is isomorphic to  $\mathcal{B}_{\card(I'') - \card(I)}$.
This proves that
$|\Delta(\theta_I(x)^{-1}((I,\alpha))_{\succ (I'',\alpha)})|$
is contractible 
for each $(I'',\alpha) \in \theta_I(x)^{-1}((I,\alpha))$.

It now follows from the poset fiber theorem \cite[Lemma 3.2]{Sturmfels-Ziegler} (also \cite[Corollary 3.4]{Bjorner-Wachs}) that
the poset map $\theta_I(x)$ induces a homological equivalence 
$|\Theta_I(x)| : |\Delta(\mathbf{Q}_I(x))| \rightarrow |\Delta(\mathbf{P}_I(x))|$.
\end{proof}

Observe that Claim~\ref{thm:main:proof:claim:3} implies in  particular that if $\card(I) = 1$, then $|\mathbf{Q}_I(x)|$ is contractible if non-empty.

\begin{claim}
\label{thm:main:proof:claim:4}
For 
$
x \in \E(\RR(\Phi,\overline{B_k(0,R)})^{J''},\R_i) = \lim_{\bar\eps}\bigcup_{\alpha \in \mathbf{P}_{m,i}(\Phi)} D_{m,i}(\Phi)(\alpha),
$
\begin{eqnarray}
\label{eqn:thm:main:proof:claim:4}
\HH^j((F_{m,i}(\Phi))^{-1}(x))
&\cong& \Z, \mbox{ for } j =0,\\
\nonumber
&=& 0, \mbox{ for } 0 < j \leq  m.
\end{eqnarray}
\end{claim}

\begin{proof}
The proof is by induction on $m$ starting with the case $m=0$.
The case $m=-1$ is trivial.

\noindent
Base case ($m=0$). We need to show that for 
\[
x \in \E(\RR(\Phi,\overline{B_k(0,R)})^{J},\R_i) = \lim_{\bar\eps}\bigcup_{\alpha \in \mathbf{P}_{0,i}(\Phi)} D_{0,i}(\Phi)(\alpha),
\]
$(F_{0,i}(\Phi))^{-1}(x)$ is connected.

First note that
\[
F_{0,i}(\Phi) = r_{0,i}(\Phi) \circ \pi_2^{D_{0,i}(\Phi)},
\]
and $r_{0,i}(\Phi)$ is a semi-algebraic deformation retraction.
Hence, $r_{0,i}(\Phi)^{-1}(x)$ is closed and semi-algebraically connected
(in fact contractible).

Let $J(x) = \{j \in J \mid D_{0,i}(\Phi)((\{j\},\emptyset)) \cap r_{0,i}(\Phi)^{-1}(x) \neq \emptyset \}$.
Since, the sets $D_{0,i}(\Phi)((\{j\},\emptyset)), j \in J(x)$ is a covering 
of the closed and semi-algebraically connected set 
$r_{0,i}(\Phi)^{-1}(x)$ by closed sets, it follows that the union
\[
\bigcup_{j \in J(x)} D_{0,i}(\Phi)((\{j\},\emptyset))
\]
is semi-algebraically connected as well. It follows that given any $j,j' \in J(x)$, there exists
a sequence $j=j_0,j_1,\ldots,j_N = j'$ such that
for each $h, 0 \leq h \leq N-1$, 
\[
D_{0,i}(\Phi)((\{j_h\},\emptyset)) \cap D_{0,i}(\Phi)((\{j_{h+1}\},\emptyset)) \cap r_{0,i}(\Phi)^{-1}(x) \neq \emptyset.
\]

So there exists for each $h, 0 \leq h \leq N-1$
$j''  = (\{j_h,j_{h+1}\},p) \in J_{0,i,\{j_h,j_{h+1}\},\Phi}$
such that 
\[
\RR(\Phi_{\{j_h,j_{h+1}\}}(p))  \cap r_{m,i}(\Phi)^{-1}(x) \neq \emptyset.
\]

So there exists 
$\alpha  = (\{j''\},\emptyset)   \in \mathbf{P}_{-1,i+1}(\Phi_{\{j_h,_{h+1}\}})$, such that
\[
D_{0,i}(\Phi)((\{j_h,j_{h+1}\},\alpha)) \cap r_{0,i}(\Phi)^{-1}(x) \neq \emptyset,
\]
and
so 
\[
(\{j_h,j_{h+1}\},\alpha) \in 
(F_{0,i}(\Phi))^{-1}(x).
\]
Moreover,
\[
(\{j_h,j_{h+1}\},\alpha) \precneq (\{j_h\},\emptyset), (\{j_{h+1}\},\emptyset)
\] 
(using Lemma~\ref{lem:property-order}).
This implies that 
$(\{j_h\},\emptyset), (\{j_{h+1}\},\emptyset)$, and thus every pair of the form
$(\{j\},\emptyset), (\{j'\},\emptyset)$ in $(F_{0,i}(\Phi))^{-1}(x)$ belongs to the same 
connected component of $(F_{0,i}(\Phi))^{-1}(x)$. Since,  for every element of the form 
$(\{j_h,j_{h+1}\},\alpha) \in (F_{0,i}(\Phi))^{-1}(x)$ we have 
\[
(\{j_h,j_{h+1}\},\alpha) \precneq (\{j_h\},\emptyset), (\{j_{h+1}\},\emptyset)
\in (F_{0,i}(\Phi))^{-1}(x),
\]
$(\{j_h,j_{h+1}\},\alpha)$ belong to the same connected component of $(F_{0,i}(\Phi))^{-1}(x)$ as 
\[
(\{j_h\},\emptyset), (\{j_{h+1}\},\emptyset)
\]
as well.
Together, these facts imply that $(F_{0,i}(\Phi))^{-1}(x)$ is connected. This proves the claim in the base case.\\

\noindent
Inductive step. Suppose we have proved the theorem for all $m', 0 \leq m' < m$, $i \geq 0$, all finite $J'$, and 
$\Phi' \in (\mathcal{F}_{k,\R_i})^{J'}$.
We now prove it for $m,i,J,\Phi$. 
\[
x \in \E(\RR(\Phi,\overline{B_k(0,R)})^{J},R_i) = \lim_{\bar\eps}\bigcup_{\alpha \in \mathbf{P}_{m,i}(\Phi)} D_{m,i}(\Phi)(\alpha),
\]

Recall from \eqref{eqn:Q:1} that
\[
(F_{m,i}(\Phi))^{-1}(x) = \left|\bigcup_{j \in J} \Delta(\mathbf{Q}_{\{j\}}(x))\right|.
\]
Let 
\[
J' = \{j \in J\;\mid \;  \mathbf{Q}_{\{j\}}(x) \neq \emptyset \}.
\]
So
\[
(F_{m,i}(\Phi))^{-1}(x) = |\bigcup_{j \in J'} \Delta(\mathbf{Q}_{\{j\}}(x))|.
\]

It follows from the Mayer-Vietoris exact sequence in cohomology for
closed subspaces (see for example, \cite[page 148]{Iversen}) that
there exists a spectral sequence
\[
E_r^{p,q} \Rightarrow \HH^{p+q}\left(\left|\bigcup_{j \in J'} \Delta(\mathbf{Q}_{\{j\}}(x))\right|\right)
\]
whose $E_1$ term is given by
\[
E_1^{p,q} = \bigoplus_{I \subset J', \card(I) = p+1}  \HH^q\left(\left|\bigcap_{j \in I} \Delta(\mathbf{Q}_{\{j\}}(x))\right|\right).
\]

Notice that 
\[
\bigcap_{j \in I} \mathbf{Q}_{\{j\}}(x) = \mathbf{Q}_I(x),
\]
and it follows from Claim~\ref{thm:main:proof:claim:3} that
$
|\Delta(\mathbf{Q}_I(x))| 
$ 
is homotopy equivalent to
$
|\Delta(\mathbf{P}_I(x))|
$.

So we get,
\[
E_1^{p,q} = \bigoplus_{I \subset J', \card(I) = p+1}  \HH^q(|\Delta(\mathbf{P}_I(x))|).
\]

Now for $I$, with $\card(I) > 1$,  we can apply the induction hypothesis to deduce that

\begin{eqnarray*}
\HH^j(|\Delta(\mathbf{P}_I(x))|) &\cong& \Z, \mbox{ for } j =0, \\
&=& 0, \mbox{ for } 0 < j \leq  m -\card(I)+1.
\end{eqnarray*}

We can deduce from this that

\begin{eqnarray*}
E_1^{p,0} &\cong& \bigoplus_{I \subset J', \card(I) = p+1} \Z, \\
E_1^{p,q} & \cong & 0,  \mbox{ for } 0 < q \leq m - p. 
\end{eqnarray*}

It follows that 
\begin{eqnarray*}
E_2^{0,0} &\cong& \Z, \\
E_2^{p,0} &\cong& 0,  p > 0   \\
E_2^{p,q} & \cong & 0,  \mbox{ for } p \geq 0, 0 < q \leq m-p. 
\end{eqnarray*}
\end{proof}

Note that it follows from Claim~\ref{thm:main:proof:claim:5} and the Mayer-Vietoris spectral sequence argument used in its proof that
ror any 
\[
J' \subset \{j \in J\;\mid \;  \mathbf{Q}_{\{j\}}(x) \neq \emptyset \},
\]
\begin{eqnarray}
\label{eqn:claim:4+}
\HH^j\left(\left|\bigcup_{j \in J'} \Delta(\mathbf{Q}_{\{j\}}(x))\right|\right)
&\cong& \Z, \mbox{ for } j =0,\\
\nonumber
&=& 0, \mbox{ for } 0 < j \leq  m.
\end{eqnarray}

\begin{claim}
\label{thm:main:proof:claim:5}
For
\[
x \in \E(\RR(\Phi,\overline{B_k(0,R)})^{J},\R_i) = \lim_{\bar\eps}\bigcup_{\alpha \in \mathbf{P}_{m,i}(\Phi)} D_{m,i}(\Phi)(\alpha),
\]
$(F_{m,i}(\Phi))^{-1}(x)$ 
is  homologically $(m-1)$-connected.
\end{claim}

\begin{proof}
Let 
$X=F_{m,i}(\Phi)^{-1}(x)$.
It follows from \cite[Theorem 12, page 248]{Spanier} that there exists a short exact sequence:
\[
0 \rightarrow \mathrm{Ext}(\HH^{q+1}(X),\Z) \rightarrow \HH_q(X) \rightarrow \Hom (\HH^{q}(X),\Z) \rightarrow 0.
\]

Thus, for each 
$q>0$
\[
\HH^{q+1}((F_{m,i}(\Phi))^{-1}(x)) =  \HH^{q}((F_{m,i}(\Phi))^{-1}(x)) = 0
\]
 implies that
$\HH_{q}((F_{m,i}(\Phi))^{-1}(x)) = 0$.

The claim 
 now follows from
\eqref{eqn:thm:main:proof:claim:4}.
\end{proof}

Claim~\ref{thm:main:proof:claim:2} now follows from Claim~\ref{thm:main:proof:claim:5} 
and the 
homological version of the 
Vietoris-Begle theorem 
(see Remark~\ref{rem:homology-vs-homotopy}).
\end{proof}
This completes the proof of Theorem~\ref{thm:main}.
\end{proof}

\begin{proof}[Proof of Theorem~\ref{thm:main'}]
Since the proof closely mirrors that of the proof of Theorem~\ref{thm:main} we only point out the places where it differs. 
For each $\alpha \in \mathbf{P}_{m,i}(\Phi)$,
we replace the infinitesimals $\bar{\eps}_0,\ldots,\bar{\eps}_m$, by sequences of appropriately small enough positive elements
$\bar{t}_0,\ldots,\bar{t}_m$ of  $ \mathbb{R}$, in the formula
defining the set $D_{m,i}(\Phi)(\alpha)$, and denote the set 
defined by the new formula (which are semi-algebraic subset of $\mathbb{R}^k$) by
$\widetilde{D}_{m,i}(\Phi)(\alpha)$.
Similarly, we will denote the retraction
\[
\bigcup_{\alpha \in \mathbf{P}_{m,i}(\Phi)} \widetilde{D}_{m,i}(\Phi)(\alpha) \rightarrow \RR(\Phi, \overline{B_k(0,R)})^J
\]
by
$\widetilde{r}_{m,i}(\Phi)$,
and the composition

\[
\widetilde{r}_{m,i}(\Phi) \circ \pi_2^{\widetilde{D}_{m,i}(\Phi)}:\hocolim(\widetilde{D}_{m,i}(\Phi)) 
\rightarrow 
\RR(\Phi, \overline{B_k(0,R)})^J
\]
by 
$\widetilde{F}_{m,i}(\Phi)$.

Claims~\ref{thm:main:proof:claim:1} and \ref{thm:main:proof:claim:2} are replaced by:

\begin{manualclaim}{\ref*{thm:main:proof:claim:1}$^\prime$}
\label{thm:main:proof:claim:1'}
The map $\pi_1^{\widetilde{D}_{m,i}(\Phi)}:\hocolim(\widetilde{D}_{m,i}(\Phi)) \rightarrow |\Delta(\mathbf{P}_{m,i}(\Phi))| $ is an $\ell$-equivalence (and so an $(m-1)$-equivalence).
\end{manualclaim}

\begin{manualclaim}{\ref*{thm:main:proof:claim:2}$^\prime$} 
\label{thm:main:proof:claim:2'}
The map 
\[
\widetilde{F}_{m,i}(\Phi) = \widetilde{r}_{m,i}(\Phi) \circ \pi_2^{\widetilde{D}_{m,i}(\Phi)}:\hocolim(\widetilde{D}_{m,i}(\Phi)) 
\rightarrow 
\RR(\Phi, \overline{B_k(0,R)})^J
\]
is an $(m-1)$-equivalence.
\end{manualclaim}

The proof of Claim~\ref{thm:main:proof:claim:1'}
is the same as the proof of 
Claim~\ref{thm:main:proof:claim:1} replacing homologically $\ell$-connected by just $\ell$-connected, and 
using the homotopy version of the Vietoris-Begle theorem (see Remark~\ref{rem:homology-vs-homotopy}).

For the proof of Claim~\ref{thm:main:proof:claim:2'}
we need an extra argument to deduce
the $(m-1)$-connectivity of the fibers of the map 
$\widetilde{F}_{m,i}(\Phi)$ 
from the fact that they
are homologically $(m-1)$-connected which is already proved in Claim~\ref{thm:main:proof:claim:5}. In order to do this we apply Hurewicz's isomorphism theorem which requires simple connectivity of the fibers 
$(\widetilde{F}_{m,i}(\Phi))^{-1}(x)$, 
which is the content of the following claim.

\begin{claim}
\label{thm:main:proof:claim:6}
For $x \in  \RR(\Phi, \overline{B_k(0,R)})^J$,
and $m \geq 1$,
$(\widetilde{F}_{m,i}(\Phi))^{-1}(x)$ is simply connected. In other words,
$(\widetilde{F}_{m,i}(\Phi))^{-1}(x)$ is connected, and
\[
\pi_1((\widetilde{F}_{m,i}(\Phi))^{-1}(x)) \cong 0.
\]
\end{claim}

\begin{proof}
Let 
\[
J' = \{j \in J\;\mid \;  \mathbf{Q}_{\{j\}}(x) \neq \emptyset \}.
\]
So
\[
(\widetilde{F}_{m,i}(\Phi))^{-1}(x) = \left|\bigcup_{j \in J'} \Delta(\mathbf{Q}_{\{j\}}(x))\right|.
\]

We prove the stronger statement that for all non-empty subsets $J'' \subset J'$,  
\[
\left|\bigcup_{j \in J''} \Delta(\mathbf{Q}_{\{j\}}(x))\right|
\]
is simply connected.

We argue using induction on $\card(J'')$. If $\card(J'') = 1$, then $\Delta(\mathbf{Q}_{\{j\}}(x))$, where $J'' = \{j\}$, 
is a cone and
so  $|\Delta(\mathbf{Q}_{\{j\}}(x))|$ is contractible and hence simply connected. 

Suppose, we have already proved that the claim holds for all subsets of $J'$ of cardinality strictly smaller than that of $J''$. 
Let $j'' \in J''$. Then, by the  induction hypothesis, we have that 
$\left|\bigcup_{j' \in J'' - \{j''\}} \Delta(\mathbf{Q}_{\{j'\}}(x))\right|$ is simply connected. 

We first 
show that 
\[
|\Delta(\mathbf{Q}_{\{j''\}}(x))| \cap \left|\bigcup_{j' \in J'' - \{j''\}} \Delta(\mathbf{Q}_{\{j'\}}(x))\right|
\]
is connected, which is equivalent to proving that 
\[
\HH^0(|\Delta(\mathbf{Q}_{\{j''\}}(x))| \cap \bigcup_{j' \in J'' - \{j''\}} |\Delta(\mathbf{Q}_{\{j'\}}(x))|) \cong \Z.
\]
The Mayer-Vietoris exact sequence in cohomology gives the following exact sequence:

$$
\displaylines{
\HH^0(\bigcup_{j' \in J''} |\Delta(\mathbf{Q}_{\{j'\}}(x))|) \rightarrow 
\HH^0(|\Delta(\mathbf{Q}_{\{j''\}}(x))|) \oplus 
\HH^0(\bigcup_{j' \in J'' - \{j''\}} |\Delta(\mathbf{Q}_{\{j'\}}(x))|)
\rightarrow  \cr
\HH^0(|\Delta(\mathbf{Q}_{\{j''\}}(x))| \cap \bigcup_{j' \in J'' - \{j''\}} |\Delta(\mathbf{Q}_{\{j'\}}(x))|) 
\rightarrow
\HH^1(\bigcup_{j' \in J''} |\Delta(\mathbf{Q}_{\{j'\}}(x))|).
}
$$

Applying \eqref{eqn:claim:4+} we have an exact sequence

\[
\Z \rightarrow 
\Z\oplus\Z
\rightarrow 
\HH^0\left(|\Delta(\mathbf{Q}_{\{j''\}}(x))| \cap \bigcup_{j' \in J'' - \{j''\}} |\Delta(\mathbf{Q}_{\{j'\}}(x))|\right) 
\rightarrow
0,
\]
where the first map is the diagonal embedding. This implies that 
\[
\HH^0\left(|\Delta(\mathbf{Q}_{\{j''\}}(x))| \cap \bigcup_{j' \in J'' - \{j''\}} |\Delta(\mathbf{Q}_{\{j'\}}(x))|\right) \cong \Z.
\]

Finally, using the fact that $|\Delta(\mathbf{Q}_{\{j''\}}(x))|$ is simply connected, it follows from the Seifert-van Kampen's theorem 
\cite[page 151]{Spanier} that  
$\left|\bigcup_{j \in J''} \Delta(\mathbf{Q}_{\{j\}}(x))\right|$ is simply connected.
\end{proof}

We also have the obvious analog of Lemma~\ref{lem:thm:main:proof:claim:2}.

\begin{manuallemma}{\ref*{lem:thm:main:proof:claim:2}$^\prime$}
\label{lem:thm:main:proof:claim:2'}
The semi-algebraic set $\RR(\Phi,\overline{B_k(0,R)})^J$ is a semi-algebraic deformation retract of 
\[
\bigcup_{\alpha \in \mathbf{P}_{m,i}(\Phi)} \widetilde{D}_{m,i}(\Phi)(\alpha),
\]
and hence $\RR(\Phi,\overline{B_k(0,R)})^J$ and $
\bigcup_{\alpha \in \mathbf{P}_{m,i}(\Phi)} \widetilde{D}_{m,i}(\Phi)(\alpha)
$
are semi-algebraically homotopy equivalent.
\end{manuallemma}

\begin{proof}
Similar to proof of Lemma~\ref{lem:thm:main:proof:claim:2} and omitted.
\end{proof}
\begin{proof}[Proof of Claim~\ref{thm:main:proof:claim:2'}]
It follows from
Claim~\ref{thm:main:proof:claim:5}, Claim~\ref{thm:main:proof:claim:6}, and 
Hurewicz isomorphism theorem \cite[Theorem 5, page 398]{Spanier}, that
for 
\[
x \in \RR(\Phi,\overline{B_k(0,R)})^J
\]
and $m \geq 1$,
$(\widetilde{F}_{m,i}(\Phi))^{-1}(x)$ 
is  $(m-1)$-connected.
Claim~\ref{thm:main:proof:claim:2'} now follows from the previous statement and
the homotopy version of the Vietoris-Begle theorem (see Remark~\ref{rem:homology-vs-homotopy}).
\end{proof}
Finally,
Theorem~\ref{thm:main'}
follows from
Claims~\ref{thm:main:proof:claim:1'}, ~\ref{thm:main:proof:claim:2'} and Lemma
\ref{lem:thm:main:proof:claim:2'}.
\end{proof}

\subsection{Upper bound on the size of the simplicial complex $\Delta(\mathbf{P}_{m,i}(\Phi))$}
\label{subsec:size}
We now prove an upper bound on the size of the simplicial complex $\Delta(\mathbf{P}_{m,i}(\Phi))$ assuming
a ``singly exponential'' upper bound on the function $\mathcal{I}_{i,k}(\cdot)$ and $\mathcal{C}_{i,k}(\cdot)$.

\begin{definition}
\label{def:comp}
For any closed formula $\phi$ with coefficients in a real closed field $\R$, let the \emph{size}
of 
$\phi$,
$\Comp(\phi)$ be the product of the number of polynomials appearing
in the formula $\phi$ and the maximum amongst the degrees of these polynomials. 
Similarly, if $J$ is any finite set, and $\Phi \in (\mathcal{F}_{\R,k})^J$, we denote 
by $\Comp(\Phi)$ the product of the total number of polynomials appearing
in the formulas $\Phi(j), j \in J$, and the maximum amongst the degrees of these polynomials. 
\end{definition}

\begin{theorem}
\label{thm:main:complexity}
Suppose that there exists $c >0$ such that for each $\phi \in \mathcal{F}_{\R_i,k}$,
\begin{eqnarray}
\nonumber
\mathcal{I}_{i,k}(\phi) &\leq& \left(\Comp(\phi)\right)^{k^c}, \\
\label{eqn:thm:main:complexity:hyp}
\max_{j \in [\mathcal{I}_{i,k}(\phi)]} \Comp(\mathcal{C}_{i,k}(\phi) (j)) &\leq& \left(\Comp(\phi)\right)^{k^c}.
\end{eqnarray}
Let $J$ be a finite set and $\Phi \in \left(\mathcal{F}_{\R_i,k}\right)^J$.
Then the number of simplices in $\Delta(\mathbf{P}_{m,i}(\Phi))$ is bounded by
\[
(\card(J)D)^{k^{O(m)}},
\]
where 
\[
D = \Comp(\Phi).
\]
\end{theorem}

\begin{proof}

Recall 
that 
the elements of $\mathbf{P}_{m,i}(\Phi)$ 
are finite  tuples
\[
(I_0,\ldots,I_{r},\emptyset),
\]
where for each, $h, 0 \leq h\leq r$,  $I_h$ is a subset of a certain
set $J_h$ defined in Section~\ref{subsubsec:tuple-of-sets}.

We first bound the cardinalities of the various $J_{h}$'s occurring in the sequence above.
\begin{claim}
\label{claim:proof:thm:main:complexity:0}
For any $i' \geq 0$, $m' \geq -1$, finite set $J'$, $I' \subset_{m'+2} J'$, and $\Phi' \in (\mathcal{F}_{\R_{i'},k})^{J'}$,
\[
\card(J'_{m',i',I',\Phi'}) \leq  (\card(J'))^{m'+1} (\Comp(\Phi'))^{k^c}.
\]
\end{claim}

\begin{proof}[Proof of Claim~\ref{claim:proof:thm:main:complexity:0}]
Let for each fixed $i,k$,
\[
F(M',N',m',D') = \max_{\substack{J', \card(J') = N', \\
I' \subset_{m'+2} J', \card(I') = M',\\
 \Phi' \in \mathcal{F}_{\R_i,k}, \Comp(\Phi') =D'}}
 \card(J'_{m',i,I',\Phi'}).
\] 
Using Eqns. \eqref{eqn:J:base} and \eqref{eqn:J:induction} and 
Eqn. \eqref{eqn:thm:main:complexity:hyp}, we obtain:

\begin{eqnarray*}
F(m'+2,N',D') &\leq& D'^{k^c}, \\
F(M',N,D') &\leq & D'^{k^c} + (N'-M') F(M'+1,N',D'), \mbox{ for } 1 < M' < m'+2.
\end{eqnarray*}

It follows that
\begin{eqnarray*}
F(M',N',D') &\leq& D'^{k^c}(1 + N' + N'^2 +\cdots+ N'^{m'+2 - M'}) \\
&\leq& D'^{k^c} N'^{m'+1} \mbox{ for } 1 < M' \leq m'+2.
\end{eqnarray*}
The claim follows from the above inequality.
\end{proof}

\begin{claim}
\label{claim:proof:thm:main:complexity:1}
For $(I_0,\ldots,I_r,\phi) \in \mathbf{P}_{m,i}(\Phi)$, $r \leq m+1$.
\end{claim}

\begin{proof}[Proof of Claim~\ref{claim:proof:thm:main:complexity:1}]
The claim follows from the fact that $\card(I_0), \ldots,\card(I_{r-1}) \geq 2$, and hence
it follows from  Eqn.~\eqref{eqn:P-alternate:2} that 
\[
2 r \leq  \sum_{0 \leq j < r} \card(I_j) \leq m+ (r-1) + 2.
\]
It follows that
\[
r \leq m+1.
\]
\end{proof}

\begin{claim}
\label{claim:proof:thm:main:complexity:2}
For
every tuple $(I_0,\ldots,I_r,\emptyset) \in \mathbf{P}_{m,i}(\Phi)$,
$0 \leq h \leq r$, 
\begin{eqnarray*}
\Comp(\Phi_h(\alpha)) &\leq& D^{k^{c h}}, \mbox{ for } \alpha \in  J_h,\\
\card(J_{h}) &\leq&  N^{(m+1)^{h}} D^{(k(m+1))^{c h}},
\end{eqnarray*}
where $J_h,\Phi_h, 0 \leq j \leq r$ are defined in Eqn. \eqref{eqn:P-alternate}, and
$N = \card(J)$.
\end{claim}

\begin{proof}[Proof of Claim~\ref{claim:proof:thm:main:complexity:2}]
The claim is obviously true for $h=0$. 
Also, note that for each $h, 0\leq h \leq r$,
\[
m_h \leq m.
\]
The claim now follows by induction on $h$, 
using the inductive definitions of $J_h,\Phi_h$ (see Eqn.~\eqref{eqn:P-alternate}),
Eqn. \eqref{eqn:thm:main:complexity:hyp}, and
Claim~\ref{claim:proof:thm:main:complexity:0}.
\end{proof}

\begin{claim}
\label{claim:proof:thm:main:complexity:3}
\[
\card(\mathbf{P}_{m,i}(\Phi)) \leq (\card(J)D)^{k^{O(m)}}.
\]
\end{claim}

\begin{proof}[Proof of Claim~\ref{claim:proof:thm:main:complexity:3}]
In order to bound the cardinality of $\mathbf{P}_{m,i}(\Phi)$, we bound the number of possible
choices of $I_0,\ldots,I_r$ for $(I_0,\ldots,I_r,\emptyset) \in \mathbf{P}_{m,i}(\Phi)$.

It follows from Eqn.~\eqref{eqn:P-alternate:2}, that for each $h, 0 \leq h \leq r$,
\begin{eqnarray*}
\card(I_h) &\leq& m - \sum_{t=0}^{h-1} \card(I_t) + h + 2 \\
&\leq& m - 2h +h +2 \; (\mbox{since } \card(I_t) \geq 2, 0 \leq t < r) \\
&\leq& m - h +2 \\
&\leq&  m+2.
\end{eqnarray*}

Since by Claim~\ref{claim:proof:thm:main:complexity:2}
for $0 \leq h \leq r$, 
\[
\card(J_{h}) \leq  N^{(m+1)^{h}} D^{(k(m+1))^{c h}},
\]
the number of choices for $I_h$ is clearly bounded by 
\[
\sum_{t=2}^{m+2} \binom{N^{(m+1)^{h}} D^{(k(m+1))^{c h}}}{h} \leq N^{m^{O(h)}} D^{k^{O(h)}},
\]
noting that $m \leq k$.
The above inequality, together with the fact that $r \leq m+1$ (by Claim~\ref{claim:proof:thm:main:complexity:1}),  proves the  claim.
\end{proof}

\begin{claim}
\label{claim:proof:thm:main:complexity:4}
The length of any chain in  $\mathbf{P}_{m,i}(\Phi)$ is bounded by $2 m+2$.
\end{claim}
\begin{proof}[Proof of Claim~\ref{claim:proof:thm:main:complexity:4}]
Suppose that $\alpha = (I_0^\alpha,\ldots,I_{r_\alpha}^\alpha,\emptyset), \beta = (I_0^\beta,\ldots, I_{r_\beta}^\beta,\emptyset) \in \mathbf{P}_{m,i}(\Phi)$, $\beta \precneq \alpha$ and $\alpha \neq \beta$.

It follows from Eqn. \eqref{eqn:order-alternate} that
\begin{equation*}
(r_\alpha \leq r_\beta) \mbox{ and } I^\alpha_h \subset I^\beta_h, 0 \leq h \leq r_\alpha.  
\end{equation*}

In particular, this implies that
$0 < \sum_{h=0}^{r_\alpha} \card(I_h^\alpha) < \sum_{h=0}^{r_\beta} \card(I_h^\beta)$.
Since for any $(I_0,\ldots,I_r,\emptyset) \in \mathbf{P}_{m,i}(\Phi)$,
we have that 
\[
\sum_{0 \leq h < r} \card(I_h) \leq m+r+2,
\]
\[
\card(I_r) = 1,
\]
and 
\[
r \leq m+1,
\]
it follows immediately that the length of a chain in $\mathbf{P}_{m,i}(\Phi)$ is bounded by $2m+2$.
\end{proof}

The theorem follows from Claims~\ref{claim:proof:thm:main:complexity:1}, \ref{claim:proof:thm:main:complexity:2}, 
\ref{claim:proof:thm:main:complexity:3}
and \ref{claim:proof:thm:main:complexity:4}.
\end{proof}

\section{Simplicial replacement: algorithm}
\label{sec:algo}
We begin with some mathematical and algorithmic preliminaries.

\subsection{Mathematical preliminaries}
\subsubsection{Making closed}
We need to take care of the following technical issue.
The output of Algorithm~\ref{alg:contractible-cover}  (Covering by Contractible Sets) described below, consists 
of a tuple of formulas whose realizations are closed and
semi-algebraically contractible semi-algebraic sets, but the
formulas themselves need not be closed. However, in the recursive
Algorithm~\ref{alg:poset}  (Computing the poset $\mathbf{P}_{m,i}(\Phi)$)
we need to assume that the input formulas are closed. 
We get around this problem by a construction which allows us to replace a
formula (not necessarily closed) defining a closed and bounded semi-algebraic set $S$ by another \emph{closed} formula
defining a  semi-algebraic set $S'$ such that
$S' \searrow S$. The construction is quite similar (but not identical) to the one by Gabrielov and Vorobjov \cite{GaV}. In the construction given in \cite{GaV} the original set is not necessarily a deformation retract of the new one. By using the extra property that we assume, namely that the given set is closed (albeit without a closed description), we are able to ensure that it is a retract of the new one defined by a closed formula. 

We remark here that the algorithmic problem of  obtaining a closed description of a given closed semi-algebraic set
(described by a not formula which is not necessarily closed) 
is a difficult problem
for which no algorithm with singly exponential complexity is known in general. We do not solve this general problem, because the closed description that we obtain does not describe the original set, but a 
closed (infinitesimal) neighborhood of it.

The key result of this section is Lemma~\ref{16:lem:star}.

Let $\mathcal{P} = \{P_1,\ldots,P_s\}
\subset \R[X_1,\ldots,X_k]$ be a finite set of polynomials, and let $B \subset \R^k$ a closed euclidean ball.

\begin{notation}
\label{not:level}
For $\sigma \in \{0,1,-1\}^\mathcal{P}$, let 
\[
\level(\sigma) = \card(\{P \in \mathcal{P} \mid \sigma(P) = 0 \}).
\]
\end{notation}	

For $c,d \in \R, 0< d < c$, and $\sigma \in \{0,1,-1\}^\mathcal{P}$, let
$\overline{\sigma}(c,d)$ denote the closed formula
\[
\bigwedge_{\sigma(P) = 0} (-d \leq P \leq d) \wedge \bigwedge_{\sigma(P) = 1} (P \geq c) \wedge \bigwedge_{\sigma(P) = -1} (P \leq -c).
\]

\begin{notation}
\label{not:Sigma-phi}
For a $\mathcal{P}$-formula $\phi$ we denote 
\[
\Sigma_{\phi} = 
\{ \sigma \in \{0,1,-1\}^\mathcal{P} \mid 
\left(\bigwedge_{P \in \mathcal{P}} (\mathrm{sign}(P) = \sigma(P))\right) \Rightarrow \phi \},
\]
where ``$\Rightarrow$''  denotes logical implication.
\end{notation}

Let 
\[
\R' = \R\la \mu_s,\nu_s, \cdots, \mu_0, \nu_0\ra = \R\la\bar\eta\ra,
\]
denoting by $\bar\eta$ the sequence $\mu_s,\nu_s, \ldots, \mu_0, \nu_0$.

\begin{notation}
\label{not:P-star}
We denote
\[
\mathcal{P}^*(\bar{\mu},\bar{\nu}) = \bigcup_{P \in \mathcal{P}} \bigcup_{j = 0}^{s} \{P \pm \mu_j, P\pm \nu_j\} \subset \R'[X_1,\ldots,X_k].
\]
\end{notation}

Finally, 
\begin{notation}
\label{not:phi-star}
We denote by
$\phi^*(\bar{\mu},\bar{\nu})$ the $\mathcal{P}^*(\bar{\mu},\bar{\nu})$-\emph{closed} formula
\[
\bigvee_{\sigma \in \Sigma_{\phi}} \overline{\sigma}(\mu_{\level(\sigma)},\nu_{\level(\sigma)})
\]
(see Notation~\ref{not:Sigma-phi}).
\end{notation}

Following the notation introduced above. 
\begin{lemma}
\label{16:lem:star}
Let $R > 0, B = \overline{B_k(0,R)}$, and 
suppose that $S = \RR(\phi,B)$ is closed. Then,
\[
S'  \searrow S,
\]
where $S' = \RR(\phi^*(\bar{\mu},\bar{\nu}),\E(B,\R'))$.
In particular, $\E(S,\R')$ is a semi-algebraic deformation retract of 
$S'$.
\end{lemma}

\begin{proof}
See Appendix \ref{sec:AppendixA}.
\end{proof}

\begin{remark}
\label{rem:16:lem:star}
It is necessary to use multiple infinitesimals in the construction given above. As a warning consider the following example.

\begin{example}
Let $k=1,s=2, B = [-2,2]$, and
\begin{eqnarray*}
P_1 &=& X^2(X-1), \\
P_2 &=& X.
\end{eqnarray*}
Let $\sigma_1,\sigma_2$ be defined by,
$$\displaylines{
\sigma_1(P_1) =1, \sigma_2(P_2) = 1, \cr
\sigma_1(P_1) =0, \sigma_2(P_2) = 1.
}
$$
Let $\phi$ be the unique formula such that $\Sigma_{\phi} = \{\sigma_1,\sigma_2\}$.
Then, $\RR(\phi,B) = [1,2]$ is a closed semi-algebraic set, but
$\phi$ is not a closed formula.

However, if we take the closed formula $\phi^*(\mu_0,\ldots,\mu_0)$
(i.e. using only one infinitesimal)
then
\[
\lim_{\mu_0} \RR(\phi^*(\mu_0,\ldots,\mu_0),B) = \{0\} \cup [1,2] \supsetneq \RR(\phi,B).
\]
However, it is easy to verify that
\[
\RR(\phi^*(\bar\mu,\bar\nu),B) \searrow  \RR(\phi,B) = [1,2].
\]
\end{example}
\end{remark}

\subsubsection{Strong general position}
We need the following notion of ``strong general position'' of a finite set of polynomials. It is a required property for the input to Algorithm~\ref{alg:contractible-cover}.

\begin{definition}
\label{def:general-position}
Let 
$\mathcal{P}  \subset \R [X_{1} , \ldots ,X_{k}
]$ be a  finite set. We say that $\mathcal{P}$ is in
\emph{$\ell$-general position}, if no more than 
$\ell$ polynomials belonging to $\mathcal{P}$ have a common zero in $\R^{k}$.
The set $\mathcal{P}$ is in \emph{strong $\ell$-general
position}  if moreover any $\ell$
polynomials belonging to $\mathcal{P}$ have at most a
finite number of common zeros in $\R^{k}$.
\end{definition}

Using the same notation as in Lemma~\ref{16:lem:star} we have:
\begin{lemma}
\label{lem:general-position}
The set 
\[
\mathcal{P}^*(\bar{\mu},\bar{\nu})
\]
is in strong $k$-general position.
\end{lemma}

\begin{proof}
The claim follows easily from the fact that  $\mu_0,\ldots,\mu_s,\nu_0,\ldots,\nu_s$ are algebraically independent over $\R$ and semi-algebraic Sard's theorem 
\cite[Theorem 5.56]{BPRbook2}.
\end{proof}

We now describe some preliminary algorithms that we will need.
\subsection{Algorithmic preliminaries}
The following algorithm is described in \cite{BPRbook2}. We briefly recall the input, output and complexity.
We made a small and harmless modification to the input by requiring that the closed semi-algebraic
of which the covering is being computed is contained in the closed ball of radius $R$ centered at the origin, rather than in the sphere of radius $R$. This is done to avoid complicating notation down the road and
is not significant since the algorithm can be easily modified to accommodate this change without 
any change in the complexity estimates.

\begin{algorithm}[H]
\caption{(Covering by Contractible Sets)}
\label{alg:contractible-cover}
\begin{algorithmic}[1]
\INPUT
\Statex{
\begin{enumerate}[(a)]
\item 
a finite set of $s$ polynomials ${\mathcal P} \subset \D[\bar\eps][X_1,\ldots,X_k]$
in strong $k$-general position on $\R^k$,
with $\deg(P_i) \leq d$ for $1 \leq i \leq s$,
\item a  ${\mathcal P}$-closed formula $\phi$ such that
semi-algebraic set $\RR(\phi) \subset \overline{B_k(0,R)}$, 
for some $R >0$, $R \in \R$.
\end{enumerate}
}
\OUTPUT
\Statex{
\begin{enumerate}[(a)]
\item a finite set of polynomials $\mathcal{H} \subset \D[\bar\eps,\bar\zeta][X_1,\ldots,X_k]$,
where $\bar\zeta = (\zeta_1,\ldots,\zeta_{2 \card(\mathcal{H})})$;
\item
a tuple  of $\mathcal{H}$-formulas $(\theta_\alpha)_{\alpha \in I}$ such that
each  $\RR(\theta_\alpha,\R\la\bar\eps,\bar\zeta\ra^k), \alpha \in I$ is a closed
semi-algebraically contractible set, and
\item 
\[
\bigcup_{\alpha \in I} 
\RR(\theta_\alpha,\R\la\bar\eps,\bar\zeta\ra^k) = \RR(\psi,\R\la\bar\eps,\bar\zeta\ra^k).
\]
\end{enumerate}
}
\COMPLEXITY
The complexity of the algorithm is 
bounded by $(\card(\mathcal{P})^{(k+1)^2}D^{k^{O(1)}}$, 
where $D = \max_{P \in \mathcal{P}} \deg_{\bar{X},\bar{\eps}} (P)$.
Moreover,
\begin{eqnarray*}
\card(I), \card(\mathcal{H}) &\leq& (\card(\mathcal{P}) D)^{k^{O(1)}}, \\ 
\deg_{\bar{Y}}(H), \deg_{\bar{\eps}}(H),\deg_{\bar{\zeta}}(H) &\leq&
D^{k^{O(1)}}.
\end{eqnarray*}

Suppose that $\bar{\eps} = (\eps_1,\ldots,\eps_t)$, and that 
each polynomial in $\mathcal{P}$ depends on at most $m$ of the $\eps_i$'s.
Then,
each polynomial appearing in $\mathcal{H}$ depends on at most
$m(k+1)^2$ of $\eps_i$'s,  and on at most one of the $\zeta_i$'s. 
\end{algorithmic}
\end{algorithm}

\begin{remark}
\label{rem:local}
Note that the last claim in the complexity of Algorithm~\ref{alg:contractible-cover}, namely that each polynomial appearing in any of the formulas $\theta_\alpha$ depends on at most
$m(k+1)^2$ of $\eps_i$'s, and on at most one of the $\zeta_i$'s,
does not appear explicitly in \cite{BPRbook2}, but is evident on a close examination of the algorithm. It is also 
reflected in the fact that the combinatorial part (i.e. the part depending on $\card(\mathcal{P})$) 
of the complexity of Algorithm 16.14 in \cite{BPRbook2} is bounded by $\card(\mathcal{P})^{(k+1)^2}$. This is
because the Algorithm 16.14 in \cite{BPRbook2} has a ``local property'', namely that all computations involve
at most a small number  (in this case $(k+1)^2$) polynomials in the input at a time.
\end{remark}

\subsection{Algorithm for computing simplicial replacement}
We now describe an algorithm that given a tuple of formula $\Phi$ and $m,i \geq 0$, computes
the corresponding poset $\mathbf{P}_{m,i}(\Phi)$, using Algorithm~\ref{alg:contractible-cover} to
compute $\mathcal{I}_{j,k}(\phi)$ and $\mathcal{C}_{j,k}(\phi)$ for different $j$ and $\phi$ which arise
in the course of the execution of the algorithm. 

\begin{algorithm}[H]
\caption{(Computing the poset $\mathbf{P}_{m,i}(\Phi))$}
\label{alg:poset}
\begin{algorithmic}[1]
\INPUT
\Statex{
\begin{enumerate}[(a)]
\item 
$\ell, 0 \leq \ell \leq k$, $m, -1 \leq m \leq \ell$, $i, 0 \leq i \leq m+2$.
\item
A finite set of polynomials $\mathcal{P} \subset \D[\bar{\eps}_0,\ldots,\bar{\eps}_i][X_1,\ldots,X_k]$, where 
$\D$ is an ordered domain contained in a real closed field $\R$.
\item
An element $r \in \D$, $r >0$.
\item
\label{itemlabel:alg:poset:input:d}
For each $j, 0 \leq j \leq N$, a 
$\mathcal{P}$-formula $\phi_j$, 
such that $\RR(\phi_j, \overline{B_k(0,1/r)})$ is 
closed and
homologically $\ell$-connected (and $\ell$-connected if $\R = \mathbb{R}$).
\end{enumerate}
}
 \OUTPUT
 \Statex{
The poset $\mathbf{P}_{m,i}(\Phi)$ (see Definition~\ref{def:poset}), 
where $\Phi$ is defined by $\Phi(j) = \phi_j, j \in [N]$, and the various $\mathcal{I}_{\cdot,k}(\cdot)$
$\mathcal{C}_{\cdot,k}(\cdot)$ are obtained by calls to  Algorithm~\ref{alg:contractible-cover}.
}

\PROCEDURE
\State{ $J \leftarrow  [N]$.}
\If{$m=-1$} 
	\State{
	Output \[
	\mathbf{P}_{-1,i}(\Phi) = 
	 \{(\{j\},\phi) \mid j \in J\},
	\] 
and the order relation to be the trivial one -- namely for 
$j,j' \in  J$,
\[
(\{j\},\emptyset)  \prec (\{j'\},\emptyset)  \Leftrightarrow j=j'.
\]

		}

\Else
    \State{ 
    \[
    \mathcal{P} \gets \mathcal{P} \cup \left\{r^2\sum_{i=1}^{k} X_i^2 - 1\right\}.
    \]
    }
    \For {$j \in J$}
        \State{ 
        \[
        \Phi(j) \gets \Phi(j) \wedge \left(r^2\sum_{i=0}^{k} X_i^2 - 1 \leq 0\right).
        \] 
        }
    \EndFor
	\For {each  subset $I \subset_{\leq m+2} J$}
		\State{Use 
		Definition~\ref{def:J}
		to compute $J_{m,i,I,\Phi}$ and $\Phi_{m,i,I,J}$, using
		Algorithm~\ref{alg:contractible-cover} with input 
		$\mathcal{P}^*(\bar{\mu},\bar{\nu}) \subset \R[\bar\eta][X_1,\ldots,X_k]$ 
		(where $\bar\eta$ denotes the alternating sequence of $\mu_i$'s
		and $\nu_i$'s appearing in Notation~\ref{not:P-star}),
		and the formula 
		$
		 \bigwedge_{j \in I} \Phi(j)^\star(\bar{\mu},\bar{\nu}),
		$
		(noting that 
		$\RR(\bigwedge_{j \in I} \Phi(j)^\star(\bar{\mu},\bar{\nu}))$
		is contained in  $\overline{B_k(0,2/r)}$),
		to compute 
		$\mathcal{I}_{i,k}(\bigwedge_{j \in I} \Phi(j))$ and $\mathcal{C}_{i,k}((\bigwedge_{j\in I} \Phi(j)))$.

		The polynomials appearing in the formulas in 
		$\mathcal{C}_{i,k}((\bigwedge_{j\in I} \Phi(j)))$ have coefficients in
		$\D[\bar{\eps}_0,\ldots, \bar{\eps}_i,\bar{\eps}_{i+1}]$,
		where 
		$\bar{\eps}_{i+1} = (\bar{eta},\bar{\zeta})$, and 
		$\bar{\zeta}$
		is a new tuple of infinitesimals.
		} 
		\label{alg:poset:line:0}
	\EndFor
\algstore{myalg}
\end{algorithmic}
\end{algorithm}
 
\begin{algorithm}[H]
\begin{algorithmic}[1]
\algrestore{myalg}

	\For{$I \subset J, 1 < \card(I) \leq m+2$} 
		\State{Use Algorithm~\ref{alg:poset} recursively with input $\ell, m -\card(I)+1, i+1, \mathcal{P}_I,\Phi_{m,i,I,J},r$, where $\mathcal{P}_I \subset \D[\bar{\eps}_0,\ldots,\bar{\eps}_{i+1}]$ is the set of polynomials occurring in $\Phi_{m,i,I,J}$.}

		\State{
\begin{equation*}
\mathbf{P}_{m,i}(\Phi)  \leftarrow 
\left\{(\{j\},\phi) \mid j \in J\right\}
\cup  \bigcup_{I \subset J, 1 < \card(I) \leq m+2}   \{I\} \times \mathbf{P}_{m -\card(I) +1,i+1}(\Phi_{m,i,I,J}).
\end{equation*}
			}
		\State{
Define partial order $\prec$ on $\mathbf{P}_{m,i}(\Phi)$ as in Definition~\ref{def:poset}.
}
	\EndFor
\EndIf

\COMPLEXITY  The complexity of the algorithm, as well as $\card(\mathbf{P}_{m,i}(\Phi))$, are bounded by 
\[
 (N s d)^{k^{O(m)}},
\]
where $s = \card(\mathcal{P})$, and $d = \max_{P \in \mathcal{P}} \deg(P)$.
\end{algorithmic}
\end{algorithm}

\begin{proof}[Proof of correctness]
The algorithm follows Definition~\ref{def:poset}. The correctness
of the algorithm follows from 
Lemma~\ref{16:lem:star}, Lemma~\ref{lem:general-position}, and the
correctness of  Algorithm~\ref{alg:contractible-cover}.
\end{proof}

\begin{proof}[Complexity analysis]
The bound on  $\card(\mathbf{P}_{m,i}(\Phi))$ is a consequence of Theorem~\ref{thm:main:complexity}.
The complexity of the algorithm follows from the complexity of the Algorithm~\ref{alg:contractible-cover} and
an argument as in the proof of Theorem~\ref{thm:main:complexity}. 

There is one additional 
point to note that in the recursive calls algorithm the arithmetic operations take place in a larger ring,
namely - $\D[\bar\eps_0,\ldots,\bar\eps_{m+2}]$.

It follows from the complexity of  Algorithm~\ref{alg:contractible-cover} that the number of different infinitesimals occurring in 
each polynomial that is computed in the course of Algorithm~\ref{alg:poset} is bounded by
$k^{O(m)}$,  and these infinitesimals occur with degrees bounded by $d^{k^{O(m)}}$.
Hence each arithmetic operation involving the coefficients with these polynomials 
costs $\left(d^{k^{O(m)}}\right)^{k^{O(m)}} = d^{k^{O(m)}}$ arithmetic operations in the ring $\D$. This does not affect the asymptotics of the complexity,
where we measure arithmetic operations in the ring $\D$. 
\end{proof}

\begin{remark}
\label{rem:alg:poset}
Suppose we define (following the same notation as in 
Properties~\ref{property:thm:main} and 
\ref{property:thm:main'}
and Algorithm~\ref{alg:poset}) 
for $\phi \in \mathcal{F}_{\R_i,k}$, 
\begin{eqnarray*}
\mathcal{I}_{i,k}(\phi) &=& \card(I) -1, \\
\mathcal{C}_{i,k}(\phi) &=& (\theta_\alpha)_{\alpha \in I},
\end{eqnarray*}
where $(\theta_\alpha)_{\alpha \in I}$ is the output of 
Algorithm~\ref{alg:contractible-cover} with input the set of polynomials
appearing in the definition of 
$\phi^*(\bar{\mu},\bar{\nu})$ (see Notation~\ref{not:phi-star}), the closed formula $\phi^*(\bar{\mu},\bar{\nu})$, 
and
$R$ set to $1/r$
(as in Line~\ref{alg:poset:line:0} of Algorithm~\ref{alg:poset}).

Then it follows from the correctness of Algorithm~\ref{alg:contractible-cover}, 
that
(denoting by 
$\R_i = \R\la\bar{\eps}_0,\ldots,\bar{\eps}_i \ra$ as in Algorithm~\ref{alg:poset})
the tuple
\[
\left((\R_i)_{i \geq 0},1/r,k, (\mathcal{I}_{i,k})_{i \geq 0}, (\mathcal{C}_{i,k})_{i \geq 0} \right)
\]
satisfies  the homological $\ell$-connectivity property over $\R$
(resp. $\ell$-connectivity property if $\R = \mathbb{R}$) for every $\ell \geq 0$
(see Property~\ref{property:thm:main} and 
Property~\ref{property:thm:main'}.
\end{remark}

\begin{algorithm}[H]
\caption{(Simplicial replacement)}
\label{alg:simplicial-replacement}
\begin{algorithmic}[1]
\INPUT
\Statex{
\begin{enumerate}[(a)]
\item
A finite set of polynomials $\mathcal{P} \subset \D[X_1,\ldots,X_k]$ where $\D$ is an ordered domain contained  in a real closed field $\R$.
\item
An integer $N \geq 0$, and for each $i \in [N]$, a $\mathcal{P}$-closed formula $\phi_i$.
\item 
$\ell, 0 \leq \ell \leq k$.
\end{enumerate}
 }

 \OUTPUT
 \Statex{
A simplicial complex $\Delta$ and for each $I \subset [N]$ a subcomplex $\Delta_I \subset \Delta$ such that there is a diagrammatic homological $\ell$-equivalence
\[
(I \mapsto \Delta_I)_{I \subset [N]} \overset{h}{\sim}_\ell \Simp^{[N]}(\RR(\Phi)),
\]
where $\Phi(i) = \phi_i, i \in [N]$.
In case $\R = \mathbb{R}$, then the simplicial complex $\Delta$ and the subcomplexes $\Delta_I$ satisfy the
stronger property, namely:
 \[
(I \mapsto \Delta_I)_{I \subset [N]} \sim_\ell \Simp^{[N]}(\RR(\Phi)),
\]
where $\Phi(i) = \phi_i, i \in [N]$.
}

\PROCEDURE
\State{ 
Let $0< \delta < 1$ be an infinitesimal.
} 
 \State{$\mathcal{P} \gets \mathcal{P} \cup \{4\cdot\delta^2 \cdot (X_1^2 +\cdots+X_k^2) -1\}$.}
 \For {$0 \leq i \leq N$}
        \State{$\phi_i \gets \phi_i \wedge (4 \cdot \delta^2 \cdot (X_1^2 +\cdots+X_k^2) -1 \leq 0)$.}
 		\State{
 		\label{line:alg:simplicial-replacement:0}
 		Call Algorithm~\ref{alg:contractible-cover} 
 		 with input 
 		 $\mathcal{P}^*(\bar\mu,\bar\nu)$ (see Notation~\ref{not:P-star})
 		 the formula
    $\phi_{i}^\star(\bar\mu,\bar\nu) 
    $ (see Notation~\ref{not:phi-star})
    as input,
    and let $\Phi_i = (\phi_{i,1},\ldots,\phi_{i,N_i})$ be the output. }
    \State{$\mathcal{P}_i \leftarrow \mbox{ the set of polynomials appearing in the formula $\Phi_i$}$.}
\EndFor

\State{$\mathcal{P}' \leftarrow  \bigcup_{i \in [N]} \mathcal{P}_i$.}
\For{ $0 \leq i \leq n$}
\State{$J_i \gets \{(i,j) \mid 1 \leq j \leq N_i\}$.}
\EndFor
\State{$J \gets  \bigcup_{i \in [N]} J_i$.}
\State {
Let 
$\Psi \in (\mathcal{F}_{\R\la\delta,\bar\eta,\bar\zeta\ra,k})^J$ be defined by
$\Psi((i,j)) = \phi_{i,j}$.}
\State{
\label{line:alg:simplicial-replacement:1}
Call Algorithm~\ref{alg:poset} with input 
\[
(\ell+1, m+1, 0, \mathcal{P}', J, \delta, \Psi),
\] and  
let $\mathbf{P}_{m,0}(\Psi)$ denote the output.}
\State{Output the simplicial complex $\Delta(\mathbf{P}_{m,0}(\Psi))$, and for each subset $I \subset [N]$,
the subcomplex $\Delta(\mathbf{P}_{m,0}(\Psi|_{\bigcup_{i\in I}J_i}))$.}
\COMPLEXITY
The complexity of the algorithm is 
bounded by $(sd)^{k^{O(\ell)}}$, where $s =\card(\mathcal{P})$ and $d = \max_{P \in \mathcal{P}} \deg(P)$.

\end{algorithmic}
 \end{algorithm}

\begin{proof}[Proof of correctness]
Observe that the image of the realization of each of the formulas $\phi_{i,j}$ 
obtained in Line~\ref{line:alg:simplicial-replacement:0}
under the $\lim_{\bar\eta}$ map is contained
in $\overline{B_k(0,1/2\delta)}$. This implies that the 
realization of each of the formulas $\phi_{i,j}$ is contained in $\overline{B_k(0,\delta)}$.
Thus, in the call to Algorithm~\ref{alg:poset} in 
Line~\ref{line:alg:simplicial-replacement:1}, the input satisfies
property \eqref{itemlabel:alg:poset:input:d} of the input
specification of Algorithm~\ref{alg:poset} with $r = \delta$.

The correctness of the algorithm now follows from
Lemma~\ref{16:lem:star}, Lemma~\ref{lem:general-position},
the correctness of Algorithm~\ref{alg:poset}, 
Remark~\ref{rem:alg:poset},
and Theorems~\ref{thm:main} and \ref{thm:main'}.
\end{proof}

\begin{proof}[Complexity analysis]
The complexity bound follows from the complexity bounds of  Algorithms~\ref{alg:contractible-cover} and \ref{alg:poset}.
\end{proof}

\begin{proof}[Proofs of Theorems~\ref{thm:alg} and  \ref{thm:alg'}]
Both theorems now follows from the correctness and the complexity analysis Algorithm~\ref{alg:simplicial-replacement}.
\end{proof}

\section{Future work and open problems}
\label{sec:conclusion}
We conclude by stating some open problems and possible future directions of research in this area.
\begin{enumerate}[1.]
\item 
It is an interesting problem to try to make the poset $\mathbf{P}_{m,i}(\Phi)$ in Theorem~\ref{thm:main} smaller in size and more efficiently computable. For instance,
in Theorem~\ref{thm:main:complexity} one should be able to improve the dependence on
$\card(J)$.
\item
There are some recent work in algorithmic semi-algebraic geometry where algorithms have been developed for computing the first few Betti numbers of semi-algebraic subsets of $\R^k$
having special properties. For example, in \cite{BR2020} the authors give an algorithm to compute 
the first $\ell$ Betti numbers of semi-algebraic subsets of $\R^k$ defined by \emph{symmetric} polynomials of
degrees bounded by  some constant $d$. The complexity of the algorithm is doubly exponential in
both $d$ and $\ell$ (though polynomial in $k$ for fixed $d$ and $\ell$). This algorithm uses
semi-algebraic triangulations which leads to the doubly exponential complexity. It is an interesting problem to investigate whether applying the efficient simplicial replacement of the current paper
the dependence on $d$ and $\ell$ can be improved. 
\end{enumerate} 

\section*{Acknowledgements}
The authors are grateful to the referees for their detailed reading and
many helpful comments and corrections. 
In particular, we thank one of the referees
for pointing out an error in an example in the previous version of the paper.
The example has been omitted in this version.

\bibliographystyle{amsplain}
\bibliography{master}

\def\cprime{$'$} \def\cprime{$'$}
\providecommand{\bysame}{\leavevmode\hbox to3em{\hrulefill}\thinspace}
\providecommand{\MR}{\relax\ifhmode\unskip\space\fi MR }
\providecommand{\MRhref}[2]{%
  \href{http://www.ams.org/mathscinet-getitem?mr=#1}{#2}
}
\providecommand{\href}[2]{#2}
\begin{thebibliography}{10}

\bibitem{Bas05-first}
S.~Basu, \emph{Computing the first few {B}etti numbers of semi-algebraic sets
  in single exponential time}, J. Symbolic Comput. \textbf{41} (2006), no.~10,
  1125--1154. \MR{2262087 (2007k:14120)}

\bibitem{Basu-survey}
\bysame, \emph{Algorithms in real algebraic geometry: a survey}, Real algebraic
  geometry, Panor. Synth\`eses, vol.~51, Soc. Math. France, Paris, 2017,
  pp.~107--153. \MR{3701212}

\bibitem{BPR99}
S.~Basu, R.~Pollack, and M.-F. Roy, \emph{Computing roadmaps of semi-algebraic
  sets on a variety}, J. Amer. Math. Soc. \textbf{13} (2000), no.~1, 55--82.
  \MR{1685780 (2000h:14048)}

\bibitem{BPRbook2}
\bysame, \emph{Algorithms in real algebraic geometry}, Algorithms and
  Computation in Mathematics, vol.~10, Springer-Verlag, Berlin, 2006 (second
  edition). \MR{1998147 (2004g:14064)}

\bibitem{BPRbettione}
\bysame, \emph{Computing the first {B}etti number of a semi-algebraic set},
  Found. Comput. Math. \textbf{8} (2008), no.~1, 97--136.

\bibitem{BR2020}
Saugata Basu and Cordian Riener, \emph{Vandermonde varieties, mirrored spaces,
  and the cohomology of symmetric semi-algebraic sets}, Foundations of
  Computational Mathematics (2021).

\bibitem{Bjorner}
A.~Bj\"{o}rner, \emph{Topological methods}, Handbook of Combinatorics
  (R.~Graham, M.~Grotschel, and L.~Lovasz, eds.), vol.~II,
  North-Holland/Elsevier, 1995, pp.~1819--1872.

\bibitem{Bjorner-Wachs}
Anders Bj\"{o}rner, Michelle~L. Wachs, and Volkmar Welker, \emph{Poset fiber
  theorems}, Trans. Amer. Math. Soc. \textbf{357} (2005), no.~5, 1877--1899.
  \MR{2115080}

\bibitem{BSS}
L.~Blum, F.~Cucker, M.~Shub, and S.~Smale, \emph{Complexity and real
  computation}, Springer-Verlag, New York, 1998, With a foreword by Richard M.
  Karp. \MR{1479636 (99a:68070)}

\bibitem{BCL2019}
Peter B\"{u}rgisser, Felipe Cucker, and Pierre Lairez, \emph{Computing the
  homology of basic semialgebraic sets in weak exponential time}, J. ACM
  \textbf{66} (2019), no.~1, Art. 5, 30, [Publication date initially given as
  2018]. \MR{3892564}

\bibitem{BCT2020.2}
Peter {B{\"u}rgisser}, Felipe {Cucker}, and Josu{\'e} {Tonelli-Cueto},
  \emph{{Computing the Homology of Semialgebraic Sets. II: General formulas}},
  arXiv e-prints (2019), arXiv:1903.10710.

\bibitem{BCT2020.1}
Peter B\"{u}rgisser, Felipe Cucker, and Josu\'{e} Tonelli-Cueto,
  \emph{Computing the homology of semialgebraic sets. {I}: {L}ax formulas},
  Found. Comput. Math. \textbf{20} (2020), no.~1, 71--118. \MR{4056926}

\bibitem{Canny93a}
J.~Canny, \emph{Computing road maps in general semi-algebraic sets}, The
  Computer Journal \textbf{36} (1993), 504--514.

\bibitem{Ferry2018}
Steve Ferry, \emph{A {V}ietoris-{B}egle theorem for connective {S}teenrod
  homology theories and cell-like maps between metric compacta}, Topology Appl.
  \textbf{239} (2018), 123--125. \MR{3777328}

\bibitem{GaV}
A.~Gabrielov and N.~Vorobjov, \emph{Betti numbers of semialgebraic sets defined
  by quantifier-free formulae}, Discrete Comput. Geom. \textbf{33} (2005),
  no.~3, 395--401. \MR{2121987 (2005i:14075)}

\bibitem{GV92}
D.~Grigoriev and N.~Vorobjov, \emph{Counting connected components of a
  semi-algebraic set in subexponential time}, Comput. Complexity \textbf{2}
  (1992), no.~2, 133--186.

\bibitem{Iversen}
B.~Iversen, \emph{Cohomology of sheaves}, Universitext, Springer-Verlag,
  Berlin, 1986. \MR{842190 (87m:14013)}

\bibitem{Markov2}
A.~Markov, \emph{The insolubility of the problem of homeomorphy}, Dokl. Akad.
  Nauk SSSR \textbf{121} (1958), 218--220. \MR{0097793}

\bibitem{Novikov}
P.~S. Novikov, \emph{On the algorithmic insolvability of the word problem in
  group theory}, American {M}athematical {S}ociety {T}ranslations, {S}er 2,
  {V}ol. 9, American Mathematical Society, Providence, R. I., 1958, pp.~1--122.
  \MR{0092784}

\bibitem{Smale}
S.~Smale, \emph{A {V}ietoris mapping theorem for homotopy}, Proc. Amer. Math.
  Soc. \textbf{8} (1957), 604--610. \MR{0087106 (19,302f)}

\bibitem{Spanier}
E.~H. Spanier, \emph{Algebraic topology}, McGraw-Hill Book Co., New York, 1966.
  \MR{0210112 (35 \#1007)}

\bibitem{Sturmfels-Ziegler}
Bernd Sturmfels and G\"{u}nter~M. Ziegler, \emph{Extension spaces of oriented
  matroids}, Discrete Comput. Geom. \textbf{10} (1993), no.~1, 23--45.
  \MR{1215321}

\bibitem{Dries}
L.~van~den Dries, \emph{Tame topology and o-minimal structures}, London
  Mathematical Society Lecture Note Series, vol. 248, Cambridge University
  Press, Cambridge, 1998. \MR{1633348 (99j:03001)}

\bibitem{Cadek}
Martin \v{C}adek, Marek Kr\v{c}\'{a}l, Ji\v{r}\'{\i} Matou\v{s}ek,
  Luk\'{a}\v{s} Vok\v{r}\'{\i}nek, and Uli Wagner, \emph{Polynomial-time
  computation of homotopy groups and {P}ostnikov systems in fixed dimension},
  SIAM J. Comput. \textbf{43} (2014), no.~5, 1728--1780. \MR{3268623}

\bibitem{Viro-Fuchs}
O.~Ya. Viro and D.~B. Fuchs, \emph{Homology and cohomology}, Topology. {II},
  Encyclopaedia Math. Sci., vol.~24, Springer, Berlin, 2004, Translated from
  the Russian by C. J. Shaddock, pp.~95--196. \MR{2054457}

\bibitem{Wachs}
Michelle~L. Wachs, \emph{Poset topology: tools and applications}, Geometric
  combinatorics, IAS/Park City Math. Ser., vol.~13, Amer. Math. Soc.,
  Providence, RI, 2007, pp.~497--615. \MR{2383132}

\end{thebibliography}

\appendix
\section{Proof of Lemma~\ref{16:lem:star}}
\label{sec:AppendixA}

\begin{proof}[Proof of Lemma~\ref{16:lem:star}]
We will denote for $0 \leq i \le s$
\begin{eqnarray*}
\R_i' &=& \R\la\mu_s,\nu_s, \ldots, \mu_i\ra, \\
\R_i &=& \R\la\mu_s,\nu_s, \ldots, \mu_i,\nu_i\ra.
\end{eqnarray*}

Note that 
\[
\R' = \R_0 \supset \R_0' \supset  \cdots \R_s \supset \R_s' \supset \R. 
\]

For $0 \leq i \leq s$ we define inductively:
\begin{eqnarray*}
S_0 &=& S',
\end{eqnarray*}
and for $i > 0$,
\begin{eqnarray*}
S_i^- &=& \lim_{\nu_i} S_i,\\
S_{i+1} &=& \lim_{\mu_i} S_i \;\;(= \lim_{\mu_i} S_i^-).
\end{eqnarray*}

The lemma will follow from the following two claims.

\begin{claim}
\label{proof:prop:claim:1}
For each $i, 0 \leq i \leq s$,
\[
S_i \searrow S_i^-.
\]
\end{claim}

\begin{proof}
Easy.
\end{proof}

\begin{claim}
\label{proof:prop:claim:2}
For each $i, 0 \leq i \leq s$,
\[
S_i^- \searrow S_{i+1}.
\]
\end{claim}
\begin{proof}
We will prove that 
\[
\E(S_{i+1}, \R_i') = S_i^-,
\]
which suffices to prove the claim.

It is obvious that
\[
\E(S_{i+1}, \R_i') \supset S_i^-.
\]

We now show that 
\[
\E(S_{i+1}, \R_i') \subset S_i^-.
\]
Define, 
\begin{eqnarray*}
S_{i+1}^{(<i)} &=& \bigcup_{\sigma \in \Sigma_{\phi}, \level(\sigma) < i}\lim_{\mu_i} \RR(\overline{\sigma}, \E(B,\R_{i+1})),\\
S_{i+1}^{(=i)} &=& \bigcup_{\sigma \in \Sigma_{\phi}, \level(\sigma) = i}\lim_{\mu_i} \RR(\overline{\sigma}, \E(B,\R_{i+1})),\\
S_{i+1}^{(>i)} &=& \bigcup_{\sigma \in \Sigma_{\phi}, \level(\sigma) > i}\lim_{\mu_i} \RR(\overline{\sigma}, \E(B,\R_{i+1})).
\end{eqnarray*}

It is easy to see that,
$$
\displaylines{
\E(S_{i+1}^{(<i)}, \R_i') \subset S_i^-, \cr
\E(S_{i+1}^{(>i)}, \R_i') \subset S_i^-.
}
$$
It remains to prove that 
\[
\E(S_{i+1}^{(=i)}, \R_i') \subset S_i^-.
\]
Let $\sigma \in \Sigma_{\phi}$, $\level(\sigma) = i$, and 
$x_0 \in \lim_{\mu_i} \RR(\overline{\sigma}, \E(B,\R_{i+1}))$.

Let $\mathcal{P}_0 = \{P \in \mathcal{P} \mid \lim_{\mu_i} P(x_0) = 0\}$.
If $\card(\mathcal{P}_0) = i$, then $x_0 \in S_i^-$ and we are done.

Otherwise, $\sigma_0 = \mathrm{sign}(\mathcal{P}(x_0)) \in \Sigma_{\phi}$ (using the 
fact that $S$ is closed). 
Let $x_1 = \lim_{\mu_{\level(\sigma_0)}} x_0$, 
$\sigma_1 = \mathrm{sign}(\mathcal{P}(x_1))$. 
If $\sigma_1 \neq  \sigma_0$, then define 
$x_2 = \lim_{\mu_{\level(\sigma_1)}} x_1$. Continue in this way and 
define $x_0,x_1,x_2, \ldots$, till 
$\sigma_j = \sigma_{j+1}$.  
Notice that $\sigma_0, \ldots,\sigma_j \in  \Sigma_{\phi}$.
Consider the point $x_j$.
Then, $x_j = \lim_{\mu_{\level(\sigma_{j-1})}} x_0$, and 
\[
x_0 \in  \RR(\overline{\sigma_j}, \E(B,\R_{i}')) \subset S_i^-,
\]
since $\sigma_j \in \Sigma_{\phi}$.
This ends the proof of the claim.
\end{proof}

It follows from Claims~\ref{proof:prop:claim:1} and 
\ref{proof:prop:claim:2} that $S' \searrow \lim_{\mu_{s}}  S' = S_{s+1}$.
Now it is obvious from the definition of $\bar\sigma$, that for
each $\sigma \in \Sigma_{\phi}$,
\[
\RR(\bar\sigma, \E(B,R')) \cap \R^k = \RR(\sigma,B).
\]

It follows that 
\[
S' \cap \R^k = S.
\]

Finally, since $S' \searrow S_{s+1} \subset \R^k$, it follows that
$S' \cap \R^{k} = S_{s+1}$, and hence $S_{s+1} = S$.
This implies that $S' \searrow S$.

Finally, it follows from Lemma~\ref{lem:monotone} that
$\E(S,\R')$ is a semi-algebraic deformation retract of $S'$.
\end{proof}

\end{document}